\documentclass{amsart}
\usepackage{cite}
\usepackage{amsmath}
\usepackage{amsthm}
\usepackage{amssymb}
\usepackage{MnSymbol}
\usepackage{algorithmicx}
\usepackage{algpseudocode}
\usepackage{tikz,tikz-cd}
\usetikzlibrary{arrows,decorations.markings, arrows.meta}
\usepackage[foot]{amsaddr}

\newtheorem{thm}{Theorem}
\newtheorem{lem}[thm]{Lemma}
\newtheorem{prop}[thm]{Proposition}
\newtheorem{cor}[thm]{Corollary}
\theoremstyle{definition}
\newtheorem{defn}[thm]{Definition}
\newtheorem{ex}{Example}
\newtheorem{exs}[ex]{Examples}
\theoremstyle{remark}
\newtheorem{rem}[thm]{Remark}
\newtheorem{rems}[thm]{Remarks}
\newtheorem{prob}{Problem}
\newtheorem{asm}[ex]{Assumption}

\newtheorem{alg}[thm]{Algorithm}

\newcommand{\CC}{\mathbb C}
\newcommand{\RR}{\mathbb R}

\newcommand{\ZZ}{\mathbb Z}
\newcommand{\FF}{\mathbb F}
\newcommand{\PP}{\mathbb P}
\newcommand{\LL}{\mathbb L}
\newcommand{\e}{\mathbf e}
\newcommand{\mh}{\mathbf h}
\newcommand{\tensor}{\otimes}
\newcommand{\mat}[1]{\begin{bmatrix} #1 \end{bmatrix}}
\newcommand{\be}{\begin{equation}}
\newcommand{\ee}{\end{equation}}
\DeclareMathOperator{\supp}{\mathrm{supp}}

\DeclareMathOperator{\Aut}{\mathrm{Aut}}
\DeclareMathOperator{\PermAut}{\mathrm{PermAut}}
\DeclareMathOperator{\sign}{\mathrm{sign}}
\DeclareMathOperator{\diag}{\mathrm{diag}}
\DeclareMathOperator{\spn}{\mathrm{span}}

\tikzcdset{every label/.append style = {font = \large}}

\tikzcdset{arrow style=tikz, diagrams={>={Computer Modern Rightarrow[width=6pt,length=4pt]}}}

\author{Assaf Goldberger$^{\ast}$ and Giora Dula}
\address{$^{\ast}$Corresponding author}
\date{\today}
\title[Cohomology-Developed Matrices]{Cohomology-Developed Matrices - constructing families of weighing matrices and automorphism actions}

\numberwithin{equation}{section}
\numberwithin{thm}{section}
\numberwithin{ex}{section}
\numberwithin{prob}{section}
\begin{document}

	\begin{abstract}
		The aim of this work is to construct families of weighing matrices via their automorphism group action. The matrices can be reconstructed from the $0,1,2$-cohomology groups of the underlying automorphism group. We use this mechanism to (re)construct the matrices out of abstract group datum. 
		As a consequence, some old and new families of weighing matrices are constructed. These include the Paley Conference, the Projective-Space, the Grassmannian, and the Flag-Variety weighing matrices. We develop a general theory relying on low dimensional group-cohomology for constructing automorphism group actions, and in turn obtain structured matrices that we call \emph{Cohomology-Developed matrices}. This ``Cohomology-Development" generalizes the Cocyclic and Group Developments. The Algebraic structure of modules of Cohomology-Developed matrices is discussed, and an orthogonality result is deduced. We also use this algebraic structure to define the notion of \emph{Quasiproducts}, which is a generalization of the Kronecker-product. 
		
	\end{abstract}

	\maketitle

	\section{Introduction}
	
	Let $w\le N$ be positive integers. A \emph{weighing matrix} of size $N$ and weight $w$ is an $N\times N$-matrix $A$ with entries from $\{0,-1,1\}$ such that $AA^T=wI$. More generally, let $\mu_n$ be the group of complex $n$th roots of unity. A \emph{generalized weighing matrix} of order $N$ and weight $w$ is a square $N\times N$ matrix with entries from $\mu_n^+=\{0\}\cup \mu_n$ such that $WW^*=wI_N$ (* stands for the conjugate-transpose). We denote the collection of all these matrices by $GW(N,w;n)$. We also say that $W$ is a $GW(N,w;n)$. In the case where $\mu_n=\mu_2=\{\pm 1\}$, we are reduced to weighing matrices. We denote the collection of these by $W(N,w)=GW(N,w;2)$. When $N=w$, such matrices are called \emph{Hadamard matrices}. If $n>2$ and $N=w$, such matrices are known in the literature as \emph{Butson} Hadamard matrices. The main question in this area is the existence of a $GW(N,w;n)$. For more information on weighing matrices, see for example \cite{Seb17}.\\
	
	In the search for weighing matrices, people have been looking for matrices with specific structure, which makes the search space considerably smaller. An example for this is the notion of \emph{group-developed matrices} (see e.g. \cite{CraiDeL16,Arasu13}). These are matrices $A$ indexed by some finite group $G$ of size $N$, such that $A_{g,g'}=f(g^{-1}g')$  for some function $f:G\to \mu_n^+$ (In some places in the literature group developed matrices are defined with a different row order). The point is that $AA^*$ is again of the same type (with respect to the ordering convention given here), so that there are only $\lceil N/2\rceil$ orthogonality constraints, and it is conceivable that there will exist a choice of $f$ which will make the matrix orthogonal. In addition, group-developed matrices may be sometimes substituted for the symbols in orthogonal designs, to obtain a Hadamard matrix (see e.g. \cite{koukouvinos1999new}), with or without adding some narrow margins \cite{FlGySeb99}. An important special case abundant in the literature is the case of the circulant matrices,  i.e. $G$ being cyclic (see the Strassler Table \cite{Strass97,Tan2018}).\\
	
	The problem with group-developed $GW(N,w;n)$, is that the weight $w$ must be a perfect square, at least for $n=2$ (see \cite[p. 206]{Arasu13}). An important modification of group-development which liberates us from this constraint, is the notion of \emph{Cocyclic-Development}. This is done by modifying a group-developed matrix, by the entries of a $\mu_n$-valued 2-cocycle on the underlying group \cite{HoDel93}. This construction has originated from multidimensional combinatorial designs. Later it was shown that some Hadamard matrices can be constructed in this way \cite{coc11} and this was extended further to weighing matrices \cite{deLauney:2000}. Much work has been put in the last 30 years on Cocyclic matrices, such as  \cite{BarreraAcevedo2019,lvarez2019,Horadam1993,deLauney2000,Egan2017,Flannery1997}, to give a partial list. For a comprehensive account of cocyclic constructions and their origins, see \cite{DelFla11}.\\
	
	Given a generalized weighing matrix $A$, there are few operations on $A$ that change the matrix, while keeping the resulting matrix orthogonal. We may permute the rows, multiply a row by an element of $\mu_n$ (a sign), and similarly for columns. All these operations and combinations thereof are called \emph{Hadamard operations}, and together form a group. The subgroup of all Hadamard operations that keep $A$ unchanged, is the \emph{Automorphism} group of $A$ (look at Definition \ref{def:aut} below for strict definitions).\\
	
	There is a close connection between cocyclic matrices and automorphism groups. The group $G$ is naturally an automorphism subgroup of any $G$ group-developed matrix \cite[Theorem 2]{coc11}, and a specific central cover $\widetilde{G}$ of $G$ is a subgroup of the automorphism group of a $G$-cocyclic matrix \cite[Theorem 3]{coc11}. Many times though, the automorphism groups of such matrices are larger.\\
	
	In this work, we study the problem of lifting automorphism (sub)groups from $\{0,1\}$-matrices to $\mu_n^+$-matrices. Namely, given a $\{0,1\}$ rectangular matrix $Z$, together with an automorphism subgroup $G$ (which consists only of permutations on the rows and columns), we wish to study $\mu_n^+$-matrices $A$, satisfying $|A|=Z$ (the absolute value is taken componentwise), such that action of $G$ on $Z$ lifts to a monomial action of $G$ on $A$. Cocyclic-developments are the solutions to this automorphism lifting problem (ALP in short), if we begin with a $G$-developed matrix $Z$.\\
	
	In this paper we solve the ALP under the mild restriction (to be removed in future work) of irreducibility (see section 2). In this setting, it turns out that solutions to the ALP are classified by cohomology classes in the cohomology groups $H^i(G,-)$, for $i=0,1,2$, with respect to suitable $G$-modules. We call solutions to the ALP by the name \emph{Cohomology-Developed matrices}. Unlike the tradional usage of cocycles in cocyclic development, here we use cohomology classes in the construction. If one chooses to work with different representing cocycles in the same cohomology classes, the resulting space of solutions will be the same up to diagonal equivalence. This difference is inconsequential if we aim at searching for Hadamard matrices. For example in the case of cocyclic development, when we modify the 2-cocycle in the same class, we obtain a new matrix obtained from the old one by changing the values in the first row, in addition to multiplying rows and columns by signs. So instead of enumerating on 2-cocycles, as is customary, we can select a representative 2-cocycle per each cohomology class, and enumerate on the space of vectors in the first row. It should also be noted that the choice of values in the first row in manifested in the choice of the cohomology class in the cohomology $H^0(G,-)$ mentioned above.\\
	
	Our work has close connection to the work of D.G. Higman on Weighted Association Schemes and Coherent configurations \cite{Higman1987,Sankey2014}. Cohomology-Developed matrices correspond 
	to the `group case' of Higman. In this language, $H^1$-Developed matrices correspond to morphisms between monomial representations, and $H^2$-developed matrices correspond to morphisms between projective monomial representations. The contribution of our work is that such morphisms can be computed by methods of homological algebra, notably by the Eckmann-Shapiro Lemma, and diagram chases along exact sequences. Also, the algebraic structure of Cohomology-Development is clearer under the homological picture (see Theorem \ref{nuniq} for instance). This cohomology interpretation can also be generalized to broader mathematical contexts, beyond representation theory, such as spectral sequences, higher cohomology-developed tensors, algebraic number theory, magic squares and more. These will be discussed in a sequel paper.\\
	
	As a consequence of Cohomology Development, some new examples of (families) of weighing matrices are constructed, along with old ones. We show how Paley's Conference matrices can be constructed and understood from the viewpoint of our theory. The well-known family of Projective-Space weighing matrices is constructed as well. The interesting point is that the orthogonality of such matrices can be concluded from a theoretical principle (see Theorem \ref{orth}), without need for computations. We derive the existence of the new families of Grassmannian and Flag-Varieties weighing matrices, which are quite complicated objects. Nevertheless, their orthogonality follows from the same theoretical principle. Probably many more families can be constructed along these lines. To give a small example, we construct from the theory, without computations, a symmetric weighing matrix, $W(15,9)$, with automorphism subgroup isomorphic to $A_6$. In all cases and proofs, the considerations we make are purely group-theoretical, with analysis of subgroups, homomorphisms and Cohomology classes. We do not need to construct the objects explicitly. However, explicit construction is possible, and an algorithm for constructing the matrices is given (Algorithm \ref{alg}).\\
	
	One further construction we define is the notion of \emph{quasiproducts}. They are a `twisted' version of the Kronecker product of matrices, and are defined under a certain group-theoretical situation. We show that in general, they are not equivalent to Kronecker products, nor even to R. Craigens's weaving products \cite{CR1995}. The measure of `twistedness' is a 2 - Cohomology class of an underlying group. As an example, we construct the families of \emph{quasiprojective, quasigrassmannain} and \emph{quasiflag} weighing matrices.\\ 
	
	The organization of the paper is as follows:
	Section 2 deals with the basic notions of automorphism groups. We define the fundamental notions of irreducibility and orientability. We will formulate the \emph{automorphism lifting problem} (ALP) and define Cohomology-Development, as a solution to the ALP, but without referring to cohomology.\\
	
	 In section 3 we relate this notion to cohomology.
	 This section is devoted to the case of $H^1$-development, which is a solution to the ALP in an important subfamily of cases. We give a practical algorithm for constructing $H^1$-developed matrices. Later we discuss the relationship between the Hadamard multiplication and the addition operation in the cohomology. Section 4 constructs the Paley's Conference and Hadamard matrices, as a consequence of $H^1$-development.\\
	 
	 Section 5 discusses the algebraic structure of Cohomology-Development. There are  algebras, modules and relations between those by the Hadamard products. A fundamental orthogonality result is developed there. We end this section with a discussion of Quasiproducts.\\
	 
	 We proceed in Section 6 to construct the families of (quasi-) projectives-space, grassmanian and flag-variety weighing matrices. In Section 7 we complete the discussion of Section 3, for $H^2$-development. Such matrices can be constructed in two ways: First by $H^1$-development with respect to some extension group. Second and more classically, by 2 - Cohomology classes. We show the equivalence of these constructions. We then show that cocyclic matrices are a special case. By modifying slightly the definition of cocyclic matrices, we arrive at a situation that cocyclic matrices under a given cohomology class form a matrix algebra.\\

	 We end this paper with Appendices A and B, which discuss in brief the necessary background on group cohomology.

	\section{Automorphism Groups}
	\subsection{Notation}
	Throughout the paper we set $\tau=\exp(2\pi i/n)$ the fundamental $n$-th root of unity,  $\mu_n=\langle \tau\rangle \subset \CC^\times$ is the group of complex $n$th roots of unity. Also denote $\mu_n^+=\{0\}\cup \mu_n$. For a complex number $z$ we denote $z^*$ for its complex conjugate. For a complex matrix $A$, let $A^*$ denote its conjugate-transpose.  Let $|A|$ be the matrix obtained from $A$ by taking componentwise moduli. Let $I_r$ be the $r\times r$ identity matrix, and $J_r=(1)$ be the $r\times r$ matrix with constant entries $1$. Sometimes we will write $I$ and $J$ when $r$ is clear from the context. Let $\FF_q$ denote the finite field of $q$ elements. For any positive integer $p$ let $S_p$ be the symmetric group of permutations on $\{1,2,\ldots,p\}$. Let $\mathbf e_i$ denote the standard row vector (of implicit length) whose $i$th-entry is $1$, and all remaining entries are $0$. Let $\mathbf j_r=\mathbf j$ denote the vector of order $r$ with constant entries $1$. 
	The Hadamard multiplication of two matrices $A,B$ of the same size is the matrix $A\circ B$ such that for all $i,j$, $(A\circ B)_{i,j}=A_{i,j}B_{i,j}$. The Hadamard power $A^{\circ m}$, $m\in \ZZ$ is defined as the matrix $B$ with $B_{i,j}=A_{i,j}^m$ if $A_{i,j}\neq 0$, and $B_{i,j}=0$ otherwise. Sometimes we will use the binary operation $\circ$ to denote composition of functions, and there should be no confusion with the Hadamard multiplication. For a complex vector $v$, let $\diag(v)$ be the diagonal matrix $D$ with $D_{i,i}=v_i$.\\

	In this work we will be interested in the following sets of matrices:
	\begin{align*}
		M\mu_n(m,p)\ &= \ \text{The set of all $m\times p$ matrices over }\mu_n^+,\\
		M\mu_n(m)\ &= \ M\mu_n(m,m).
	\end{align*}
	
	Let $Mon(p,\mu_n)$ be the group of all \emph{monomial} $p\times p$ matrices with values in $\mu_n^+$. A monomial matrix is a square matrix with a unique nonzero element in each row and column. The group $S_p$ can be viewed  as the subgroup of permutation matrices in $Mon(p,\mu_n)$. There is the split exact sequence
	
	\begin{equation}
	1 \to \mu_n^p \to Mon(p,\mu_n) \to S_p \to 1,
	\end{equation}
	where $\mu_n^p$ embeds as the diagonal subgroup of $Mon(p,\mu_n)$, and a matrix $P\in Mon(p,\mu_n)$ maps to $|P|$ as a permutation matrix. For every monomial matrix $M\in Mon(p,\mu_n)$, we can write uniquely $M=SP$ where $S$ is a diagonal matrix with diagonal entries in $\mu_n$, and $P$ is a permutation matrix.\\

	There is a natural left action of the group $Mon(m,\mu_n)\times Mon(p,\mu_n)$ on $M\mu_n(m,p)$, given by
	$$(P,Q):\ W \mapsto PWQ^*.$$ Note that the subgroup $Triv:=\langle (\tau I_m,\tau I_p)\rangle$ acts trivially on $M\mu_n(m,p)$, hence the action descends to $(Mon(m,\mu_n)\times Mon(p,\mu_n))/Triv$. For any $A\in M\mu_n(m,p)$ and $g\in (Mon(m,\mu_n)\times Mon(p,\mu_n))/Triv$ , write $gA$ for the action of $g$ on $A$.
	
	\begin{defn}\label{def:aut}\begin{itemize}
			\item[]
			\item[(i)] Two arrays $A,B\in M\mu_n(m,p)$ are said to be \emph{H-equivalent} (H for Hadamard) if $A=LBR^*$ for $L\in Mon(m,\mu_n)$ and $R\in Mon(p,\mu_n)$.
			\item[(ii)] For  $A \in M\mu_n(m,p)$, an \emph{automorphism} of $A$ is a pair $(P,Q)\in  Mon(m,\mu_n)\times Mon(p,\mu_n)$ that leaves $A$ unchanged. 
			\item[(iii)] The set of all automorphisms of $A$ is a group, called the \emph{automorphism group} of $A$ and denoted by $\Aut(A)$. The \emph{reduced automorphism group} is the image $\overline{\Aut(A)}$ of $\Aut(A)$ in $Mon(m,\mu_n)\times Mon(p,\mu_n)/Triv$.
		\end{itemize}
		
	\end{defn}

	The group $S_m\times S_p$ acts on $m\times p$ matrices, and we let $\PermAut(A)$ be the subgroup leaving $A$ invariant. There are natural group homomorphisms
	\begin{equation}
	\Aut(A)\to \overline{\Aut(A)}\to \PermAut(|A|).
	\end{equation}
	
	\begin{ex}\label{examp1}
		Let $A=\mat{1 & -1\\ i & -i}\in M\mu_4(2,2)$. The pairs $P_1=\left(\mat{0 & -1\\ 1 & 0}, \mat{-i & 0 \\ 0 & -i}    \right)$, $P_2=\left(\mat{1 & 0 \\ 0 & 1}, \mat{0 & -1\\ -1 & 0}\right)$, and $P_3=\left( iI,iI\right)$ are in $\Aut(A)$, and in fact generate this group. $P_3$ generates $Triv$. In this example $\Aut(A)$ is isomorphic to the abelian group $\ZZ/4\oplus\ZZ/2\oplus\ZZ/2$, where $P_3$ corresponds to $(1,0,0)$, $P_1$ corresponds to $(1,1,0)$ and $P_2$ corresponds to $(0,0,1)$. $\overline{\Aut(A)}\cong \PermAut(A)\cong \ZZ/2\oplus\ZZ/2$.
	\end{ex}
	
	The following lemma is clear.
	\begin{lem}
		If $A$ is a nonsingular matrix, then both projections
		$\pi_1:\Aut(A) \rightarrow Mon(m,\mu_n)$ and 
		$\pi_2:\Aut(A) \rightarrow Mon(p,\mu_n)$
		   given by $(P,Q)\mapsto P$ and $(P,Q)\mapsto Q$ are injective.
	\end{lem}
    Note that $A$ in Example \ref{examp1} is singular and the pair $P_2$ is in the kernel of $\pi_1$.

    \subsection{Automorphisms and matrices from $G$-sets}
    
    In the next step we want to reformulate the notion of automorphisms and automorphism groups in the language of $G$-sets. This will give us the advantage of studying the group-theoretic infrastructure behind automorphisms, without looking at a specific matrix. As a consequence we will be able to (re)construct the matrix from this automorphism data. This is in fact excatly the philosophy of our paper.\\

	Let $G$ be a finite group. By a $G$-set we mean a set $X$ with a left $G$-action. Througout the paper let $X,Y$ be finite $G$-sets of cardinalities $m$ and $p$ respectively. A $\mu_n^+$ valued $X\times Y$ matrix is a rectangular $m\times p$ matrix, whose positions are indexed by $X$ and $Y$. In more precise terms, such a matrix is a function $X\times Y\to \mu_n^+$.  We shall denote the set of all $X\times Y$-matrices by $M\mu_n(X,Y)$. This set carries a natural left $G$-action given on a matrix $A=(A_{x,y})$ by $g: (A_{x,y})\mapsto (A_{g^{-1}x,g^{-1}y})$. As is customary, we shall denote this new matrix by $gA$.\\
    
    Note that $X$ and $Y$ may be non-isomorphic $G$-sets, and may even not be of the same cardinality. This in particular means that the $G$-action on the two axes of the matrix is principally different. For example, this is the case of projective-space matrices discussed below. There are useful constructions with rectangular $X\times Y$-matrices, such as \emph{Formal Orthogonal Pairs} \cite{GK2020}. We will not be discusse them here.\\

    It is useful to visualize the orbits of the $G$-action on $X\times Y$ pictorially. For example, take $G=\ZZ/4\ZZ$, a cyclic group, and let the generator $1$ of $G$ act on $X=\{1,2\}$  as the permutation $(1,2)$, and on $Y=\{1,2,3,4\}$ as the permutation $(1,2,3,4)$. Then the orbits on $X\times Y$ can be visualized  as
    $$\begin{bmatrix}
      * & + & * & +\\
      + & * & + & *
    \end{bmatrix}.$$
    
    Inside $M\mu_n(X,Y)$ there is contained the subset of $G$-invariant matrices, i.e. matrices that satisfy $gA=A$ for all $g\in G$. The following easy lemma is easy:
    \begin{lem}
    	$A\in M\mu_n(X,Y)$ is $G$-invariant if and only if it has constant value along each orbit. 
    \end{lem}
	This is a standard result in permutation groups (see e.g. \cite{Wielandt1964-hk}). Note that $G$ acts by permutations on $X,Y$, however, $G$ is not a permutation group since the action might have kernels.
    \begin{ex}
        Suppose that $X=Y=G$ as left $G$-sets. Then $A\in M\mu_n(G,G)$ is $G$-invariant if and only if it is $G$-developed. Recall the a matrix $A$ is $G$-developed if $A_{x,y}=f(x^{-1}y)$ for some function $f$. Circulant matrices are the special case for $G=\ZZ/m\ZZ$.
    \end{ex}

    \begin{rem}
    	In some places in the literature, \emph{Group-Development} is defined by the equation $A_{x,y}=f(xy)$ for all $x,y\in G$, and for some function $f:G\to \mu_n$. To accommodate this definition, one should take $X=G$, but with the \emph{modified} left $G$-action given by $(g,x)\to xg^{-1}$.
    \end{rem}

	In all $G$-invariant matrices, the group $G$ has a natural map into the automorphism group of $A$ given by $g\mapsto (L(g),R(g))$ where $L(g)$ (resp. $R(g)$) is the permutation matrix corresponding to the action of $g$. All matrices $A\in M\mu_n(X,Y)$ invariant under $G$ can be obtained by computing the $G$-orbits in $X\times Y$ and then assigning to each orbit a constant value in $\mu_n^+$. A matrix $A\in M\mu_n(X,Y)$ is said to be \emph{above} $A_0\in M\mu_1(X\times Y)$, if $A_0=|A|$.\\
	
	\subsection{Irreducibility}\label{2.3}
	We now define an important notion of an irreducible matrix. This is a technical condition that will be needed in the sequel in order to avoid some complications in the theory. In a future work we will get rid of this condition. For the current discussion (\S\ref{2.3}), the group $G$ is not important, and the sets $X,Y$ can be considered as abstract sets.\\
	
	The kernel $$\Delta(A):=\ker\left( \Aut(A) \to \PermAut(|A|)  \right)$$ contains all pairs $(L,R)\in \Aut(A)$ for $L$ and $R$ diagonal. Obviously, $\Delta(A)$ contains the subgroup $Triv=\langle (\tau I_M,\tau I_N)\rangle$. It may sometimes happen, though, that $\Delta(A)$ will be larger than $Triv$. 
	
	\begin{defn}
		Let $A$ be in $M\mu_n(X,Y)$. The \emph{support} of $A$ is the set $\supp(A)$ of all  $(x,y)\in X\times Y$ such that $A_{x,y}\neq 0$.
	\end{defn}
	
	\begin{defn}
		A nonzero matrix $A\in M\mu_n(X,Y)$ is \emph{reducible} if there are nontrivial partitions $X=X_1\cup X_2$ and $Y=Y_1\cup Y_2$, $X_1\cap X_2=\varnothing$, $Y_1\cap Y_2=\varnothing$, such that $\supp(A)$ is contained in $X_1\times Y_1 \sqcup X_2\times Y_2$. Otherwise we say that $A$ is \emph{irreducible}.
	\end{defn}

    \begin{ex}\label{examp2}
    	Let $A=\mat{1 & -1 & 0\\ 0 & 0 & 1}\in M\mu_2(2,3)$. Then $A$ is reducible, for we can take the partitions $X_1=\{1\}$, $X_2=\{2\}$ and $Y_1=\{1,2\}$, $Y_2=\{3\}$, and $\supp(A)=\{(1,1),(1,2),(2,3)\}\subset X_1\times Y_1 \sqcup X_2\times Y_2$.
    \end{ex}
    \begin{ex}
    	If $A\in M\mu_n(X,Y)$ is a matrix with no zero entries (e.g. a Hadamard matrix), then $A$ is irreducible.
    \end{ex}
    
    Note that irreducibility is a property of $|A|$ or equivalently of $\supp(A)$. We now give two equivalent conditions to irreducibility. Let $A$ be an $X\times Y$ matrix. We construct a bipartite graph $\mathcal G=\mathcal G(A)$ on the vertex set $X\sqcup Y$, with an edge  connecting $x\in X$ and $y\in Y$ if and only if $(x,y)\in\supp(A).$\\ 
	
	The notion of matrix irreducibility appears in matrix theory, for example in Perron-Frobenius theory (see \cite{Horn2012}). However this is not identical to what we discuss here. The difference is that in Perron-Frobenius theory the graph is directed, and here the graph is bipartite. The following lemma has an analog in Perron-Frobenius, and the proof is left to the reader. 
    
    \begin{lem}\label{irred}
     The following conditions on $A$ are equivalent:
    	\begin{itemize}
    		\item[(i)] $A$ is irreducible.
    		\item[(ii)] The bipartite graph $\mathcal G(A)$ is connected.
    		\item[(iii)] $\Delta(A)=Triv$ $\Box$.
    	\end{itemize} 
    \end{lem}

    Note that for $A$ in Example \ref{examp2}, $\Delta(A)$ contains the pair $(\diag(-1,1),\diag(-1,-1,1))$, which is not in $Triv$. Also, the graph $\mathcal G(A)$ can be depicted as
    $$
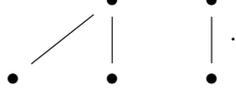
\begin{tikzcd}
    & \bullet \ar[-,dl] \ar[-,d]& \bullet \ar[-,d]\\
    \bullet & \bullet & \bullet
    \end{tikzcd}.$$

	\begin{asm}[*]
		From now, through the end of the paper, we shall assume that $|A|$ is irreducible. In particular we have an exact sequence 
		$$1\to Triv \to \Aut(A) \to \PermAut(|A|).$$ We further assume that $\PermAut(|A|)$ acts transitively
		on the rows and columns of $|A|$.
	\end{asm}

	\subsection{The Automorphism Lifting Problem}\label{lift}
	 Suppose that a finite group $G$ and two finite transitive $G$-sets $X$ and $Y$ are given. Let $\mathcal O\subset X\times Y$ be a $G$-stable subset, which means that $\mathcal O$ is a disjoint union of $G$-orbits. Let $|A|=:A_\mathcal O$ be the characteristic matrix of $\mathcal{O}$: $|A|_{x,y}=1$ if $(x,y)\in \mathcal O$, and $|A|_{x,y}=0$ otherwise. Note that this is the adjacency matrix of $\mathcal G(A_\mathcal O)$.
	 We shall assume that $|A|$ is irreducible. For an ordered set $S$, let $Perm(S)$ denote the group of all permutation matrices indexed by $S$ and $Mon(S,\mu_n)$ the group of all invertible monomial $\mu_n^+$-matrices indexed by $S$. There is a homomorphism $$abs:Mon(S,\mu_n)\to Perm(S)$$ given by the entrywise absolute value $P\mapsto |P|$. Also, if $S$ is a $G$-set, then there is a natural homomorphism $p_S:G\to Perm(S)$.\\
	
	\begin{defn}\label{def:mon_cov}
		Let $G$ be a finite group and $X,Y$ finite $G$-sets. A \emph{monomial cover} of $(G,X,Y)$ with values in $\mu_n$ is a commutative diagram of group homomorphisms	\begin{equation}\label{diag4} \begin{tikzcd}[column sep = huge, row sep=huge]
		\ \ \widetilde{G}\ \  \arrow[dashed,two heads,shorten >= 5pt,shorten <= 5pt]{r}{\rho}  \ar[dashed]{d}{\pi_X\times \pi_Y}
		&\ \ G \ \ \ar{d}{\ p_X\times p_Y}\\
		Mon(X,\mu_n)\times Mon(Y,\mu_n)  \arrow{r}{abs\times abs} & Perm(X)\times Perm(Y)
		\end{tikzcd},
		\end{equation}
		such that $\widetilde{G}$ is a finite group, the map $\rho:\widetilde{G}\to G$ is surjective, and its kernel is mapped by $\pi_X\times \pi_Y$ into $Triv$.
		For brevity we also say that $\widetilde{G}$ is a monomial cover of $G$, or that $\widetilde{G}\to Mon(X,\mu_n)\times Mon(Y,\mu_n)$ is a monomial cover of $G$.
	\end{defn}

    Suppose that $\widetilde{G}$ is a monomial cover of $G$. Let $\pi_X$ and $\pi_Y$ denote the projections from $\widetilde{G}$ to $Mon(X,\mu_n)$ and $Mon(Y,\mu_n)$. We have a left action of $\widetilde{G}$ on $M\mu_n(X,Y)$ given by $ g: A \mapsto g*A:=\pi_X(g)A\pi_Y (g)^*$. We also let $\widetilde{G}$ act on $X$ and $Y$ via the surjection to $G$. Thus we shall write without hesitation $gx$ for $\rho(g)*x$. There will be cases where $\widetilde{G}$ will be equal to $G$. Then it is important to distinguish between the $G$-action (by permutation matrices) and the $\widetilde{G}$ action (by monomial matrices) on a matrix $A$.\\

    This allows us to formulate the \emph{Automorphism Lifting Problem}, which is the main theme of this paper. Informally, we are given a $\{0,1\}$-matrix $A_0$ together with some automorphism subgroup $G\subseteq \PermAut(A_0)$. Then we wish to find a matrix $A$ with $|A|=A_0$, and such that the automorphism subgroup $G$ `lifts' to an automorphism subgroup $\widetilde{G}\subseteq \Aut(A)$. The formal definition of the problem is as follows:
    
    \begin{prob}[The Automorphism Lifting Problem (ALP)]\label{prob1}
    	Suppose that a $\{0,1\}$-$G$-invariant-irreducible-matrix $A_0\in M\mu_1(X,Y)$ is given. Find all matrices $A\in M\mu_n(X,Y)$ satisfying $|A|=A_0$, and a monomial cover $\rho:\widetilde{G}\to G$, such that $g*A=A$ for all $g\in \widetilde{G}$. We then say that $A$ \emph{corresponds} to the monomial cover $\rho:\widetilde{G}\to G$.
    \end{prob}
	
	Notice that in this formulation, we did not require that $G\to Perm(X)\times Perm(Y)$ will be injective, nor did we for $\pi_X\times \pi_Y$. This technical subtlety allows us to use groups which do not act faithfully on the rows or columns or even on the matrix itself. While in most situations we will indeed take this morphism as injection (or even inclusion), there will be other cases (like the projective-space matrices in \S \ref{sec:proj}), where it will be more convenient to use unfaithful actions (cf. remark \ref{notinj} below). We now give a criterion for a matrix $A$ to be the solution of the ALP over $A_0$ without the need to construct explicitly the monomial cover $\widetilde{G}$.
  
    \begin{lem}\label{without cover}
    	Let $X,Y$ be finite $G$-sets, and suppose that $A_0\in M\mu_1(X,Y)$ is irreducible. A matrix $A$ is a solution to the ALP over $A_0$, if and only if for every $g\in G$ there are diagonal matrices $L(g),R(g)$ such that $A=L(g)(gA)R(g)^*$.
    \end{lem}

    \begin{proof}
    	Suppose first that $A$ is a solution to the ALP above $A_0$, together with a monomial cover $\rho:\widetilde{G}\to G$, as in diagram \eqref{diag4}. Then for any $g\in G$, pick up a lift $\tilde g\in \widetilde{G}$ s.t. $\rho(\tilde{g})=g$, and write uniquely $\pi_X(\tilde g)=L(g)|\pi_X(\tilde g)|=L(g)p_X(g)$, and similarly $\pi_Y(\tilde g)=R(g)|\pi_Y(\tilde g)|=R(g)p_Y(g)$ for diagonal matrices $L(g),R(g)$. Then $\tilde g A=A$ is equivalent to $A=L(g)(gA)R(g)^*$. This completes one direction.\\
    	
    	Suppose now that for every $g\in G$ there are diagonal matrices $L(g),R(g)$ s.t. $A=L(g)(gA)R(g)^*$. We need to find a group $\widetilde{G}$, together with a surjective homomorhism $\rho:\widetilde{G}\to G$ and a map $\pi_X\times \pi_Y:\widetilde{G}\to Mon(X,\mu_n)\times Mon(Y,\mu_n)$ such that (i) diagram \eqref{diag4} is commutative, (ii) $\pi_X\times\pi_Y (ker \rho) \subseteq Triv$, and (iii) $\pi_X(\tilde g)A\pi_Y(\tilde g)^*=A$ for all $\tilde g\in \widetilde{G}$. Define
    	\begin{multline}\label{large_ext}
    	\widetilde{G}\ := \ \left\{(P,Q,g) \in Mon(X,\mu_n)\times Mon(Y,\mu_n)\times G \ |\right. \\  \left. |P|=p_X(g), \ |Q|=p_Y(g),\ A=PAQ^* \right\},
    	\end{multline}
    	which is a subgroup of the product $Mon(X,\mu_n)\times Mon(Y,\mu_n)\times G$.
    	The map $\rho:\widetilde{G}\to G$ is given by projecting onto the third coordinate. The maps $\pi_X$ and $\pi_Y$ are the projections to the first two coordinates. The map $\rho$ is surjective, since for every $g\in G$ we may take $\tilde g:=(L(g)p_X(g),R(g)p_Y(g),g)\in \widetilde{G}$ above $g$. Diagram \eqref{diag4} is commutative by the definition of $\widetilde{G}$. Similarly condition (iii) is a consequence of the definition. It remains to prove (ii). If $(P,Q,g)\in ker \rho$, then $g=1_G$, and $|P|$ and $|Q|$ are the identity matrices. Thus $(P,Q)\in \Delta(A)$. But we have assumed that $A$ is irreducible, thus by Lemma \ref{irred} $(P,Q)\in Triv$. This completes the proof.
    	
    \end{proof}

	\begin{rem}
		The monomial cover $\widetilde{G}$ corresponding to the matrix $A$ is not uniquely determined. The proof of Lemma \ref{without cover} generates a specific monomial cover $\widetilde{G}^{univ}$ which is universal in the following sense: For every other monomial cover $\widetilde{G}$ corresponding to $A$, there exists a unique homomorphism $\widetilde{G}\to \widetilde{G}^{univ}$ which commutes with the arrows $\pi_X\times \pi_Y$ and $\rho$ in the corresponding diagrams \eqref{diag4}.
	\end{rem}
	
	In the next few sections, we will understand how to find all monomial covers of $(G,X,Y)$ with the aid of cohomology. 
	Once we obtain a monomial cover, the idea is to (i) break $\mathcal O$ into $G$-orbits, (ii) Fix arbitrary basepoints $(x_O,y_O)\in O$ for every orbit $O\subset \mathcal O$, (iii) Fix arbitrary values $A_{x_O,y_O}\in \mu_n^+$, and (iv) Use the $\widetilde{G}$-action on $A$ to develop uniquely the values $A_{x,y}$ all over $\mathcal O$. But we will soon see that for some orbits $O$ we will be forced set $A_{x,y}=0$ along $O$, as the group action will force conflicting signs. This phenomena will be addressed soon as \emph{non-orientability}.\\
	
	For any $P\in Mon(S,\mu_n)$ and any $s\in S$, denote by $\sign(s,P)\in \mu_n$ the value of the unique nonzero entry of $P$ appearing at the row corresponding to $s$.\\ 
	 
	\begin{defn}\label{orient}
		Suppose that a monomial cover $\widetilde{G}\to G$ is given.  A pair $(x,y) \in X\times Y$ is said to be \emph{orientable} with respect to $\widetilde{G}$, if for every $g\in \widetilde{G}$ with $gx=x$ and $gy=y$, we have $\sign(x,\pi_X(g))=\sign(y,\pi_Y(g))$. Otherwise, we will say that $(x,y)$ is \emph{nonorientable}.
	\end{defn}
	
	Orientability turns out to be a property of an orbit:
		\begin{lem}\label{welldef}
		$(x,y)$ is orientable if and only if all $(x',y')$ in its $G$-orbit are orientable. 
	\end{lem}
	
	\begin{proof}
		Orientability at $(x,y)$ means that $\sign(x,\pi_X(g))=\sign(y,\pi_Y(g))$ for all $g$ stabilizing $(x,y)$. Let $(x',y')=(h^{-1}x,h^{-1}y)$ be in the same orbit. Then we need to show that $\sign(x',\pi_X(g'))=\sign(y',\pi_Y(g'))$ for all $g'$ stabilizing $(x',y')$. We know that $g:=h^{-1}g'h$ stabilizes $(x,y)$. Thus it will suffice to show that $\sign(x',\pi_X(g'))=\sign(x,\pi_X(g))$, and similarly for $y$.\\

		We have $$\pi_X(g)=\pi_X(h^{-1}g'h)=\pi_X(h)^{-1}\pi_X(g)\pi_X(h).$$
		Thus 
		\begin{multline*}
			\sign(x,\pi_X(g))=\sign(x,\pi_X(h)^{-1}\pi_X(g')\pi_X(h))\\
			=\sign(x,\pi_X(h)^{-1})\sign(x',\pi_X(g'))\sign(x',\pi_X(h))=\sign(x',\pi_X(g'))
		\end{multline*}
		as desired.
		
	\end{proof}

	\begin{ex}\label{examp3}
		Consider the group $G=B_3$ - the symmetry group of the 3D-cube. $G\subset O(3,\RR)$ is represented as the group of all monomial $3\times 3$ matrices over $\{-1,1\}$, acting on the Euclidean cube $C=[-1,1]^3\subset \RR^3$ by matrix-vector multiplication. We have $|G|=48$. Let 
		$X$ be the $G$-set of all faces of $C$, or what is equivalent, $$X=\{\pm \e_1,\pm \e_2, \pm \e_3\}.$$ Let $Y$ by the $G$-set of all edges of $C$, modulo antipodicity, that is, we identify two edges if they are antipodal. Equivalently,
		$$Y=\{ \pm \e_i \pm \e_j | i<j \}/\{ \pm 1 \}.$$		
		We have $|X|=|Y|=6$. There are two $G$-orbits in $\mathcal O=X\times Y$; One orbit is for pairs $(x,y)$ such that the edge $\pm y$ lies on the face $x$, and the other is the complement. This can be depicted in the following matrix.
		$$ \left({\begin{array}{ccccccc}

			   o & o& o & o & o2 &  o2\\
			 o & o& o & o & o2 &  o2\\
			 o & o& o2 & o2& o & o \\
			 o & o& o2 &  o2& o & o \\
			 o2 &  o2& o & o& o & o \\
			 o2 &  o2& o & o& o & o  \\
	\end{array}}\right)$$
	In the matrix $X$ is ordered as follows: $\pm \e_1<\pm \e_2<\pm \e_3$, $-\e_i<\e_i$. $Y$ is ordered as follows: We always use the representative $\e_i\pm \e_j$ when $i<j$. We set $\e_i\pm \e_j<\e_k\pm \e_l$ if $i<k$ or $i=k$ and $k<l$. Also $\e_i-\e_j<\e_i+\e_j$.
	\end{ex}
	Here we put the $o$ symbol for the positions of the first orbit, and the $o2$ respectively for the second orbit.\\
	
	There are two group homomorphisms $\chi_1,\chi_2:B_3\to \{\pm 1\}$, given by 
	\begin{align*}
	\chi_1(g)&:= \prod_{i,g_{i,j}\neq 0} g_{i,j} \\
	\chi_2(g)&:=  \det(g).
	\end{align*}
	Let $P(g)$ (resp. $Q(g)$) be the $6\times 6$ permutation matrix of the action of $g$ on $X$ (resp. $Y$). We construct the monomial cover
	$ \pi: \widetilde{G}=G \to Mon(X,\mu_2)\times Mon(Y,\mu_2)$, given by
	\be\label{color1} \pi(g) \ := \ \left(\chi_1(g)P(g),\chi_2(g)Q(g)\right).\ee
	Let us show now that the first orbit is orientable for $\widetilde{G}$, but the second one is not.\\
	
	Take the point with $x=\e_1$ and $y=\e_1+\e_2$, which is in the first orbit. The subgroup $H$ stabilizing $(x,y)$, is the group of all diagonal $3\times 3$-matrices generated by $\diag(1,1,-1)$. Restricted to this group, $\chi_1(g)=\chi_2(g)$ for all $g\in H$, which in turn shows that $\sign(x,\chi_1(g)P(g))=\sign(y,\chi_2(g)Q(g))$.\\
	 
	 On the other hand, the point with $x=\e_1$ and $y=\e_2+\e_3$ lies in the second orbit. The group $K$ stabilizing that point contains the element
	 $$t=\mat{1 & 0 & 0\\ 0 & 0 & 1\\ 0 & 1 & 0}.$$
	 We have $\chi_1(t)=1$ and $\chi_2(t)=-1$. In particular $\sign(x,\chi_1(t)P(t))=1$ and $\sign(y,\chi_2(t)Q(t))=-1$, hence this orbit is not orientable.
	 
	\begin{rem}
		The monomial cover $\pi$ in Equation \eqref{color1} is obtained from the `unsigned' monomial cover $g\mapsto (P(g),Q(g))$ by `twisting' with two characters $\chi_1,\chi_2:G\to \mu_n$. We call such monomial covers \emph{principal}. They are the simplest nontrivial monomial covers. We will see below \S\ref{sec3} that there are non-principal monomial covers.
	\end{rem}

	\subsection{Cohomology-Developed matrices} 
	
	\begin{defn}\label{coh-dev}
		Let $G,X,Y,\mathcal O,\mu_n$ ($\mathcal O$ is irreducible) be given data as above. A $X\times Y$ matrix $A$ over $\mu_n^+$ with support $\mathcal O$ is a \emph{cohomology-developed matrix} (CDM in brief) with respect to this data if is satisfies the condition of Lemma \ref{without cover}, i.e.
		$$\forall  g\in G \ \  \exists  \text{ diagonal }L(g),R(g) \text{ s.t. } A=L(g)(gA)R(g)^*.
		$$
	\end{defn}
	By Lemma \ref{without cover} this is equivalent to $A$ being a solution to the ALP over the support $\mathcal O$. Notice that this definition contains no appeal to cohomology. But cohomology is hidden there, and will appear below, see Theorem \ref{nuniq} and \S\ref{cdm2}.\\
	
	 We denote by $\mathbf h(G,X,Y,\mathcal O,\mu_n)$ the collection of these cohomology-developed matrices. We will also use the shorter notation $\mathbf h(G,\mathcal O)$ when $X,Y,n$ are understood.

	\begin{prop}\label{prop:214}
		$\mathbf h(G,\mathcal O)$ is an abelian group under the Hadamard multiplication of matrices.
	\end{prop}

    \begin{proof}
    	Let $D_X$(resp. $D_Y$) denote the group of diagonal $\mu_n$-matrices indexed by $X$ (resp. $Y$). Suppose that $A,A'\in \mathbf h(G,\mathcal O)$. Then there are functions $L,L':G \to D_X$ and $R,R':G \to D_Y$ such that for all $g\in G$, $A=L(g)(gA)R(g)^*$ and $A'=L'(g)(gA')R'(g)^*$.  We have for every two $X\times Y$ matrices $A_1,A_2$, every $L_1,L_2\in D_X$ and every $ R_1,R_2\in D_Y$ the equalities $g(A_1\circ A_2)=(gA_1 \circ gA_2)$  and $L_1L_2(A_1\circ A_2)(R_1R_2)^*= (L_1A_1R_1^*)\circ (L_2A_2R_2^*)$. Using these, the functions $g\mapsto L(g)L'(g)$ and $g\mapsto R(g)R'(g)$ satisfy that $(L(g)L'(g))(g(A\circ A'))(R(g)R'(g))^*=A\circ A'$, and we have shown closure under Hadamard multiplication. The neutral element is $|A|$ itself, and the inverse of $A$ is given by the matrix $B$ such that $B_{x,y}=A_{x,y}^{-1}$ when $(x,y)\in \mathcal O$ and $B_{x,y}=0$ otherwise (which is $A^{\circ -1}$ in the notation given above).
    \end{proof}
We say that two matrices in $M\mu_n(X,Y)$ are diagonally equivalent (or briefly $D$-equivalent), if $A=LBR^*$ for invertible diagonal matrices $L$ and $R$ over $\mu_n$. We shall write this relation as $A\sim_D B$. It is easy to see that $\mathbf h(G,\mathcal O)$ is closed under $D$-equivalence. We will see now that $\mathbf h(G,\mathcal{O})$ admits a filtration of length $2$ of subgroups which are closed under $D$-equivalence as well. 
    
    \begin{defn}
    	
    	\begin{itemize} \label{cohd}
    		\item[]
    		\item[(a)] A matrix $A\in \mathbf h(G,\mathcal O)$ is said to be $H^0$-developed if $A\sim_D B$ and $B$ is $G$-invariant (that is $gB=B$ for all $g\in G$ under the permutation action). Denote the collection of $H^0$-developed matrices by $\mathbf h^0(G,\mathcal O)$.
    		\item[(b)]  A matrix $A\in \mathbf h(G,\mathcal O)$ is said to be $H^1$-developed if there exists a monomial cover $\rho:\widetilde{G}\to G$, with $\rho$ being an isomorphism, under which $g*A=A$ for all $g\in \widetilde{G}$. Denote the collection of $H^1$-developed matrices by $\mathbf h^1(G,\mathcal O)$.
    		\item[(c)] Any matrix $A\in \mathbf h(G,\mathcal O)$ is said to be $H^2$-developed. We also denote $\mathbf h(G,\mathcal O)=\mathbf h^2(G,\mathcal O).$
        \end{itemize}
    \end{defn}

    \begin{rem}
    	The case (c) of the definition contains matrices for which the monomial cover $\rho:\widetilde{G}\to G$ cannot be chosen to have a section $s:G\to\widetilde{G}$ ($s$ is a section of $\rho$ if $\rho\circ s=id_G$). Otherwise we could redefine $\widetilde{G}=s(G)$ and $\rho$ would be an isomorphism. Such matrices are fundamentally different from $H^1$-developed matrices. We will see later that nontrivial cocyclic matrices belong to this case.
    \end{rem}

    \begin{rem}\label{notinj}
    	The notion of $H^1$-development is relative to the choice of the group $G$. For we may redefine a posteriory $G$ to be equal to $\widetilde{G}$, and our matrix $A$ will turn out to be $H^1$-developed. But we can still define an \emph{absolute} notion of $H^1$-development if we take $G$ to be a subgroup of $\PermAut(|A|)$.  
    \end{rem}

    \begin{ex} Continuing Example \ref{examp3},
    	consider the monomial cover $\widetilde{G}=G=B_3$ together with the homomorphism $\pi:\widetilde{G}\to Mon(X,\mu_2)\times Mon(Y,\mu_2)$ as defined by equation \eqref{color1}. We take $\mathcal O$ as the $G$-orbit containing $(x_0,y_0)=(\mathbf e_1,\e_1+\e_2)$.	
    	Let us construct a $\widetilde{G}$-invariant matrix $A$, by first fixing $A_{x_0,y_0}=1$, and then developing the values $A_{x,y}$ for all $(x,y)\in \mathcal O$ with the aid of the $\widetilde{G}$-action. For example, take the matrix $g=\mat{1 & 0 & 0\\0 & 0 & -1\\0 & 1 & 0}\in \widetilde{G}$. Then $g$ maps $(x_0,y_0)$ to $(x_1,y_1)=(\e_1,\e_1+\e_3)$, $\chi_1(g)=-1$ and $\chi_2(g)=1$. It follows by \eqref{color1} that $A_{x_1,y_1}=-1$. Due to orientability, every other choice of $g$ mapping $(x_0,y_0)$ to $(x_1,y_1)$ will yield the  same result. Repeating this practice over all $(x,y)\in \mathcal O$ will yield the matrix
    	$$A=\mat{1 & 1 & -1 & -1 & 0 & 0\\
    	         1 & 1 & -1 & -1 & 0 & 0\\
                 -1 & -1 & 0 & 0 & 1 & 1\\
                 -1 & -1 & 0 & 0 & 1 & 1\\
                 0 & 0 & 1 & 1 & -1 & -1\\
                 0 & 0 & 1 & 1 & -1 & -1}.$$
        The matrix $A$ is $H^1$-developed, as we have $\widetilde{G}=G$. It can be seen directly that $A$ is not $H^0$-developed, because $rank(A)=2$, while $rank|A|=3$. Alternatively, this is a consequence of Theorem \ref{nuniq}(b) below.
    \end{ex} 

    \begin{ex} \label{Had2}
    	Let $G=\ZZ/2$ act on $X=Y=G$ by the group addition law. Consider $\mathcal O=X\times Y$. Let $\widetilde{G}=\ZZ/4$ with the surjective homomorphism $\rho:\widetilde{G}\to G$, which sends $1\mod 4$ to $1\mod 2$. Let $\pi:\widetilde{G}\to Mon(X,\mu_2)\times Mon(Y,\mu_2)$ be the group homomorphism defined by the rule
    	$$ \pi(1) = \left(\mat{0 & -1\\ 1 & 0},\mat{0 & -1\\ 1 & 0}\right).$$ We obtain a monomial lifting of $(G,X,Y)$. By fixing the top row of our $X\times Y$ matrix to be $(1,1)$, we obtain the $\widetilde{G}$-invariant (Hadamard) matrix
    	$$A=\mat{1 & 1\\-1&1}.$$
    	This matrix is not $H^0$-developed, since all $H^0$ matrices in this case have determinant zero.  Further, it is not $H^1$-developed w.r.t. $G$. For if it were, we would then must have a homomorphism $G\to Mon(X,\mu_2)\times Mon(Y,\mu_2)$, taking the nontrivial element of $G$ to an anti-diagonal automorphism. But the only anti-diagonal automorphism of $A$ is $\pi(1)$ (up to an element of $Triv$), which has order $4$. This is in fact a $\ZZ/2$ cocyclic matrix.
    \end{ex}

   \begin{prop}\label{210}
   	We have a filtration of abelian groups, closed under D-equivalence:
   	$$0 \subseteq \mathbf h^0(G,\mathcal O)\subseteq \mathbf h^1(G,\mathcal O)\subseteq  \mathbf h^2(G,\mathcal O)=\mathbf h(G,\mathcal O)$$
   \end{prop}

   \begin{proof}
   	  $\mathbf h^0(G,\mathcal O)$ is by definition closed under $D$-equivalence. If $A_1=L_1B_1R_1^*$ and $A_2=L_2B_2R_2^*$ for diagonal $L_1,L_2,R_1,R_2$ and $G$-invariant $B_1,B_2$, then $A_1\circ A_2=(L_1L_2)(B_1\circ B_2)(R_1R_2)^*$, and $B_1\circ B_2$ is still $G$-invariant, proving that  $\mathbf h^0(G,\mathcal O)$ is closed under $\circ$. The group inverse of $A=LBR^*$ is $A^*=L^*B^*R$, hence $\mathbf h^0(G,\mathcal O)$ is closed under inverses as well. It remains to show that (i)  $\mathbf h^0(G,\mathcal O)\subseteq  \mathbf h^1(G,\mathcal O)$, (ii) that  $\mathbf h^1(G,\mathcal O)$ is a subgroup, and (iii) that  $\mathbf h^1(G,\mathcal O)$ is closed under $D$-equivalence.\\
   	  
   	  For (i), Suppose that $A=LBR^*$ is $H^0$-developed, $L,R$ diagonal and $gB=B$. To form the monomial cover \eqref{diag4}, take $\widetilde{G}=G$, $\rho=id_G$, and let the maps $\pi_X:\widetilde{G}\to Mon(X,\mu_n)$, and $\pi_Y:\widetilde{G}\to Mon(Y,\mu_n)$ be defined by $\pi_X(g)=LP(g)L^*$ and $\pi_Y(g)=RQ(g)R^*$, where $(P(g),Q(g))$ is taken to be the image of $g=\rho(g)$ in $\PermAut(|A|)$. This makes diagram \eqref{diag4} commutative, and $A$ becomes $\widetilde{G}$-invariant, as required.\\

   	 The proof of (iii) is quite similar to the $D$-equivalence closure of $\mathbf h^0(G,\mathcal O).$ Suppose that $A\in \mathbf h^1(G,\mathcal O)$, and $\rho:\widetilde{G} \to G$ is an isomorphism giving the monomial cover with respect to $A$. If we take $B=LAR^*\sim_D A$, for diagonal $L,R$, then modify $\pi_X(g)$ to $L\pi_X(g)L^*$ and $\pi_Y(g)$ to $R\pi_X(g)R^*$, leaving $\rho$ the same, which now defines a monomial cover giving $B$.\\
   	 
   	 Finally, let us prove (ii). Suppose that $A_1,A_2\in \mathbf h^1(G,\mathcal O)$, and for both, without loss of generality, we choose $\widetilde{G}=G$ and $\rho=id$. The problem is that $\pi_X(g),\pi_Y(g)$ are different for $A_1$ and $A_2$. To ease notation, let us denote the projection images $(\pi_X\times \pi_Y)(g)$ of $g\in \widetilde{G}$ by  $(M_i(g),N_i(g))$, $i=1,2$. Then we can write uniquely $M_i(g)=L_i(g)P_{X}(g)$ and $N_i(g)=R_i(g)P_{Y}(g)$, for permutation matrices $P_{X},P_{Y}$, and diagonal matrices $L_i,R_i$. Thus $A_i=L_i(g)(\rho(g)A_i)R_i(g)^*$ and it follows that $A_1\circ A_2=(L_1(g)L_2(g))(\rho(g)(A_1\circ A_2))(R_1(g)R_2(g))^*.$ It will be sufficient to show that the maps $g\mapsto L_1(g)L_2(g)P_{X}(g)$ and $g\mapsto R_1(g)R_2(g)P_{Y}(g)$ are group homomorphisms, because they will serve as the new map $\widetilde{G}\to Mon(X,\mu_n)\times Mon(Y,\mu_n)$ for $A_1\circ A_2$.\\
   	 
   	 We proceed for $X$. For all $g,h\in G$, using the fact that $A_i$ are $H^1$-developed,
   	 \begin{multline*}
   	 L_1(gh)L_2(gh)P_X(gh)= L_1(gh) L_2(g)P_X(g)L_2(h)P_X(h)\\
   	 =L_2(g)L_1(gh)P_X(g)L_2(h)P_X(h)=L_2(g)L_1(gh)P_X(gh)P_X(h)^{-1}L_2(h)P_X(h)=\\
   	 L_2(g)L_1(g)P_X(g)L_1(h)P_X(h)P_X(h)^{-1}L_2(h)P_X(h)\\
   	 =(L_1(g)L_2(g)P_X(g))\cdot (L_1(h)L_2(h)P_X(h)),
   	 \end{multline*} 
   	 which proves the homomorphism property and finishes the proof of (ii).  
   \end{proof}

   \begin{rem}
   \begin{itemize}
   	   \item[]
   	   \item[(a)] We will see below (Theorem \ref{nuniq}), that the associated graded quotients $\mathbf h^{i}(G,\mathcal O)/\mathbf h^{i-1}(G,\mathcal O)$ are approximated by cohomology groups of $G$, which is the reason for the name `cohomology-developed'.
   	   \item[(b)] In the special case $X=Y=G$, we will see that $\mathbf h^0(G,\mathcal O)=\mathbf h^1(G,\mathcal O)$, and that $\mathbf h^2(G,\mathcal O)$ is exactly the set of cocyclic $G$-matrices. In this case there is an isomorphism $\mathbf h^2(G,\mathcal O)/\mathbf h^0(G,\mathcal O)\cong H^2(G,\mu_n)$. The intermediate quotient $\mathbf h^1(G,\mathcal O)/\mathbf h^0(G,\mathcal O)$ is not seen by the theory of cocyclic matrices (cf. Theorem \ref{thm:710} below), but in the general case it is significant, see Theorem \ref{nuniq}(b) below.
   \end{itemize}
   \end{rem}

	\section{$H^1$-developed matrices}\label{sec3}
	
	In this section we restrict our attention to the construction of $H^1$-developed $X\times Y$-matrices. We call this the \emph{split} case. Recall that $G$ acts by permutations on $X$ and $Y$, and this action induces a map $G\to Perm(X)\times Perm(Y)$. Restricting our attention to the split case, Problem \ref{prob1} reduces to the following:
	
	\begin{prob}\label{prob2} Given a group $G$, two transitive $G$-sets $X,Y$, and a $G$-stable irreducible subset $\mathcal O \subset X\times Y$, construct all matrices $A\in M\mu_n(X,Y)$ with $\supp(A)=\mathcal O$, admitting a homomorphism $s:G\to Mon(X,\mu_n)\times Mon(Y,\mu_n)$, which lifts the homomorphism $G\to Perm(X)\times Perm(Y)$, such that $s(g)A=A$ for all $g\in G$.
	\end{prob}

	The first step towards constructing $H^1$-developed matrices, is to obtain the monomial covers of $(G,X,Y)$, with $\widetilde{G}=G$.
	Let $D_X=D_X(\mu_n)$ be the group of diagonal $\mu_n$-matrices indexed by $X$. Consider the following (split) exact sequence in the horizontal direction:\\
	\be \label{diag2}
		\begin{tikzcd}			1 \arrow[r] &D_X\times D_Y \arrow[r]& Mon(X,\mu_n)\times Mon(Y,\mu_n)\arrow{r}{\pi} & Perm(X)\times Perm(Y) \arrow[bend right=30,swap]{l}{t}
			\arrow[r] & 1\\
			& & & G \arrow[swap]{u}{\beta} \arrow[ul,"s",dashed,shorten >= 7pt,shorten <= 7pt]
			\arrow[bend right=-20,dashed,shorten >= 7pt,shorten <= 7pt]{ul}{s_0}
		\end{tikzcd}.
	\ee
	The map $\pi$ is $abs\times abs$ and $\beta$ is the structure map of the $G$-action on $X$ and $Y$. $\pi$ has a section $t:Perm(X)\times Perm(Y)\to Mon(X,\mu_n)\times Mon(Y,\mu_n)$ mapping a pair of permutation matrices to itself. A monomial cover of $(G,X,Y)$ with $\widetilde{G}=G$ is then given by a homomorphism $s$, as in the diagram \eqref{diag2}, such that $\pi\circ s=\beta$. We have the \emph{unsigned monomial cover}, which is given by $s_0:=t\circ \beta:G\to Mon(X,\mu_n)\times Mon(Y,\mu_n)$. The `difference' between $s$ and the unsigned monomial cover $s_0$ is measured by the first cohomology, which we now explain Notice that the map $s$ does not necessarily commute with the other maps.\\

	At this point, the first group cohomology of $G$ enters the picture. The readers unfamiliar with basic definitions of group cohomology are referred to Appendix \ref{appA} and the references thereof.\\ 
	
	The groups $D_X,D_Y$ are $G$-modules, via the actions of $G$ on $X,Y$.  There is a bijection 

	\be \label{moncov-homology} H^1(G,D_X\times D_Y)\ \longleftrightarrow \ \text{Maps $s$ as in \eqref{diag2} up to conjugcy by }
		D_X\times D_Y. 
	\ee
	This isomorphism is given by the following recipe: Given a map $s=(s_X,s_Y)$, we write uniquely $s_X(g)=d_X(g)p_X(g)$ and $s_Y(g)=d_Y(g)p_Y(g)$ for diagonals $d_X,d_Y$ and permutations $p_X,p_Y$. Then it is easy to check that the function $g\mapsto (d_X(g),d_Y(g))$ is a 1-cocycle with values in $D_X\times D_Y$. See Appendix \ref{appA} and Theorem \ref{class_sect} for more details.\\
	
	There is a natural isomorphism
	\be \label{diag3} H^1(G,D_X\times D_Y) \cong  H^1(G,D_X) \oplus H^1(G,D_Y).\ee
	which is already an isomorphism at the level of 1-cocycles $Z^1(G,D_X)\oplus Z^1(G,D_Y)\stackrel{\cong}{\longrightarrow} Z^1(G,D_X\times D_Y)$ by $(z_1,z_2) \mapsto z(g)=(z_1(g),z_2(g))$.\\

	 We naturally identify $D_X=\mu_n[X]$ and $D_Y=\mu_n[Y]$ as $G$-modules, where 
	 $$\mu_n[S]:= \left\{\sum_{s\in S} \zeta_s[s] \ \big| \ \zeta_s \in \mu_n\right \}$$
	 is the free abelian group of formal sums with coefficients in $\mu_n$, over a basis $G$-set $S$, with the $G$-action inherited from $S$. The identification sends a diagonal matrix $\diag(v)$ indexed by $S$ to $\sum v_s[s]$. This requires us to compute $H^1(G,\mu_n[X])$ and $H^1(G,\mu_n[Y])$. We will invoke the notion of (co)induced moduels and the Eckmann-Shapiro Lemma. Given a group $G$, a subgroup $H$ and an $H$-module $N$, the coinduced module $CInd_H^GN$ is a $G$-module, with the property that $H^i(H,N)\cong H^i(G,CInd_H^GN)$ for all $i\ge 0$. For definitions and properties of (co)induced modules and for a precise statement of the lemma of Eckmann-Shapiro, the reader in enouraged to look at Appendix \ref{appB}.\\
 	
	 Fix basepoints $x_0\in X$ and $y_0\in Y$. Let $H_X$ (resp. $H_Y$) be the stabilizer of $x_0$ (resp. $y_0$). 
	 Let $N=\mu_n$ be considered as an $H_X$ (resp. $H_Y$)-module with the trivial action. The next lemma claims that $\mu_n[X]$ and $\mu_n[Y]$ are coinduced $G$-modules from $N=\mu_n$.
  	
  	\begin{lem}\label{3.4}
  		As $G$-modules, $$CInd_{H_X}^G\mu_n\cong \mu_n[X], \text{ and } CInd_{H_Y}^G\mu_n\cong \mu_n[Y].$$ The identification for $X$ is as follows: 
  		We have,
  		\begin{multline*}
  		CInd_{H_X}^G \mu_n:=Hom_{\ZZ[H]}(\ZZ[G],\mu_n)\cong 
  		Maps_H(G,\mu_n)\\:=\left\{\text{Functions }\phi:G \to \mu_n \ | \phi(hg)=\phi(g),\ \forall h\in H, \forall g\in G\right\}.
  		\end{multline*}

  		To a function $\phi\in Maps_H(G,\mu_n)$ we match
  		$$\phi \longleftrightarrow \sum_{x\in X} \phi(g_x^{-1})[x]\in \mu_n[X],$$
  		where $g_x\in G$ are elements such that $g_x(x_0)=x$ for all $x\in X$.
  	\end{lem}
  
	 \begin{proof}
	 	It is straightforward to check that this identification is well defined, respects the $G$-action, and bijective. Note that the transitivity of $X$ is important. We leave the details to the reader.
	 \end{proof}
 	 
 	 This puts us in the position to apply the Eckmann-Shapiro Lemma (Theorem \ref{shap}). To this end we need to set up representatives for the left cosets in $H_X\setminus G$ and $H_Y\setminus G$. As in Lemma \ref{3.4} we fix elements $g_x\in G$ such that $g_xx_0=x$ for all $x\in X$, and $g_y\in G$ such that $g_yy_0=y$ for all $y\in Y$. With this choice, we take $\{g_x^{-1}\}$ as the set of coset representatives to $H_X\setminus G$, and $\{g_y^{-1}\}$ as the set of coset representatives to $H_Y\setminus G$. The following theorem is a direct consequence of the Eckmann-Shapiro lemma.
 	 
	 \begin{thm}\label{cor:35}
	 We have an isomorphism
	 \be\label{38} H^1(G,\mu_n[X]\times \mu_n[Y])\cong Hom(H_X,\mu_n)\oplus Hom(H_Y,\mu_n).\ee
	 
	 The two inverse maps consisting of the isomorphism are as follows. Given two homomorphisms $\psi_X:H_X \to \mu_n$ and $\psi_Y:H_Y\to \mu_n$, a matching 1-cocycle $z\in Z^1(G,\mu_n[X]\oplus \mu_n[Y])$ is given by
	 \be\label{3.9} z(g)\ = \ \left(\sum_{x\in X}\psi_X((g_x^{-1}g)\cdot\overline{(g_x^{-1}g)}^{-1})[x],\sum_{y\in Y}\psi_Y((g_y^{-1}g)\cdot \overline{(g_y^{-1}g)}^{-1})[y] \right).\ee
	 In this equation we use the notation $\overline{q}$ for $q\in G$ to be the element $\overline{q}:=g_z^{-1}$ for the point $z=q^{-1}x_0\in X$. Similarly for $Y$.\\
	
	 In the other direction, given a 1-cocycle $z$, the two homomorphisms are given by \
	 \be \psi_X(h)=z_X(h)_{x=x_0} \text{ and } \psi_Y(h)=z_Y(h)_{y=y_0}.\ee
	 \end{thm}
 
 	\begin{proof}
 		This is a consequence of the Eckmann-Shapiro Lemma (see Theorem \ref{shap} and the formulae there), Lemma \ref{3.4} and that fact that $H^1(G,M)=Hom(G,M)$ for modules $M$ with a trivial $G$-action. 
 	\end{proof}
 
 	 \begin{rem}
 	 	\begin{itemize}
 	 		\item[]
 	 		\item[(a)] Warning: In formula \eqref{3.9}, be aware that the bar notation $\overline{g}$ has twofold meanings, as the coset representative of $g$ in $H_X\backslash G$ or $H_Y \backslash G$, depending on the context. 
 	 		\item[(b)] The precise 1-cocycle $z$ we obtain may depend on the choice of our coset representatives. However, the theory guarantees that its cohomology class is well-defined.
 	 	\end{itemize}
 	 	
 	 \end{rem}
     
     \subsection{Determining the $H^1$-Monomial Covers from the Cohomology.}
     At this point we have collected enough information to construct all monomial covers of $(G,X,Y)$ with $\widetilde{G}=G$. To this end, we need to obtain maps $s:G\to Mon(X,\mu_n)\times Mon(Y,\mu_n)$, that fit in diagram \eqref{diag2}, such that $s=\pi\circ \beta$. Every such map $s$ can be computed from a 1-cocycle $z\in Z^1(G,D_X\times D_Y)$ by the rule
     \be\label{cov} s(g) = z(g)s_0(g).\ee Two cohomologous 1-cocycles will result in $D_X\times D_Y$-conjugate maps $s$. Conversely, two $D_X\times D_Y$ conjugate maps $s,s'$ yield via \eqref{cov} cohomologous 1-cocycles. This is the content of Theorem \ref{class_sect}.\\  
     
     On the other hand, 1-cocycles $z(g)$ can be obtained in turn, up to cohomology,  from a pair of homomorphisms 
     $$\psi_X\in Hom(H_X,\mu_n) \text{ and } \psi_Y\in Hom(H_Y,\mu_n),$$ by applying formula \eqref{3.9}, and 
     using Corollary \ref{cor:35}. This proves the following
     
     \begin{thm}\label{pairA} [cf. Theorem \ref{nuniq}(b) below]\label{cover/conj} Let $X,Y$ be $G$-transitive $G$-sets, and let $H_X,H_Y$ be the stabilizers of the basepoints $x_0\in X$ and  $y_0\in Y$. 
     	Then there is a bijection
     	$$\left\{ \begin{matrix} \text{$\mu_n$-valued Monomial Covers with }\widetilde{G}=G\\
     	\text{up to conjugacy by } D_X\times D_Y    \end{matrix} \right\}  \longleftrightarrow Hom(H_X,\mu_n)\oplus Hom(H_Y,\mu_n).$$
     \end{thm}
	 We give the next lemma as an exercise to the reader:
     \begin{lem}[Principal monomial covers]
		Suppose that $\chi_1,\chi_2:G\to\mu_n$ are two characters. Let $\psi_X=\chi_1|_{H_X}$ and $\psi_Y=\chi_2|_{H_Y}$. Then a monomial cover that corresponds to the pair $(\psi_X,\psi_Y)$ is given by the formula (Using the notation used in Diagram \eqref{diag4})
		$$\pi_X(g)=\chi_1(g)p_X(g),\ \ \pi_Y(g)=\chi_2(g)p_Y(g).$$
		Such monomial covers are called \emph{principal} (cf. Example \ref{examp3}). $\Box$
     \end{lem}
	 In this case the $1$-cocyle in formula \eqref{3.9} should be $$z(g)=\left(\sum_x \chi_1(g)[x],\sum_y \chi_2(g)[y]\right).$$ This is not quite what we get if we substitute $\psi_X=\chi_1$ and $\psi_Y=\chi_2$ in \eqref{3.9}. However, the two cocycles are cohomologous. This amounts to two monomial covers that are conjugates in $D_X\times D_Y$.
     
     \subsection{Example - bilinear form matrices}\label{BilMat}

	 Let $F$ be a finite field and let $V$ be a finite dimensional vector space over $F$ of dimension $d$.
	 The projective space $P(V)$ is the set $V\setminus\{0\}$ modulo the equivalence relation defined by $v\sim_P w$ if and only if $F\cdot v=F\cdot w$. We write $[v]\in P(V)$ for the equivalence class of $v$. 
	 Let $B:V\times V\to F$ be a bilinear form. Let $G\subseteq GL(V)$ be the subgroup of all $g$ such that $B(gv,gw)=B(v,w)$ for all $v,w\in V$.\\

	 We construct a $P(V)\times P(V)$-matrix $M$ as follows. We fix a homomorphism $\chi:F^\times \to \mu_n$, and extend by $\chi(0)=1$. For each class $[v]\in P(V)$, we fix a choice of a representative $[v]\ \widetilde{}\in V$, and let $$M_{[v],[w]} \ := \ \chi\circ B([v]\ \widetilde{},[w]\ \widetilde{}\ ).$$
	 This definition depends on the choices made, but note that different choices of representatives will produce diagonally equivalent matrices. The group $G$ acts naturally on $X=P(V)$, and for $g\in G$ we have that $gM$ is diagonally equivalent to $M$ since the choice of representatives is not compatible with the action of $G$. Thus $M$ is cohomology-developed in the sense of Definition \ref{coh-dev} with respect to $G$, and the group action defines a monomial cover $s:G\to Mon(X,\mu_n)\times Mon(Y,\mu_n)$ given by 
	 $$s(g)\ := \ (D(g)\Pi(g),D(g)\Pi(g)),$$ where $\Pi(g)$ is the permutation matrix of $g$ on $P(V)$ and $D(g)$ is the diagonal matrix defined by
	 \be \label{bilD} D(g)_{[v],[v]}\ := \ \chi\circ \frac{[gv]\ \widetilde{}}{g [v]\ \widetilde{}}\ .\ee
	 Here we `divide' proportional nonzero vectors.\\

	 Choose a basepoint $x_0=[\xi]\in P(V)$ and restricting \eqref{bilD} to $[\xi]$ and $H_X$ we get a map
	 $$\psi_X:H_X\to \mu_n, \psi_X(h)=\chi\circ \frac{[\xi]\ \widetilde{}}{h [\xi]\ \widetilde{}}\ ,$$
	and it is easy to check that this is a homomorphism. Then $(\psi_X,\psi_X)$ is the pair of homomorphisms that matches the monomial cover according to Theorem \ref{pairA}. Note that $\psi_X$ is independent of the choice of representatives, as it really should be.\\

	As a special case we can construct the Payley matrices. In this case we assume that $V=F^2$ is two dimensional and $|P(V)|=|F|+1$. Let $B:V\times V\to F$ be the dererminant form and consequently $G=SL(V)\cong SL(2,F)$. Let us fix the choice $[(x,1)]\ \widetilde{}=(x,1)$ and $[(1,0)]\ \widetilde{}=(1,0)$. With respect to this choice we see that $M_{[(x,1)],[(y,1)]}=\chi(x-y)$ for $x\neq y$ and $1$ if $x=y$. We get the well-known Paley type I matrix (Which is Hadamard If $|F|\equiv 3 \mod 4$ and $n=2$).  Notice the similarity of the affine part with the matrix in \eqref{PayKer}
	below. Moreover, taking $\xi=(0,1)$ then $H_X$ is the group of lower trinagular matrices and we compute 
	$$\psi_X(h)=\chi\circ h_{2,2}^{-1}=\chi\circ h_{1,1}.$$

     \subsection{Orientability of $H^1$-development.}
     Having constructed a monomial cover $G\to Mon(X,\mu_n)\times Mon(Y,\mu_n)$, we now face the issue of orientability, as discussed in \S\ref{lift}. The first observation to make is that orientability of a point $(x,y)$ (or its orbit) depends only on the cohomology class, i.e. only on the two homomorphisms $\psi_X,\psi_Y$.
     
     \begin{lem}
     The orientability of a given point $(x,y)\in X\times Y$ w.r.t. a monomial cover $s:G \to Mon(X,\mu_n)\times Mon(Y,\mu_n)$ depends only on the corresponding homomorphisms $\psi_X,\psi_Y$ and the point $(x,y)$ 
     \end{lem} 
 
     \begin{proof}
     	By Theorem \ref{cover/conj}, two monomial covers that correspond to the same pair $(\psi_X,\psi_Y)$ are diagonally conjugate. Let $s_1,s_2$ be diagonally conjugate monomial covers of $(G,X,Y)$. If $h\in G$ stabilizes a point $(x,y)$, and $s_i(h)=(P_i,Q_i)$, $i=1,2$, then $P_2=DP_1D^{-1}$ and $Q_2=EQ_1E^{-1}$ for some $D,E$ diagonal. This implies that  $\sign(x,P_1)=\sign(x,P_2)$ and $\sign(y,Q_1)=\sign(y,Q_2)$, which proves the lemma. 
     \end{proof}
 
     It thus reasonable to expect a direct criterion for orientability, depending only on $\psi_X,\psi_Y$ and $(x,y)$. Note that the stabilizer of the point $(x,y)$ is the intersection $Stab(x,y):=g_{x}H_Xg_x^{-1}\cap g_yH_Yg_y^{-1}$.
     \begin{prop}\label{nc}
     	A point $(x,y)$ is orientable for an $H^1$-monomial-cover corresponding to $(\psi_X,\psi_Y)$, if and only if 
     	\be\label{310} \forall g\in Stab(x,y), \ \ \ \psi_X(g_x^{-1}gg_x) = \psi_Y(g_y^{-1}gg_y).\ee
     \end{prop}
	 Note that for $g\in Stab(x,y)$ we have $g_x^{-1}gg_x\in H_X$ and $g_yH_Yg_y^{-1}\in H_Y$, so the lemma makes sense.
     
     \begin{proof}
     	For $g\in G$ s.t. $gx=x$ and $gy=y$, the signs $\sign(x,g)$ and $\sign(y,g)$ are given (in the additive notation) by the $x$ and $y$ coordinates of two components of \eqref{3.9} respectively. Let us compute those components. The two conditions $gx=x$ and $gy=y$ are equivalent to the condition that $g\in g_{x}H_Xg_x^{-1}\cap g_yH_Yg_y^{-1}$. For such $g$ we have by \eqref{3.9} that $\sign(x,g)=\psi_X((g_x^{-1}g)\cdot \overline{(g_x^{-1}g)}^{-1})$. But $$g\in g_xH_X g_x^{-1}\implies g_x^{-1}g\in H_Xg_x^{-1} \implies \overline{g_x^{-1}g}=g_x^{-1},$$ and consequently $\sign(x,g)=\psi_X(g_x^{-1}gg_x)$. The analogous formula holds for $y$. Thus the condition $\sign(x,g)=\sign(y,g)$ is equivalent to $\psi_X(g_x^{-1}gg_x) = \psi_Y(g_y^{-1}gg_y)$.
     \end{proof}

	 \begin{cor}
		For principal monomial covers, $(x,y)$ is orientable if and only if for all $g\in Stab(x,y)$ $\psi_X(g)=\psi_Y(g)$.
	 \end{cor}
 
     \begin{ex}\label{159}
     	Let $G=A_6$, the alternating group on the set $S=\{1,2,3,4,5,6\}$. Consider the transitive $G$-set 
     	$$X=Y= \{ K\subset S\ | \ |K|=2\}.$$ The cardinality of $X$ is $15$. Let $x_0=\{1,2\}$, whose stabilizer $H_X$ is a group of order $24$. Actually, $H_X$ is the subgroup of $A_6$ of all permutations that preserve the partition $S=\{1,2\}\cup \{3,4,5,6\}$. By restricting to $\{3,4,5,6\}$ we see that $H_X\cong S_4$. There is a split exact sequence
     	$$\begin{tikzcd}
     	1\ar[r] & A_4\ar[r] & H_X\ar{r}{\psi_X} & \mu_2 \ar[r] & 1
     	\end{tikzcd},$$
     	where $A_4$ is interpreted as the subgroup fixing $1$ and $2$, and $\psi_X:H_X\to \mu_2$ is the map given by $\psi_X(\sigma)=sgn(\sigma|_{\{3,4,5,6\}})$ ($sgn$:=permutation sign). We may use the pair $(\psi_X,\psi_X)$ to construct a monomial cover  $A_6\to Mon(X,\mu_2)\times Mon(X,\mu_2)$. We shall now restrict our attention to orientability w.r.t. this monomial cover.\\
     	
     	There are three orbits: the diagonal $O_1=\{(x,x)|x\in X\}$,
     	$O_2=\{(x,y)| |x\cap y|=1\}$, and $O_3=\{(x,y)| \ x\cap y=\varnothing\}.$ The diagonal $O_1$ is obviously orientable, by Proposition \ref{nc}. Let us show that $O_2$ is orientable and $O_3$ isn't. Take the pair $x=\{1,2\}$ and $y=\{3,4\}$. Then $g\in g_xH_Xg_x^{-1}\cap g_yH_Yg_y^{-1}$ is equivalent to $gx=x$ and $gy=y$, which is equivalent to $g$ preserving the partition $S=\{1,2\}\cup\{3,4\}\cup \{5,6\}$. We may take $g_x=1$ and $g_y=(1,3)(2,4)$. Thus $\psi_Y(g_y^{-1}gg_y)=sgn(g|_{\{1,2,5,6\}})$. This does not agree with $\psi_X(g_x^{-1}gg_x)=\psi_X(g)$, e.g. on $g=(1,2)(5,6)$. Thus $O_3$ is not orientable.\\
     	
     	As for $O_2$, we may take the pair $x=\{1,2\}$ and $y=\{1,3\}$. The condition $gx=x$ and $gy=y$ is now equivalent to $g$ preserving the partition $S=\{1\}\cup \{2\}\cup \{3\}\cup\{4,5,6\}$. Take $g_x=1$ and $g_y=(2,3)(5,6)$. Then $\psi_Y(g_y^{-1}gg_y)=sgn(g|_{\{2,4,5,6\}})=sgn(g|_{\{4,5,6\}})=\psi_X(g_x^{-1}gg_x)$ and $O_2$ is orientable.\\
     	
     	The \emph{weight} of $O_2$ (=number of nonzero entries in each row or column of the characteristic matrix) is $8$. $O_1$ has weight $1$. Later on we will use this monomial cover to construct a  $W(15,9)$.

		\begin{rem}
			This construction extends to $S_6$. More precisely, $S_6$ still acts on $X$, and produces the same orbits $O_1,O_2,O_3$ on $X\times X$. The stabilizer $H_X\cong S_2\times S_4$ and there are four characters on $H_X$: The sign at $S_4$, the sign at $S_2$, the total sign and the trivial character. In the first two cases we get the previous orientability results. In the last two cases (which are principal monomial covers), all orbits are orientable.
		\end{rem}
     \end{ex}

	 \subsection{Putting all information together}
	We collect all of the analysis above, mainly Theorems \ref{cor:35} and \ref{pairA}, Proposition \ref{nc}, and equations \eqref{moncov-homology} and \eqref{cov} to form a practical algorithm. The input is an integer $n$, a finite group $G$, and two transitive $G$-sets $X,Y$. The output is a list of all $H^1$-developed $X\times Y$ with values in $\mu_n^+$ over this data, containing an exhaustive list of all CDMs up to diagonal equivalece (but the list may contain more that one element in a class, see Theorem \ref{nuniq}).\\
	
	\begin{alg}
		\begin{itemize}
			\item[] {\Large Subroutine MONCOV($g,\psi_X,\psi_Y$)} 
			\item[] Compute the $H^1$-developed cover $s:G\to Mon(X,\mu_n)\times Mon(Y,\mu_n)$ from two homomorphisms $\psi_X$ and $\psi_Y$.\\
			\begin{algorithmic}[1]
				\State Input global variables $G,X,Y,x_0,y_0,H_X,H_Y$
				\State Input $g\in G$ and homomorphisms $\psi_X:H_X\to \mu_n$ and $\psi_Y:H_Y\to \mu_n$.
				\State Compute the 1-cocycle value $z(g)\in D_X\times D_Y$ using Formula \eqref{3.9}.
				\State Let $(P,Q)\gets z(g)s_0(g)$
				\State \textbf{return} $(P,Q)$.
			\end{algorithmic}
		\end{itemize}
	\end{alg}
    \vskip 1cm

	\begin{alg}\label{alg}
		\begin{itemize}
			\item[] {\Large MAIN($G,X,Y,n$) Algorithm}
		\end{itemize}
	\begin{algorithmic}[1]
		\State Input: $G$, $X,Y$ and $n$.
		\State Initialize Collection $\mathcal A = \varnothing$.
		\State Compute all $G$-orbits on $X\times Y$.
		\State Choose basepoints $x_0\in X$ and $y_0\in Y$. 
		\State Compute the stabilizers $H_X,H_Y\subseteq G$.
		\State Compute elements $g_x$ and $g_y$ for all $x\in X$ and $y\in Y$ such that $x=g_x(x_0)$ and $y=g_y(y_0)$.
		\State Store globally $G,X,Y,\{g_x\},\{g_y\},x_0,y_0,H_X,H_Y$.
		\For {homomorphisms $\psi_X:H_X\to \mu_n$ and $\psi_Y:H_Y\to \mu_n$}
		\State \label{l9} Determine the set $OrinetOrbits$ of orientable orbits by using Proposition \ref{nc}.
		\State Initialize $A=0$.
		\For {$Z \in (\mu_n^+)^{\mathrm{OrientOrbits}}$}
		\For {$O\in$ OrientOrbits}

				\State Choose a point $(x_O,y_O)\in O$ (an \emph{orbit-head})
				\State Set up $\zeta\gets Z_O$.

				\For {$g\in G$}\label{20}
				    \State compute the pair $(P,Q)=\mathrm{MONCOV}(g,\psi_X,\psi_Y)$
				    \State Update $A_{gx_O,gy_O}\gets \zeta \cdot\sign(gx_O,P) \cdot \sign(gy_O,Q)^*$.
			\EndFor
			\EndFor
			
			\State Add $A$ to $\mathcal A$.
			\EndFor

		\EndFor
		\State \textbf{return} $\mathcal A$.
		
	\end{algorithmic}
    \end{alg}

    Here is a short outline of the main algorithm. In steps 1-7 we pre-compute the group-theoretic data to be used later and store it in memory. In step 8 we start looping over all pairs of homomorphisms $(\psi_X,\psi_Y)$, which is the data needed to obtain all monomial covers up to diagonal conjugacy (Theorem \ref{cover/conj}). Then step 9 singles out the orientable orbits, according to Proposition \ref{nc}. In steps 11-14 we loop over all orientable orbits $O$, and choose a value  $\zeta\in \mu_n^+$ at a certain basepoint $(x_O,y_O)\in O$. In steps 15-18 we loop over all elements $g\in G$, and compute the matrix entry $A_{gx_O,gy_O}$ from the value $A_{x_O,y_O}=\zeta$, and the $\widetilde{G}$-action. This computes a $\widetilde{G}$-invariant matrix $A$ for the specific monomial cover, and adds it to the memory.\\

	\begin{ex}\label{159color}
		Continuing Example \ref{159}, let us implement some parts of the algorithm.  First, we must compute representatives $\{g_x\}$ and $\{g_y\}$. We have $X=Y$ and $x_0=y_0=\{1,2\}$. Let $x=\{p,q\}$, for $p<q$. We need to obtain an element $g_x$ such that $g_x\{1,2\}=\{p,q\}$. Let
		$$g_x=\begin{cases}
		1_G & \text{if } p=1,q=2\\
		(1,2,q) & \text{if } p=2<q\\
		(1,q,2) & \text{if } p=1<2<q\\
		(1,p)(2,q) & \text{if } 2<p<q
		\end{cases}.$$
		We make the same choice for $Y$. Consider the orbits $O=O_2$, which is orientable, and take the basepoint $(x_O,y_O)=(\{1,2\},\{2,3\})$. Fix the value $A_{\{1,2\},\{2,3\}}=\zeta=1$. We wish to compute the value of $A_{\{1,3\},\{3,6\}}$ given that $A$ is $\widetilde{G}$-invariant w.r.t. the monomial cover defined by $(\psi_X,\psi_X)$ of Example \ref{159}.\\
		
		Take $g=(2,3,6)$, which satisfies $g\{1,2\}=\{1,3\}$ and $g\{2,3\}=\{3,6\}$. It's time to use formula \eqref{3.9} to compute MONCOV($g,\psi_X,\psi_X$) at the point $x=\{1,3\}$ and $y=\{3,6\}$. We have $g_x=(1,3,2)$, $g_y=(1,3)(2,6)$ and
		\begin{align*}
		g_x^{-1}g&=(1,2,3)(2,3,6)=(1,2)(3,6)\\
		g_y^{-1}g&=(1,3)(2,6)(2,3,6)=(1,3,2).
		\end{align*}
		Next we need to compute $\overline{g_x^{-1}g}$ and $\overline{g_y^{-1}g}$. We have (Corollary \ref{cor:35}) $\overline{g_x^{-1}g}=g_z^{-1}$ for $z=(g_x^{-1}g)^{-1}\{1,2\}=\{1,2\}$. By definition, $g_z=1_G\implies \overline{g_x^{-1}g}=1_G$. Likewise $\overline{g_y^{-1}g}=g_u^{-1}$ for $u=(g_y^{-1}g)^{-1}\{1,2\}=\{2,3\}$. By definition, $g_u=(1,2,3)\implies \overline{g_y^{-1}g}=(1,3,2)$. Putting this in Formula \eqref{3.9} we obtain $(P,Q)=\mathrm{MONCOV}(g,\psi_X,\psi_X)$ and we compute
		\begin{align*}
		\sign(x,P)&=\psi_X((g_x^{-1}g)\cdot(\overline{g_x^{-1}g})^{-1})=\psi_X((1,2)(3,6))=-1\\
		\sign(y,Q)&=\psi_X((g_y^{-1}g)\cdot(\overline{g_y^{-1}g})^{-1})=\psi_X(1_G)=1.
		\end{align*}
		In line 17 of the main algorithm we conclude that $A_{\{2,3\},\{3,6\}}=-1$.
	\end{ex}

	\subsection{The group structure of $H^0$ and $H^1$-developed matrices} 
	 We have defined in \S 2 the filtration groups of CDMs $\mh^i(G,\mathcal O)$, $i=0,1,2$. We will show now that the group structure of $\mh^i(G,\mathcal O)$ is closely related to the cohomology groups $H^i(G,-)$, at least for $i=0,1$. This strengthens the relationship between cohomology development and group cohomlogy.\\ 
	 
	The \emph{diagonal map}, $\Delta: \mu_n \to \mu_n[X]\oplus \mu_n[Y]$ sending $\zeta\in \mu_n$ to $(\sum_{x\in X}\zeta[x],\sum_{y\in Y}\zeta[y])$, is a map of $G$-modules. By covariant functoriality, $\Delta$ induces a map on cohomology groups: $\Delta_*:H^1(G,\mu_n)\to H^1(G,\mu_n[X]\oplus\mu_n[Y])$. The \emph{$0$th Cohomology Group} of a module $M$ is defined as $H^0(G,M)=M^G$, the subgroup of $G$-invariant elements. 
	We have
	
	\begin{thm}\label{nuniq}
		\begin{itemize} For an irreducible $\mathcal O$,
			\item[]
			\item[(a)] There is an injection of groups
			\be\label{H0} H^0(G,\mu_n[\mathcal O])\hookrightarrow \mh^0(G,\mathcal O),\ee and the right hand side equals the $D$-equivalence closure of the image of the left hand side. 
			\item[(b)] 	There is an injection of groups
			\be\label{inj1}\mh^1(G,\mathcal O)/\mh^0(G,\mathcal O) \hookrightarrow H^1(G,\mu_n[X]\oplus \mu_n[Y])/\Delta_*H^1(G,\mu_n).\ee
			and the image is generated by classes in $H^1(G,\mu_n[X]\oplus \mu_n[Y])$ that yield  monomial covers for which all orbits in $\mathcal O$ are orientable.
		\end{itemize}

	\end{thm}

	\begin{proof}
		In this proof, we will identify a pair of 1-cocycles, $(c_X,c_Y)$, for $c_X\in Z^1(G,\mu_n[X])$ and $c_Y\in Z^1(G,\mu_n[Y])$, with the $1$-cocycle $c\in Z^1(G,\mu_n[X]\oplus\mu_n[Y])$ defined by $c(g)=(c_X(g),c_Y(g)))$.\\
		
		(a) $H^0(G,\mu_n[\mathcal O])=\mu_n[\mathcal O]^G$, and every $G$-invariant element\\  $t=\sum_{(x,y)\in \mathcal O} \zeta(x,y)[(x,y)]\in \mu_n[\mathcal O]$ is being identified with the $G$-invariant matrix $A(t)_{x,y}=\zeta(x,y)$. The association $t\mapsto A(t)$ gives the map in \eqref{H0}, which is clearly injective. Notice that this map is a group homomorphism (multiplication of matrices being Hadamard multiplication). The second assertion in (a) follows from the definition of $\mh^0(G,\mathcal{O}).$\\

		 (b) The set $\mh^1(G,\mathcal O)$ is the set of all matrices $A\in M\mu_n(X,Y)$ with support $\mathcal O$ such that there exists a pair of 1-cocycles $c_X\in Z^1(G,D_X)$ and $c_Y\in Z^1(G,D_Y)$ with $c_X(g)(gA)c_Y(g)^*=A$ for all $g\in G$. As usual we identify $D_X=\mu_n[X]$ and $D_Y=\mu_n[Y]$. We define the map in \eqref{inj1} by sending $A$ to the cohomology class of the pair $(c_X,c_Y)$ modulo $\Delta_*H^1(G,\mu_n)$. We need to show that this map is well-defined.\\
		 
		  The pair $(c_X,c_Y)$ may not be unique. If $(c'_X,c'_Y)$ is another pair for $A$, then the irreducibility of $\mathcal O$ (Lemma \ref{irred}(c)) assures that for all $g\in G$,   $c_X(g)c'_X(g)^{-1}$ and $c_Y(g)c'_Y(g)^{-1}$ are scalar matrices with the same scalar value. In particular $c_X-c_{X}'=c_Y-c_Y'$ is a constant function (passing to additive notation). It follows that the pair of 1-cocycles $(c_X-c'_X,c_Y-c'_Y)$  is in the image of $\Delta_*Z^1(G,\mu_n)$ (where the map $\Delta_*$ has been originally defined at the level of cocycles) and $A\mapsto (c_X,c_Y)$ is a well defined map: 
		 $$ \mh^1(G,\mathcal O) \to Z^1(G,\mu_n[X]\oplus \mu_n[Y])/\Delta_*Z^1(G,\mu_n).$$  It is easy to check that this map is a group homomorphism. This map descends to cohomology, giving a map
		 
		 \be \label{hom2}
		  \mh^1(G,\mathcal O) \to H^1(G,\mu_n[X]\oplus \mu_n[Y])/\Delta_*H^1(G,\mu_n).\ee
		 
		 Let us compute the kernel. If a matrix $A\in \mh^1(G,\mathcal O)$ is in the kernel, it means that the pair $(c_X,c_Y)$ is cohomologous to some $\Delta_*\lambda$ where $\lambda:G\to \mu_n$ is a $1$-cocycle. By Theorem \ref{cover/conj},  cohomologous 1-cocycles give rise to
		 $D_X\times D_Y$-conjugate monomial covers. Thus we may generate another matrix $A'\in \mh^1(G,\mathcal O)$, with $A'\sim_D A$, and the associated cocycle becomes $(c'_X,c'_Y)=\Delta_*(\lambda)$. But this implies that $A'$ is $G$-invariant, since the monomial cover $(P,Q)$ satisfies
		 $P=\xi |P|$ and $Q=\xi |Q|$ for some $\xi=\xi(g)\in \mu_n$. This implies that $A\in \mh^0(G,\mathcal O)$. Conversely, if $A\in \mh^0(G,\mathcal O)$, then $A$ is $D$-equivalent to a $G$-invariant matrix $A'$. For $A'$ we may assign the unsigned monomial cover, which is associated to the zero 1-cocycle. Thus for $A$ we can assign a $D_X\times D_Y$-conjugate of the unsigned monomial cover, and in turn this yields a $1$-cocycle  $(c_X,c_Y)$ which is a cohomologous to $0$, hence $A$ is in the kernel. We have proved that the kernel of \eqref{hom2} is $\mh^0(G,\mathcal O)$ and that \eqref{inj1} is well-defined and injective.\\
		 
		 It remains to describe the image of \eqref{inj1}. The pair of cocycles $(c_X,c_Y)$ used in the construction of this map  gives rise to a monomial cover $(P,Q):G \to Mon(X,\mu_n)\times Mon(Y,\mu_n)$ defined by $P(g)=c_X(g)|P(g)|$ and $Q(g)=c_Y(g)|Q(g)|$ (Again using the identification $\mu_n[S]=D_S$). If $(c_X,C_Y)$ is in the image, then it comes from a matrix $A$ with support $\mathcal O$ satisfying $A=c_X(g)(gA)c_Y(g)^*=P(g)AQ(g)^*$. This implies in turn that all orbits in $\mathcal O$ are orientable for this action.  Conversely, if an element of the RHS yields a monomial cover such that all orbits in $\mathcal O$ are orientable, then it is possible as per Algorithm \ref{alg} to generate a matrix $A\in \mh^1(G,\mathcal O)$ whose image is $(c_X,c_Y)$.
	\end{proof}

	\begin{rem}\label{rem:1} Assume that $\mathcal O$ is irreducible.
		\begin{enumerate}
			\item Dividing by $\mh^0(G,\mathcal O)$ in the LHS of \eqref{inj1} has a twofold meaning. When we divide by the subgroup of $G$-invariant matrices 		
			we account for that different possible choices of the values values $z_O$ at the orbit-heads $(x_O,y_O)$ as in the notation of Algorithm \ref{alg}, line 14.
			Secondly, we also divide by diagonal equivalences as they appear in $\mh^0(G,\mathcal O)$. This choice allows us to discuss cohomology groups (appearing in the RHS), rather than cocycle groups.
			\item It follows from (1) that different classes in $\mh^1(G,\mathcal O)/\mh^0(G,\mathcal O)$ are diagonally inequivalent. cf. Theorem \ref{class_sect}.
		\end{enumerate}
	\end{rem}

   \subsection{$H^1$-Development and Hadamard Equivalence}
   We know (Remark \ref{rem:1}(2)) that elements of $\mh^1(G,\mathcal O)$ that correspond to different classes in $H^1(G,\mu_n[X]\oplus \mu_n[Y])/\Delta_*H^1(G,\mu_n)$ must be diagonally inequivalent. We would like sometimes to know if they are also Hadamard inequivalent. This is not true in general and we will see an example in Remark \ref{HadProj} below. The following theorem gives a sufficient condition for Hadamard inequivalence.  We use the notation for the maps $\pi,s$ and $\beta$ as they appear in Diagram \eqref{diag2}. Let $\langle S \rangle$ denote the subgroup generated by a subset $S$ of a given group. Let $x^g$ denote the conjugation $gxg^{-1}$, and $S^g:=\{s^g\ | \ s\in S\}$.
   
   \begin{thm}\label{HadConj}
   	   Let $\mathcal O$ be irreducible, and $|A|$ be its characteristic matrix. Suppose that every $\xi\in \PermAut(|A|)$ satisfies $\xi\in \langle \beta(G)^\xi\cup \beta(G) \rangle $. Then different elements in $\mh^1(G,\mathcal O)/\mh^0(G,\mathcal O)$ are Hadamard inequivalent.
   \end{thm}
   Note that the assumption in the theorem is satisfied in the case that $\beta(G)=\PermAut(|A|)$.
   \begin{proof}
   	   Suppose that $A,A'\in \mh^1(G,\mathcal O)$ are Hadamard equivalent. Then $A'=PAQ^*$ for monomial $P,Q$. Taking absolute values we have $|A|=|P||A||Q|^*$, and $\xi:=(|P|,|Q|)\in \PermAut(|A|)$. Let $s,s':G\to Mon(X,\mu_n)\times Mon(Y,\mu_n)$ be monomial covers corresponding to $A,A'$ respectively, and $\beta=\pi\circ s=\pi\circ s'$. Write $s=(s_X,s_Y)$ and $s'=(s'_X,s'_Y)$. Then
   	   $$A'=PAQ^* \implies \forall g\in G,\ \ (P^{-1}s'_X(g)P,Q^{-1}s'_Y(g)Q)\in \Aut(A).$$
   	   
   	   Let $\Gamma=\langle s'(G)^{(P,Q)}\cup s(G)\rangle.$ Then $\Gamma\subseteq \Aut(A)$, and $\pi(\Gamma)=\langle \beta(G)^\xi\cup \beta(G)\rangle.$ By our assumption, $\xi=(|P|,|Q|)\in \pi(\Gamma)$, so there exists some pair $(P',Q')\in \Gamma$ above $\xi$. In particular $P=LP'$ and $Q=RQ'$ for \emph{diagonal} $L,R$. But since $(P',Q')\in \Aut(A)$, then 
   	   $$A'=PAQ^*=L(P'AQ^{'*})R^*=LAR^* \implies \ A'\sim_D A,$$
   	   and in particular $A,A'$ map to the same element in $\mh^1(G,\mathcal{O})/\mh^0(G,\mathcal{O})$.
   \end{proof}

   \section{Example of $H^1$-development - Paley's Conference and Hadamard matrices.} \label{sec:paley}
   The Paley Hadamard and conference matrices were introduced in the 1933 paper \cite{Paley33}. The construction there uses Legendre symbols over prime powers. To date this is the most dense known family of Hadamard matrices. Here we will reconstruct and reinterpret these matrices using our cohomology machinary.  
   \subsection{Paley's Conference matrix} \label{s41}
   Let $F=\FF_q$ be a finite field and take $X=Y=F$ acted upon by the group $G$ of affine transformations:
   \be \nonumber G\ = \ \left\{x\mapsto ax+b\ \big| \ a\in F^\times \text{ and } b\in F   \right\}.\ee
   Take an integer $n|(q-1)$ for $\mu_n$.
   The action of $G$ breaks $X\times Y$ to two orbits: the diagonal
   $O_0=\{(x,x)| x\in F\}$ and the off-diagonal $O_1=\{(x,y)| x,y\in F,\ x\neq y  \}.$
    We choose the basepoints $x_0=y_0=0\in F$.
    The stabilizers of the basepoints are $H_X=H_Y=\{x\mapsto ax \  | \ a\in F^\times\}$. We shall work with the homomorphisms  $\psi_X:H_X\to \mu_n$ and $\psi_Y:H_Y:\to \mu_n$ given by \begin{align*} \psi_X(x\mapsto ax)&=\left(\frac{a}{F}\right)_n\text{  (the $n$th power residue symbol)}\\
    \psi_Y&=1.\end{align*}
    The $n$th power residue symbol is a homomorphism $\left( \frac{\cdot}{F}\right):F^\times \to \mu_n$, sending a generator $\alpha$ of the cyclic group $F^\times$ to $\tau=\exp(2\pi i/n)$. Its definition depends on the choice of the generator $\alpha$, which we fix in advance. For each $t\in F$, let us choose the element $g_t$ as the translation map $x\mapsto x+t$, both for $X$ and $Y$. We wish to check orientability.\\

     The orbit $O_0$ containing $(0,0)$ is nonorientable because $\psi_X$ and $\psi_Y$ do not agree on $H_X=H_X\cap H_Y$. On the other hand, the orbit $O_1$ is orientable, and it suffices to check this at the point $(x,y)=(0,1)$. The stabilizer of $1$ is the group $H_Y'=\{x\mapsto 1+a(x-1)\}=g_1^{-1}H_Yg_1,$ and the intersection $g_0^{-1}H_Xg_0\cap g_1^{-1}H_Yg_1= H_X\cap H_Y'=\{1\}$. Thus both homomorphisms in Proposition \ref{nc}, equation \eqref{310} agree, and $O_1$ is orientable.\\
     
     Let us construct an $H^1$-developed matrix $A$. We put $A_{0,0}=0$ and $A_{0,1}=1$. In step 3 of the subroutine MONCOV, we use formula \eqref{3.9} to compute the $1$-cocycle $z(g)$. Take $g\in G$, $g(x)=ax+b$. We compute simultaneously for $X$ and $Y$:
     
      $$g_t^{-1}g(x)=g_t^{-1}(ax+b)=ax+b-t.$$ Now $\overline{g_t^{-1}g}$ should equal to $g_s^{-1}$ for $s$ satisfying $g_t^{-1}g(s)=x_0=0$. Solving the equation $as+b-t=0$, we obtain $s=-(b-t)/a$. Hence
      \begin{gather*}
     \implies \overline{g_t^{-1}g}(x)=g_s^{-1}(x)=x+(b-t)/a\\
     \implies g_t^{-1}g\overline{g_t^{-1}g}^{-1}(x)=g_t^{-1}g(x-(b-t)/a)=ax.\\
     \implies \text{ (by eq. \eqref{3.9}), }\ \  z(g)=\left(\left(\frac{a}{F}\right)_n\sum_t[t],\sum_t[t] \right)
     \end{gather*}
     Now, for a point $(s,t)\in F^2$ satisfying $s\neq t$, take the group element $g(x)=(t-s)x+s$ which maps $(0,1)$ to $(s,t)$. Then for $(P,Q)=MONCOV(g)$ we have $\sign(s,P)=\left(\frac{t-s}{F}\right)_n$ and $\sign(t,Q)=1$. According to step 17 of the algorithm, the value of our desired matrix $A$ at the $(s,t)$-position is
     
     \be \label{PayKer} A_{s,t}=z(g)_1(s)\cdot z(g)_2(t)^*\cdot A_{0,1}=\left(\frac{t-s}{F}\right)_n.\ee
     For $s=t$ we set $A_{s,t}=0$.\\

	 This yields the Paley's Conference Kernel matrix. The reader should compare this with the bilinear form matrix constructed above in \S\ref{BilMat}. This is essentially the same as the submatrix of $M$ there corresponding to the affine line. There are few differences though. The diagonal here is non orientable, while there it was orientable. Note also that the acting group is different here, as well as the action here which is different for the $X$ and $Y$ axes.\\
	 
	 Let us now check the Gram matrix $AA^*$.
	 \begin{lem}
	 	We have $$AA^*=A^*A=qI-J.$$
	 \end{lem} 
 
	 \begin{proof}   
	     The group $G$ acts 2-transitively on the sets of rows and columns, hence acts transitively on pairs of row-column with different index. Hence the Gram matrices $A^*A$ and $AA^*$ are constant up to sign at the off-diagonal entries. Also the action of $G$ on columnsis by unsigned permutation matrices $|Q|=|Q|(g)$ for $g\in G$, hence $A^*A$ is invariant under the conjugation by $|Q|(g)$ for all $g$. By the 2-transitivity of the action, we deduce that $A^*A$ is constant off the diagonal. Let this constant value be equal to $c$. Write $\mathbf j=[1,1,\ldots,1]\in \RR^q$. Setting up $\left(\tfrac{0}{F}\right)_n=0$, 
	     $$(A\mathbf{j}^T)_i=\sum_{a\in F} \left(\frac{a-i}{F}\right)_n=\sum_{a\in F} \left(\frac{a}{F}\right)_n=0.$$
	      We compute the sum of entries of $A^*A$.
	     $$q(q-1)+cq(q-1)=\mathbf jA^*A\mathbf j^T=0,$$
	     hence $c=-1$, proving that $A^*A=qI-J$. Since $A^T=\left(\tfrac{-1}{F}\right)_nA$, then $B=A^T$ satisfies the same identity $B^*B=qI-J$. By taking complex conjugates, we get $AA^*=qI-J$.
	     
	\end{proof}

	\begin{cor}\label{c42}
		For $q$ a prime power and $n|(q-1)$, the matrix $$\begin{bmatrix}
		A & \mathbf j^T \\
		\mathbf j & 0
		\end{bmatrix}$$ 
		is a $GW(q+1,q;n).$
		This is the well-known Paley Conference  matrix.
	\end{cor}

    \subsection{Paley's Hadamard matrices}.
    Sometimes we can restore orientability of some $G$-orbit, if we reduce to a subgroup $G_0\subset G$. Hopefully $G_0$ is not too small as compared to $G$. This is the case with the above affine group where we take
    $$G_0=\{x\mapsto ax+b 
     \in G\ | \ \left(\frac{a}{F}\right)_n=1\}.$$
     We have $[G:G_0]=n$, and the orbit $O_1$ splits into $n$ $G_0$-sub orbits. However, the orbit $O_0$ becomes orientable, since our two homomorphisms restricted to $H_X\cap G_0$ and $H_Y\cap G_0$ become trivial. Thus $B=I+A$ is $G_0$ invariant. This is the \emph{Paley Hadamard Kernel matrix}. Suppose now that $n=2$. If $q\equiv 3 \mod 4$, then $A$ is antisymmetric since  $A^T=\left(\frac{-1}{F}\right)_n{A}=-{A}$. Thus $BB^*=I+AA^*=(q+1)I-J$, and $B\mathbf{j}^T=\mathbf{j}^T$. The following matrix 
     $$HP_I=\begin{bmatrix}
     B & \mathbf j^T \\
     \mathbf j & -1
     \end{bmatrix}$$
     is now seen to be Hadamard and is the well-known Paley-Type I Hadamard matrix. This matrix differs from the matrix in Example \S\ref{BilMat} by the negation of the last row.\\
	 
	 The group $G_0$, acts on $(q+1) \times (q+1)$ matrices by embedding $s(G_0)$ as a block in $GL(\mathbb C,q+1)\times GL(\mathbb C,q+1)$, making $HP_I$ invariant under this action. If $q\equiv 1\mod 4$, $HP_I$ is symmetric and $BB^*=(q+1)I-J+B+B^*$. We fix the situation by letting $B'=-I+A$, and define $HP_I'$ similarly with $B'$ instead of $B$. Now it can be checked that $HP_IHP_I^T+(HP'_I)(HP'_I)^T=(2q+2)I$ and that $(HP_I)(HP'_I)^T$ is symmetric. These imply that
     $$HP_{II}=\begin{bmatrix}
     HP_I & HP_I'\\
     -HP_I' & HP_I
     \end{bmatrix}.$$
    is a Hadamard matrix of size $2q+2$ and $G_0$ acts as an automorphism subgroup on this matrix. The full automorphism groups of the Paley matrices of both types were computed in \cite{deLauney:2000}. Notice that for $n>2$ and $q=3\mod 4$ $HP_I$ is no longer Hadamard, but $Re(HP_IHP_I^*)=(q+1)I$. A similar fact holds for $HP_{II}$.
     
     \section{Algebras, Modules and Orientable Orbits}
     
     Cohomology developed matrices admit more algebraic structure. We continue to assume that $X,Y$ are transitive $G$-sets and that $\mathcal O\subset X\times Y$ is irreducible. In this section we discuss Cohomology-Developement in general, and do not restrict ourselves to $H^1$-development. In particular $G$ and $\widetilde{G}$ need not be isomorphic.\\
     
     Let a monomial cover $\pi_X\times \pi_Y:\widetilde{G}\to Mon(X,\mu_n)\times Mon(Y,\mu_n)$  be given. For convenience we shall write $\pi_X(g)=P(g)$ and $\pi_Y(g)=Q(g)$, $P,Q$ are homomorphisms. As usual we let $\widetilde{G}$ act on matrices by $g,A\mapsto P(g)AQ(g)^*$. We shall call this the \emph{the $(P,Q)$-action}. We extend the $(P,Q)$-action to the space $\CC(X,Y)$ of all \emph{Complex} $X\times Y$-matrices.\\
     
     The following result is an easy consequence and can be immediately verified.
     
     \begin{thm}\label{alg1}
     	Let $(P,Q)$ be a monomial cover of $G$.
     	\begin{itemize}
     		\item[(a)] The subspace $\mathcal A_P(X)\subset \CC(X,X)$ of $G$-invariant matrices under the $(P,P)$-action is a matrix algebra, containing the identity and closed under the conjugate-transpose.
     		\item[(b)] The subspace $\mathcal A_{P,Q}(X,Y)\subset \CC(X,Y)$ of $G$-invariant matrices under the $(P,Q)$-action is a left module over the algebra $\mathcal A_P(X)$, and a right module over the algebra $\mathcal A_Q(Y)$, with respect to matrix multiplication.
     		\item[(c)] We have 
     		$$\mathcal A_{P,Q}(X,Y)\mathcal A_{Q,R}(Y,Z)\subseteq A_{P,R}(X,Z).$$
     		The product of spaces here means the $\CC$-space generated by all products.
     		\item[(d)] Closure under the conjugate-transpose:
     		$$\mathcal A_{P,Q}(X,Y)^*=\mathcal A_{Q,P}(Y,X).$$
     		
     	\end{itemize}
     \end{thm}
 
     As a corollary we have the following important result:
     \begin{thm}\label{orth}  Let $O_1,\ldots, O_r$ be the set of all orientable orbits w.r.t. the $(P,Q)$-action. 
     	\begin{itemize}
     		\item[(a)] $\mathcal A_{P,Q}(X,Y)$ has a basis $\{B_i\}$ of $\mu_n^+$-matrices, such that $supp(B_i)=O_i$, $1\le i\le r$. In particular the dimension over $\CC$ of $\mathcal A_{P,Q}(X,Y)$ equals to the number of orientable orbits.
     		\item[(b)] Suppose that $|X|= |Y|$ and that $\mathcal A_{P,Q}(X,Y)$ contains an \emph{invertible} matrix. Then
     		$$\dim_{\CC} \mathcal A_{P,Q}(X,Y)=\dim_{\CC} \mathcal A_{P}(X).$$
     		In particular both the $(P,Q)$ and the $(P,P)$ actions have the same number of orientable orbits.
     		\item[(c)] If the $(P,P)$-action has a single orientable orbit, then every matrix $A\in \mathcal A_{P,Q}(X,Y)$ is \emph{orthogonal}, i.e. satisfies $AA^*=cI$.
     	\end{itemize}
     \end{thm}
     Theorem \ref{orth}(c) is the source of many weighing matrices. The orthogonality of Projective-Spaces, Grassmanians and Flag-Variety weighing matrices discussed in \S 6 is a consequence of this principal. Another application of \emph{formal orthogonal pairs} appears in \cite{GK2020}.
     
     \begin{defn}
     	The basis $\{B_i\}$ is called an \emph{orbital basis}. It is unique up to multiplication by scalars.
     \end{defn}

	 The reader might want to compare Theorem \ref{orth} to the discussion in Higman \cite[p. 415]{Higman1987}. He discusses there the notion of \emph{weights}, whose collection is a matrix algebra spanned by an orbital basis, and their entrywise absolute values are members of the Bose-Mesner algebra of an association scheme. CDMs with a $(P,P)$-action are a special case (Called  the 'group case'). Higman does not speak about $(P,Q)$ actions and modules, and does not use cohomlogy.\\

     \begin{proof}
     	(a) Let $A\in \mathcal A_{P,Q}(X,Y)$. For every $G$-orbits $O\subseteq X\times Y$ the $(P,Q)$ action determines the values of $A_{x,y}$ for all $(x,y)\in O$ from a single entry $(x',y')\in O$ (named above an `orbit head'). For every orientable orbit $O_i$ choose a basepoint $(x_i,y_i)\in O_i$, and let $B_i$ be the unique $(P,Q)$-invariant matrix with $(B_i)_{x_i,y_i}=1$, and $(B_i)_{x_j,y_j}=0$ for $j\neq i$. Then $\supp B_i=O_i$ and $\{B_i\}$ are linearly independent. Every $A\in \mathcal A_{P,Q}(X,Y)$ can be written uniquely as $A=\sum A_{x_i,y_i} B_i$, since both sides agree on all the basepoints. Thus $\{B_i\}$ is a basis.\\
     	
     	(b) If $C\in \mathcal A_{P,Q}(X,Y)$ is of full rank, then right multiplication by $C^*$ gives a injection $ \mathcal A_{P,Q}(X,Y) \hookrightarrow \mathcal A_P(X)$. We can have an arrow in the reverse direction by right multiplying by $C$. Thus both spaces have the same dimension and the same number of orientable orbits.\\
     	
     	(c) The condition says that $\mathcal A_P(X)$ is $1$-dimensional. Since it contains the identity matrix, then $\mathcal A_P(X)$ is the algebra of scalar matrices, and by Theorem \ref{alg1}(c) $AA^*=cI\in \mathcal A_P(X)$ for all $A\in \mathcal A_{P,Q}(X,Y)$.
     \end{proof}
 
     \begin{ex}
     	We continue Examples \ref{159}--\ref{159color}. We work with $X=Y$ and $\psi_X=\psi_Y$, hence the monomial cover has $P=Q$, and $\mathcal A_P(X)$ is a two dimensional algebra, with basis $\{B_1=I,B_2\}$, with $\supp B_2=O_2$. Now, $B_2$ must be either symmetric or anti-symmetric, since $B_2^*=B_2^T\in \mathcal A_P(X)$ and since necessarily $B_2^T=aB_2+bI$ for some $a$ and $b$. The diagonal of $B_2$ is zero, hence $b=0$, and $a=\pm 1$ since transposition is an involution.\\
     	
     	Let us show that $B_2$ is symmetric. $B_2$ satisfies a quadratic equation, $B_2^2-\alpha B_2+\beta I=0.$ If $B_2$ was anti-symmetric, then $B_2^2$ was symmetric and by comparing off-diagonal entries we would conclude $\alpha=0$. Since $B_2$ has weight $8$, then $\beta=-8$, and $B_2$ would be a $W(15,8)$. Non square weights are impossible for odd orders, and we conclude that $B_2$ is symmetric.\\
     	
     	The minimal polynomial of $B_2$, $X^2-\alpha X-8$ is quadratic, and the characteristic polynomial is of degree $15$ with the same roots. Therefore the roots of $X^2-\alpha X-8$ are rational integers. Thus the roots (up to ordering and common sign) are either $(8,-1)$ or $(4,-2)$. Let $p,q$ be the multiplicity of each root in the characteristic polynomial. If $(8,-1)$ are the roots, then it must be satisfied that $8p-q=0$ ($trace(B_2)=0$) and $p+q=15$ (the degree). This is impossible over the integers. A different argument is that by Perron-Frobenius theorem, the spectral radius of $B_2$ is strictly smaller than the spectral radius of $|B_2|$ which is $8$. Therefore $\pm 8$ cannot be an eigenvalue of $B_2$ (We are grateful to an annonymus referee for pointing this out).
		On the other hand for the pair $(4,-2)$ we need $4p-2q=0$, which can be solved with $p=5$ and $q=10$. We conclude that $\alpha=\pm 2$. In particular, one of the matrices, $I+B_2$ and $I-B_2$ has eigenvalues $(3,-3)$. Then this matrix is a symmetric weighing matrix $W(15,9)$, with automorphism subgroup $A_6$.
     \end{ex}
 
     \begin{ex}
     	In the case of Paley's Kernel matrix (\S \ref{s41}), the action was a $(P,Q)$-action with $P\neq Q$. It can be checked, by checking orientable orbits that $\mathcal A_P(X)$ and $\mathcal A_Q(Y)$ are $2$-dimensional, in contrast to  $\mathcal A_{P,Q}(X,Y)$ which is only one dimensional (we have actually showed that the diagonal orbit is non-orientable). This must imply (Theorem \ref{orth}(b)) that the Paley kernel matrix $A\in \mathcal A_{P,Q}(X,Y)$ is non-invertible. Indeed, $AA^*=qI-J$ has $\mathbf{j}$ in its (left) kernel. 
     \end{ex}
 
\begin{rem}\
     	\begin{itemize}
     		\item[(i)] If the $(P,P)$-action of $\widetilde{G}$ is \emph{unsigned}, that is $P(g)$ are permutation matrices, then $\mathcal A_P(X)$ is the Bose-Mesner Algebra (or the Adjacency Algebra)  of the (non-commutative) association-scheme generated by the orbits of $G$ on $X\times X$. In the language pf Association-schemes this is the \emph{Schurian Case}. See \cite{bailey2004association} for background. The algebras $\mathcal A_P(X)$ in the unsigned case are also known by the name \emph{Hecke algebras}.
     		\item[(ii)] For general $(P,P)$-action, the algebra $\mathcal A_P(X)$ is a Weighted Coherent Algebra in the sense of \cite{Sankey2014}. Note that the treatment here is more general because we allow non-orientable orbits.
     		\item[(iii)] The $(P,Q)$-action motivates the definition of \emph{Association-(Bi)Modules} and \emph{Weighted-Coherent-(Bi)Modules} on $\mathcal A_{P,Q}(X,Y)$, over the algebras $\mathcal A_P(X)$ and $\mathcal A_Q(Y)$. We were unable to find in the literature these natural extensions.
     	\end{itemize}
     	 
     \end{rem}
 
     \subsection{Behavior under Hadamard Multiplication}
     We have seen above in Proposition \ref{prop:214} that $\mathbf h(G,\mathcal O)$ is closed under the Hadamard multiplication. This fact can be adapted easily to the context of modules $\mathcal A_{P,Q}(X,Y)$ and algebras $\mathcal A_P(X)$. For complex matrix spaces $M,M'$, let $M\circ M'$ be the matrix vector-space generated by the Hadamard products $m\circ m'$ for all $m\in M$ and $m'\in M'$.
     \begin{thm} Let $(P,Q)$ and $(P',Q')$ be two monomial covers of the action of a group $G$ on $\mathcal O$-matrices.
     	For a suitable monomial cover $(P'',Q'')$ we have:
     	\begin{align*}
     	    \mathcal A_{P}(X)\circ \mathcal A_{P'}(X)&\subseteq \mathcal A_{P''}(X), \text{ and}\\
     	    \mathcal A_{P,Q}(X,Y)\circ \mathcal A_{P',Q'}(X,Y)& \subseteq \mathcal A_{P'',Q''}(X,Y).
     	\end{align*}
	
     \end{thm}
 
     \begin{proof}
		The main problem here is that $(P,Q)$ is a monomial cover with a given cover group $\rho: \widetilde{G}\to G$ and $(P',Q')$ is a monomial cover with another cover group $\rho':\widetilde{G}'\to G$. We need to find a third cover group $\rho'':\widetilde{G}''\to G$ which is a cover of both $\widetilde{G}$ and $\widetilde{G}'$. We define the fibered product 
		$$\widetilde{G}'':=\widetilde{G'}\times_G\widetilde{G} \ := \ \{(g',g)\in \widetilde{G}'\times \widetilde{G}\ | \ \rho'(g')=\rho(g)\}$$ to which $P,Q,P',Q'$ can be extended via the projections. It is routine to check Definition \eqref{def:mon_cov} that $(P,Q)$ and $(P',Q')$ are monomial covers with respect to the new group and that the algebras/modules they define are the same. This reduces us to the case that the covering group $\widetilde{G}$ is the same for all monomial covers involved.\\

     	Note that $|P|=|P'|$ and $|Q|=|Q'|$ are the permutation matrices for the action of $G$ on $X$ and $Y$.
     	We write $P=L|P|$, $Q=R|Q|$, $P'=L'|P|$ and $Q'=R'|Q|$ for diagonal $\mu_n$-matrices $L,L':\widetilde{G}\to D_X$ and $R,R':\widetilde{G}\to D_Y$. Then matrices in $ A_{P,Q}(X,Y)\circ \mathcal A_{P',Q'}(X,Y)$ are invariant under the monomial pair $(P'',Q'')$ defined by $P''=LL'|P|=P\circ P'$ and $Q''=RR'|Q|=Q\circ Q'$. 
     \end{proof}
     
     \subsection{Quasiproducts}\label{orb_mor}
     
      As an application of the constructions of algebras and modules, we introduce \emph{Quasiproducts}. They can be thought of as a `twisted' version of the Kronecker product, and are defined under a group-theoretic situation that we now describe. The idea is that under certain circumstances the quasiproduct of weighing matrices is again a weighing matrix. In this following subsection we will restrict the discussion to $H^1$-developments.\\

    Suppose we are given a short exact sequence of groups 
    \be\label{nonsplit} 1\to Z \to G \to PG\to 1,\ee and $Z\subseteq Z(G)$ is in the center of $G$. This extension is not necessarily split. As usual we are given transitive $G$-sets $X,Y$ with basepoints $x_0,y_0$, stabilizers  $H_X,H_Y$, and homomorphisms $\psi_X$ and $\psi_Y$. We assume that the subgroup $Z$ is contained in $H_X,H_Y$, and so the sequence \eqref{nonsplit} restricts to the sequences
    \begin{align}
    & \label{HX} 1 \to Z \to H_X \to PH_X\to 1\\
    & \label{HY} 1 \to Z \to H_Y \to PH_Y\to 1.
    \end{align} We shall assume that the two restricted sequences are \emph{split}. In particular there are subgroups $H_X',H_Y'$ such that $H_X=ZH_X'\cong Z\times H_X'$ and $H_Y=ZH_Y'\cong Z\times H_Y'$. Throughout \S \ref{orb_mor} we will identify $X=G/H_X$ and $Y=G/H_Y$ as $G$-sets. Define $X'=G/H_X'$ and $Y'=G/H_Y'$. There are natural maps $X'\to X$, $Y'\to Y$ and $X'\times Y'\to X\times Y$. Given a full set $\{g_x\}$ of coset representatives for $G/H_X$, the set $\{g_x z\}_{x\in X,z\in Z}$ is a full set of coset representatives of $G/H_{X'}$, and similarly for $Y$. Any $G$-orbit in $X\times Y$ can be generated from a point $(H_X,yH_Y)$ for some $y\in G$.
    
    \begin{lem}\label{lem:Z/O}
    	Every $G$-orbit $O'\subseteq X'\times Y'$ maps onto some $G$-orbit in $O\subseteq X\times Y$, and $|O'|=|Z||O|$. There are $|Z|$ $G$-orbits in $X'\times Y'$ above any $G$-orbit in $X\times Y$. More explicitly, if $(H_X,yH_Y)$ generates an orbit $O\subseteq X\times Y$, then $\{(H_{X'},yzH_{Y'})\}_{z\in Z}$ are generators of the $|Z|$ distinct orbits above $O$.
    \end{lem} 

    \begin{proof}
    	Let $O'$ be the orbit generated by $(H_{X'},yH_{Y'})$. Let $O$ be the orbit below $O'$, that is $O$ is generated by $(H_X,yH_Y)$. Clearly the natural map $O'\to O$ is surjective. The stabilizers at the generating points at $O$ and $O'$ are $S:=H_X\cap yH_Yy^{-1}$ and $S':=H'_{X}\cap yH'_{Y}y^{-1}$ respectively. We have $S'\subseteq S$ and $S=ZS'\cong Z\times S'$. This proves that $|O'|=|Z||O|$.\\
    	
        If $(H'_{X},y_1H'_{Y})\in O_1'$ and  $(H'_{X},y_2H'_{Y})\in O_2'$ are generators of two $G$-orbits $O_1',O_2'$ above the same orbit in $O\subset X\times Y$, then there exist $g\in G$ and $z,z'\in Z$ such that $gH'_{X}=zH'_X$ and $gy_2H_{Y'}=y_1z'H_{Y'}$. This shows that $(zH_X',y_1z'H_Y')\in O_2'$, and hence $(H_X',y_1z'z^{-1}H_Y')\in O_2'$. Thus $\{(H'_{X},yzH'_{Y})\}_{z\in Z}$ generate all orbits above the orbit of $(H_X,yH_Y)$.\\
        
         It remains to show that the orbits $O'_z$ above $O$, generated by $(H'_X,yzH'_Y)$ for $z\in Z$ are distinct. Since this is full list of all orbits above $O$, by the first part the lemma, the cardinality of $\bigcup_z O'_z$ is less than or equal to $|Z|^2\cdot |O|$. This holds for all orbits $O\subset X\times Y$, hence $|X'\times Y'|\le |Z|^2 \cdot |X\times Y|$. But there is an equality there, which implies that  $|\bigcup_z O'_z|=|Z|^2\cdot |O|$ for every $O$. Thus the union must be  disjoint, and the orbits are distinct.     
    \end{proof}

    For every orbit $O\subseteq X\times Y$ we fix (non canonically) a generator $(H_X,g_OH_Y)$, and let $O_z\subseteq X'\times Y'$ be the orbit above $O$ generated by $(H'_X,g_OzH'_Y)$.  
    
    \begin{lem}\label{lem:orient5} Suppose that $\psi_X|_Z=\psi_Y|_Z=\psi$. 
    	Then an orbit $O$ is orientable, if and only if (some orbit) (all orbits) $O_z$ above $O$ are orientable.
    \end{lem}

    \begin{proof}
    	According to Proposition \ref{nc}, the orbit $O=G(H_X,g_OH_Y)$ is orientable iff $\psi_X(g)=\psi_Y(g_O^{-1}gg_O)$ for all $g\in S:=H_X\cap g_O^{-1}Hg_O$. Similarly $O_z$ is orientable iff $\psi_Y(g)=\psi_Y((g_Oz)^{-1}gg_Oz)$ for all $g\in S'=H'_X\cap (g_{O}z)^{-1}H'_Y(g_Oz)=H'_X\cap g_{O}^{-1}H'_Yg_O.$ This last condition is equivalent to $\psi_X(g)=\psi_Y(g_O^{-1}gg_O)$ for all $g\in S'$, since $Z\le H_Y$ and $\psi_Y$ is a homomorphism defined on $H_Y$. By our assumption that $\psi_X|_Z=\psi_Y|_Z$ and since $S=S'Z$ this equality extents	to all $g\in S$. Thus the orientability of $O_z$ for some $z$ is equivalent to the orientability of $O$.
    \end{proof}

    For every orbit $O\subset X\times Y$ we will choose a basepoint $(x_O,y_O)=(H_X,g_OH_Y)$, such that $g_O=g_y$ for $y=y_O$.
    Let $O_1,\ldots,O_r\subseteq X\times Y$ be the list of orientable orbits w.r.t. the homomorphisms $\psi_X,\psi_Y$. We assume that $\psi_X|_Z=\psi_Y|_Z$. Then $O_{i,z}$, $1\le i \le r$ and $z\in Z$ are the orientable orbits of $X'\times Y'$. Let $\{B_i\}$, $1\le i \le r$ be the orbital basis for the monomial action defined by $\psi_X,\psi_Y$ and the coset systems $\{g_x\},\{g_y\}$, normalized so that $(B_i)_{(H_X,g_{O_i}H_Y)}=1$.\\
    
     Similarly let $B_{i,z}$, $1\le i \le r$ and $z\in Z$ be the
     orbital basis for the monomial action defined by $\psi_X,\psi_Y$ and the coset systems $\{g_xz\},\{g_yz\}$, normalized so that $(B_{i,z})_{(H'_X,g_{O_{i}}zH'_Y)}=1$. Our goal is to find a relationship between the monomial bases $\{B_i\}$ and $\{B_{i,z}\}$.\\
     
     Let us figure out the value of $(B_i)_{x,y}$ where $(x,y)=g(H_X,g_{O_i}H_Y)$. The value is given by the cocycle values in \eqref{3.9}:
     $$(B_i)_{x,y}=\psi_X(g_x^{-1}g (\overline{(g_x^{-1}g)})^{-1}) \cdot \psi_Y(g_y^{-1}g (\overline{(g_y^{-1}g)})^{-1})^*.$$
     We have $gg_{O_i}H_Y=g_yH_Y$, so $g_y^{-1}gg_{O_i}\in H_Y$. But since $g_{O_i}$ has been chosen to be a coset representative, then $g_y^{-1}g (\overline{(g_y^{-1}g)})^{-1}=g_y^{-1}gg_{O_i}$. The analogous formula for $X$ is just $g_x^{-1}g$. Thus we obtain the simpler formula,
     $$(B_i)_{x,y}=\psi_X(g_x^{-1}g)\psi_Y(g_y^{-1}gg_{O_i})^*.$$
     
     Now, let $(x',y')=g(H'_X,g_{O_i}zH'_Y)=(z_1g_xH'_X,z_2g_yH'_Y)$ be a point in $O_{i,z}$ above $(x,y)$ (computed with the same $g$). Similar considerations lead to an analogous formula: 
     $$ (B_{i,z})_{x',y'}=\psi_X((z_1g_x)^{-1}g)\psi_Y((z_2g_y)^{-1}gg_{O_i}z)^*=(B_i)_{x,y} \psi_X(z_1)^*\psi_Y(z_2)\psi_Y(z)^*,$$
     and equivalently \be \label{bizx'y'} \forall (x',y')\in O_{i,z}, \ \  (B_{i,z})_{x',y'}=(B_i)_{x,y}\psi((zz_1)^{-1}z_2).\ee
     
     For convenience, let us summarize now the axioms we assume for the setup of \emph{Quasiporducts}.
     
     \begin{itemize}
     	\item[A1.] We are given groups $Z,G,PG$ sitting in an exact sequence \eqref{nonsplit}, $Z\subseteq Z(G)$ is in the center.
     	\item[A2.] $H_X,H_Y\le G$ are subgroups containing $Z$. $\psi_X:H_X\to \mu_n$ and $\psi_Y:H_Y\to \mu_n$ are homomorphisms.
     	\item[A3.] The restricted sequences \eqref{HX}-\eqref{HY} are split, with sections $s_X:PH_X\to H_X$ and $s_Y:PH_Y\to H_Y$. Let $H'_X=s_X(PH_X)$ and $H'_Y=s_Y(PH_Y)$. 
     	\item[A4.] $\psi_X|_Z=\psi_Y|_Z=\psi$.
     	\item[A5.] We choose coset representatives $\{g_x\}$ for $x\in X=G/H_X$ and $\{g_xz\}$, $x\in X$ and $z\in Z$ for $G/H_X'$. Similarly for $Y$. This data is enough to generate unique ($H^1$-) monomial covers $G\to (P,Q)\in  Mon(X,\mu_n)\times Mon(Y,\mu_n)$, and  $G\to (P',Q')\in  Mon(X',\mu_n)\times Mon(Y',\mu_n)$.
     	\item[A6.] For every $G$-orbit in $X\times Y$, we pick up elements $g_O\in G$ s.t. $(H_X,g_OH_Y)\in O$, and $g_O$ is a coset representative for $H_Y$. Let $\{B_i\}$ be the orbital basis of $\mathcal A_{P,Q}(X,Y)$, characterized by the normalization $B_{(H_X,g_OH_Y)}=1$. Likewise, let $\{ B_{i,z}\}$ be the orbital basis of $\mathcal A_{P',Q'}(X',Y')$, characterized by the normalization $B_{(H'_X,g_OzH'_Y)}=1$.
     \end{itemize}
 
     \begin{defn}\label{def:quasip}
     	 Let $A\in \mathcal A_{P,Q}(X,Y)$ and let $T$ be a $Z$-developed matrix (in the sense of group development), identified  as an element of $\mathbb C[Z]$. Write $A=\sum a_iB_i$ and $T=\sum t_z[z]$ for the orbital basis decompositions. Then the \emph{quasiproduct} of $A$ and $T$ is defined as 
     	$$ A \boxtimes T= \sum a_it_z B_{i,z}.$$ 
     \end{defn}
 
     \begin{ex}
     	Consider the case $G=Z\times PG$, and take $\psi_X,\psi_Y$ to be characters on $PH_X$ and $PH_Y$, extended to $H_X=Z\times PH_X$ and $H_Y=Z\times PH_Y$ by the projections to $PH_X$ and $PH_Y$. By appropriately indexing the matrices, it is easy to show that $B_{i,z}=B_i\tensor P_z$ (Kronecker Product), where $P_z$ is the permutation $Z$-developed matrix defined by $(P_z)_{z_1,z_2}=1$ if $z_2z_1^{-1}=z$, and $0$ otherwise. In particular, the quasiproduct $A\boxtimes T=A\tensor T$, the Kronecker product.
     \end{ex}

     We are interested in cases where the quasiproduct $A\boxtimes T$ is a weighing matrix. Suppose that $A\in M\mu_n(X,Y)$ is invariant under the $(P,Q)$-action, and that the $(P,P)$-action has a single orientable orbit. Then by Theorem \ref{orth}(c), $A$ is a weighing matrix. We have
     
     \begin{thm}\label{quasi2}
     	Suppose that the $(P,P)$-action has a single orientable orbit in $X\times X$ (necessarily the diagonal). Suppose that $A\in M\mu_n(X,Y)$ is an invertible $(P,Q)$-invariant matrix. Then the quasiproduct $A\boxtimes T$ is a weighing matrix, if and only if $T$ is a weighing matrix.
     \end{thm}
 
     \begin{proof} Let $(P',Q')$ be the action on $X'\times Y'$-matrices above $(P,Q)$. We know by Theorem \ref{orth}(c) that $A$ itself is a weighing matrix, and from part (b) there is a unique $(P,Q)$-orientable orbit, which is the support of $A$. By Lemma \ref{lem:Z/O} there are $|Z|$ $(P',Q')$- orientable orbits above $supp(A)$, and an orbital basis $\{B_z\}_{z\in Z}$ for $\mathcal A_{P',Q'}(X,Y)$. Likewise there is an orbital basis $\{\Delta_z\}_{z\in Z}$ for the $(P',P')$-action, above the main diagonal $\Delta\subseteq X\times X$. We wish to study the matrix product structure of $\mathcal A_{P',Q'}(X,Y)$, or more precisely to write the products $D=B_zB_w^*$ as linear combinations of $\{\Delta_z\}$. It is enough to compute this product at a point $(x_0',v')$, where $x_0'=1\cdot H_X'$. It is clear by Lemma \ref{lem:orient5} that $D_{x_0',v'}=0$ if $v'\notin Zx_0'$, since that point lies above a point away from the main diagonal, and thus is a non-orientable point. 
     So we shall take $v'=\zeta H'_X$ for some $\zeta \in Z$. Using \eqref{bizx'y'}, for such $v'$  we have:
     	
    \begin{gather*}
    D_{x_0',v'}=\sum_{y':\ (x'_0,y')\in supp(B_z)\ \wedge \ (v',y')\in supp(B_w)} (B_{z})_{(x'_0,y')}(B_{w}^*)_{(y',v')}= \\ 
    =\sum_{y\in Y; z_2\in Z} \delta_{y,z_2} \cdot A_{(x_0,y)}\psi(z^{-1}z_2) (A_{(x_0,y)})^*\psi((w\zeta)^{-1}z_2)^* (\text{by \eqref{bizx'y'}})\\
    \text{where } x_0'=1\cdot H'_X;\ y'=g_yz_2H'_Y;\  v'=\zeta H'_X,\\
	=\sum_{y\in Y; z_2\in Z} \delta_{y,z_2}|A_{x_0,y}|^2 \psi(z^{-1}w\zeta),
    \end{gather*}
    and $\delta_{y,z_2}\in \{0,1\}$ equals $1$ if and only if $(x_0',y')=(H'_{X},g_yz_2H'_Y)\in supp(B_z)$ and $(v',y')=(\zeta H'_X,g_yz_2H'_Y)\in supp (B_w)$.\\
    
    Let us compute $\delta_{y,z_2}$. $(x_0',y')\in \supp (B_z)$ iff $z=z_2$, and $(v',y')\in \supp(B_w)$ iff $z_2\zeta^{-1}=w$.
	So throughout the summation we fix $z_2=z$. But then notice that $z^{-1}w\zeta=1\implies \psi(z^{-1}w\zeta)=1$ so we have
	$$D_{x_0',v'}=\sum_{y\in Y} |A_{x_0,y}|^2=\lambda,$$ where $\lambda$ is the weight of the weiging matrix $A$. To summarize,

	$$D_{x_0',v'}=\begin{cases}
		\lambda & v'\in Zx_0'\\
		0 & v'\notin Zx_0'
	\end{cases}.$$

	From this and from $\zeta=zw^{-1}$ we conclude that
    
    \be \label{qpord} B_zB_w^*=\lambda \Delta_{zw^{-1}}.\ee 
    
    Now the quasiproduct $M=A\boxtimes T=\sum_{z}t_zB_z$ and by \eqref{qpord} we have $MM^*=\sum_{z,w}t_z\overline{t_w}\lambda \Delta_{zw^{-1}}$. This matrix is scalar if and only if $TT^*$ is scalar, which proves the theorem.

    \end{proof}

	In the next section we will see examples of quasiproducts which are not kronnecker products.

    \section{More Examples - Projective-Space, Grassmannian and Flag-Variety weighing-matrices}\label{sec:proj}
	The Projective-Space matrices constructed here are not new. They have originally appeared in the papers \cite{berman1977weighing,berman1978families} of G. Berman, constructed from finite geometries. Here we reconstruct and reinterpret these matrices from our cohomology machinary. The orthogonality of the matrices will be a direct consequence of Theorem \ref{orth}(c). The Quasi-Projective-Space, Grassmannian and Flag weighing matrices are presumably new.

	\subsection{Projective-Space and Quasi-Projective-Space weighing matrices.}\label{sec:proj}
	
	The background on finite projective spaces is well-known and standard, see e.g. \cite{Dembowski1968,Baer2005-eb}. For conveniece and to set up notation we briefly sketch the material.\\

	Let $F=\FF_q$ be the finite field of $q$ elements, $q$ a prime power, and $d>0$ an integer. The \emph{Projective-Space} of dimension $d$ over $F$ is
	$$\PP^d(F)\ := \ \left\{\text{Lines $\subset F^{d+1}$ through the origin }   \right\}=\left\{ v\in F^{d+1} \ | v\neq 0 \right\}/\sim$$
	with the relation $\sim$ defined by $v\sim w$ iff $v=\lambda w$ for some $\lambda\in F^\times$.
	We write points in $\PP^d(F)$ as $[v_0:v_1:\ldots :v_d]=[v_0,v_1,\cdots,v_d]/\sim$. It is well-known that $\#\PP^d(F)=(q^{d+1}-1)/(q-1).$\\
	
	The set of \emph{Projective Hyperplanes} is
	$$\LL^d(F)\ := \ \left\{\text{Linear subspaces }V\subset F^{d+1} \ | \dim V=d  \right\}.$$
	Fix a nondegenerate bilinear form $\langle\ ,\ \rangle:F^{d+1}\times F^{d+1} \to F$. There is a bijection $\PP^d(F)\longleftrightarrow\LL^d(F)$ by a \emph{Duality map}. Namely, 
	$$ \mathrm{Duality}: \ v/\sim \ \longleftrightarrow\ H=\{x\in F^{d+1} \ |\ \langle v,x\rangle =0 \ \}.$$
	We next define occurrence relations. A pair $(v/\sim,H)$ is \emph{occurring} if $v\in H$, and otherwise it is \emph{nonoccurring}.
	The linear group $GL_{d+1}(F)$ acts on both $\PP^d(F)$ and $\LL^d(F)$, via the natural linear action on $F^{d+1}$, and preserves occurrence relations.\\
	
	To construct our Cohomology-Developed matrix, we choose $G,X,Y$ as follows:
	Let $G=GL_{d+1}(F)$, $X=\PP^d(F)$ and $Y=\LL^d(F)$. The action of $G$ on $X$ and $Y$ is 2-transitive, and in addition it breaks $X\times Y$ into two orbits: The orbit $O_{oc}$ of all occurring pairs $(v/\sim,H)$, and the orbit $O_{noc}$ of all nonoccurring pairs. For the duality map we choose to work with the standard bilinear form 
	$$\langle v,w \rangle=\sum_i v_iw_i.$$
	Then the duality map identifies $X=Y$ merely as sets, not as $G$-sets. As a $G$-set, a matrix $g\in G=GL_{d+1}(F)$ acts on $y\in Y=X$ by $y\mapsto (g^{-1})^Ty$. Therefore we shall identify $X$ and $Y$ as sets, but remember the different $G$-actions.\\
	
	We will construct a generalized weighing matrix $A\in GW(\tfrac{q^{d+1}-1}{q-1},q^d;n)$ for $n|(q-1)$ with an automorphism subgroup $G$. That is, we shall construct an monomial lift $G\to Mon(X,\mu_n)\times Mon(Y,\mu_n)$, giving rise to a weighing matrix.\\ 
	
	  Let us begin by determining the stabilizing groups $H_X$ and $H_Y$. Pick the basepoint $x_0=[1:0:0\cdots :0]\in X=\PP^d(F)$ and let $y_0=[1:0:0\cdots :0]\in Y$. The groups stabilizing $x_0$ and $y_0$ are
	\begin{align}
		 H_X\ &= \ \left\{ \begin{bmatrix}
		a & B\\ 0 &D 
		\end{bmatrix}\in GL_{d+1}(F)\ \left|\ a \text{ is } 1\times 1 \right.    \right\},\text{ and }\\
		 H_Y=H_X^T\ &= \ \left\{ \begin{bmatrix}
		a & 0\\ C &D 
		\end{bmatrix}\in GL_{d+1}(F)\ \left|\ a \text{ is } 1\times 1 \right.    \right\}
	\end{align}
	Define the homomorphisms $\psi_X$ and $\psi_Y$ as follows:
	
	\begin{align}
	\psi_X\left( \begin{bmatrix}
	a & B\\ 0 &D 
	\end{bmatrix}\right) \ &= \ \left(\frac{a}{F}\right)_n\in \mu_n, \text{ and }\\
	\psi_Y\left( \begin{bmatrix}
	a & 0\\ C &D 
	\end{bmatrix}\right) \ &= \ \left(\frac{a}{F}\right)_n\in \mu_n.
	\end{align}
	
	Next, we need to choose the coset representatives $\{g_x\}$ and $\{g_y\}$. Let $x_1=y_1=[0:1:0\cdots :0]$. We select $g_{x_0}=I=g_{y_0}$, and $$g_{x_1}=g_{y_1}=\begin{bmatrix}
	0 & 1 & 0\\
	1 & 0 & 0\\
	0 & 0 & I_{d-1}
	\end{bmatrix}.$$
	 We do not feel the need to specify explicitly the choices of $g_x,g_y$ for other $x$ or $y$, as this will be inconsequential to the discussion below. It should be commented though that the specific CDM will depend on such choice (but only up to diagonal equivalence).\\
		
	 The pair $(x_0,y_0)$ is nonoccurring, hence we can use it to check the orientability of $O_{noc}$. Clearly, the two homomorphisms agree on $H_X\cap H_Y$, and by Lemma \ref{nc}, $O_{noc}$ is orientable. In contrast, the pair $(x_0,y_1)$ is occurring. So we must ask ourselves if $\psi_X(h)=\psi_Y(g_{y_1}^{-1}hg_{y_1})$ for all $h\in H_X\cap g_{y_1}H_Yg_{y_1}^{-1}$. We have 
	$$H_X\cap g_{y_1}H_Yg_{y_1}^{-1} \ = \ \left\{ \begin{bmatrix} \alpha & \beta &\gamma\\
	                                                     0 &       a    &  0    \\
	                                                     0   &      b_2  &  c \end{bmatrix}\in GL_{d+1}(F) \ \left| \ \alpha,a,\beta \text{ are } 1\times 1\right.  \right\}.$$
    For a general element $h$ in the intersection, $\psi_X(h)=(\alpha/F)_n$, and
    $\psi_Y(g_{y_1}^{-1}gg_{y_1})=(a/F)_n$. These two numbers generally do not agree, hence  $O_{oc}$ is non orientable We see that $\mathcal A_{P,Q}(X,Y)$ is one dimensional, with a generator supported on $O_{noc}$. We have

    \begin{thm}[cf. \cite{JunTon99}, Theorem 2]\label{projj}
    	For every prime power $q$, integers $d>0$ and $1<n|(q-1)$ there exists a generalized weighing matrix 
    	$$A\in GW\left(\frac{q^{d+1}-1}{q-1},q^d;n\right).$$
    	The matrix $|A|$ is the characteristic matrix of the $G$-orbit of all $(x,y)$ such that $x$ is not contained in the hyperplane dual to $y$.
    	Moreover, every automorphism in the image of $GL_{d+1}(F)\to \PermAut(|A|)$ lifts to an automorphism of $A$.
    	
    \end{thm}
	\begin{proof}
		We have seen that $\mathcal A_{P,Q}(X,Y)=\spn (A)$ is one dimensional. Let us check $\mathcal A_P(X)$. The action of $G$ on $X$ is 2-transitive, which means that $X\times X$ breaks into two orbits: the diagonal $O_0$ and the off-diagonal $O_1$. We will see now that $O_1$ is not orientable. Take the pair $(x_0,x_1)\in O_1$. We need to compare the two numbers $\psi_X(h)$ and $\psi_X(g_{x_1}^{-1}hg_{x_1})$ for all $h\in H_X\cap g_{x_1}H_Yg_{x_1}^{-1}$. Similarly to the above,
		$$H_X\cap g_{x_1}H_Xg_{x_1}^{-1} \ = \ \left\{ \begin{bmatrix} \alpha & 0 &\gamma\\
		0 &       a    &  \delta    \\
		0   &      0  &  c \end{bmatrix}\in GL_{d+1}(F) \ \left| \ \alpha,a,\beta \text{ are } 1\times 1\right.  \right\}.$$ For a general $h$ in that group, $\psi_X(h)=(\alpha/F)_n$, and
		$\psi_Y(g_{x_1}^{-1}hg_{x_1})=(a/F)_n$, and $O_1$ is not orientable. We conclude that $\mathcal A_P(X)=\spn(I)$. But by Theorem \ref{orth}(c), $A$ is orthogonal. The claim on automorphisms is clear since $(\psi_X,\psi_Y)$ generate a monomial cover lifting the action of $G$.
		\end{proof}

	\begin{rem} \begin{itemize}
			\item[]
			\item[(a)] In the special case $d=1$, we recover the Paley conference matrix, $W(q+1,q;n)$. The construction in \S \ref{s41} of the Paley kernel is the restriction of $X$ and $Y$ to the affine line $\mathbb A^1(F) :=\PP^1(F)\setminus \{[0:1]\}$, and the margins added in Corollary \ref{c42} are the compactification at $\infty=[0:1]$.
			\item[(b)] Again for $d=1$, if we replace $GL_2(F)$ with the subgroup $SL_2(F)$, then the orientability of the off-diagonal in $X\times X$ is restored, and the same happens for the occurrence orbit (=the diagonal) in $X\times Y$. In fact, the Paley Hadamard type I matrix is $SL_2(F)$-Cohomology-Developed.
		\end{itemize}
		
	\end{rem}

    \begin{rem}
    	The Projective-Space weighing matrix constructed here depends on some choices, such as the coset representatives $g_x$ and $g_y$, and the choice of the residue symbol $\left(\tfrac{\cdot}{F}\right)_n$. If we change representatives, we get $D$-equivalent matrices. However, changing the residue symbol to another, say to  $\left(\tfrac{\cdot}{F}\right)'_n=\left(\tfrac{\cdot}{F}\right)^c_n$, $gcd(c,n)=1$, does not preserve $D$-equivalence class, as the cohomology classes (=homomorphisms) $\psi_X$ and $\psi_Y$ have been changed (see Remark \ref{rem:1} above). The resulting matrix $A'$ is $D$-equivalent to $A^{\circ \, c}$ (Hadamard power).\\
    \end{rem}

    \begin{rem}[cf. Theorem \ref{HadConj}]\label{HadProj}
    	While $A^{\circ \, c}$  is not diagonally equivalent to $A$, it is sometimes Hadamard equivalent. This happens when $F=\mathbb F_q=\mathbb F_{p^r}$, $p$ prime, is not a prime field, and $Gal(F/\mathbb F_p)\cong \ZZ/r$ is not the trivial group. The automorphism group of $\mathbb P^d(F)$ is the group generated by the linear transformations in $PGL_{d+1}(F)$ and by the Galois action on the points in this space. Both kinds of transformations generate  the group of \emph{semilinear transformations}, $G\Gamma L_{d+1}(F)$, which is a semidirect product $GL_{d+1}(F)\rtimes Gal(F/\mathbb F_p)$. This group acts on $X$ and $Y$.
    	Let $\sigma\in G\Gamma_{d+1}(F)$ be the $p$-power Frobenius. If our matrix $A$ is normalized to have $A_{x_0,y_0}=1$, and in addition we choose our coset representatives to satisfy $g_{\sigma x}=(g_x)^\sigma$ and $g_{\sigma y}=(g_y)^\sigma$, then all formulas used in Algorithm \ref{alg} are compatible with Galois, and we obtain $\sigma A=A^{\circ\, p}$. But $\sigma A$ is Hadamard equivalent to $A$ since $\sigma$ acts on $X$ and $Y$ by permutations. It should be noted that, as expected, that condition of Theorem \ref{HadConj} is not satisfied here, because $\xi:=\beta(\sigma)\in \PermAut(|A|)$ normalizes $\beta(GL_{d+1}(F))$, but $\xi\notin \beta(GL_{d+1}(F))$.
    \end{rem}

    \subsubsection{Quasiprojective weighing matrices}\label{sec:quasi}
    An example for the quasiproduct as per \S\ref{orb_mor} can be given by \emph{quasiprojective} matrices. We have the short exact sequence
    
    \be\label{projext} 1\to  F^\times \to GL_{d+1}(F) \to PGL_{d+1}(F)\to 1.\ee
    The injection $F^\times \to GL_{d+1}(F)$ is given by $a\mapsto aI_{d+1}$, and $PGL_{d+1}(F)$ is defined as the quotient. The groups $H_X,H_Y$ are defined as above. Both groups contain the group $Z=F^\times I_{d+1}$ of scalar matrices. The sequence \eqref{projext} splits when restricted to $H_X$ and $H_Y$. Indeed, any element in $PH_X$ has a unique element $\begin{bmatrix} 1 & b\\ 0 & D \end{bmatrix}$ representing it, and we let $H'_X$ be the subgroup of all such matrices. A similar definition works for $Y$. Finally, note that $\psi_X$ and $\psi_Y$ agree on scalar matrices. Thus axioms A1--A6 of quasiproducts are satisfied and for any Projective Space weighing matrix $W$ and any $F^\times$-developed matrix $T$, we may form the quasiproduct $W\boxtimes T$. Since $F^\times$ is cyclic, we have
    \newcommand{\inn}[1]{\langle #1 \rangle}
    \begin{thm}
    	For any Projective Space weighing matrix over a finite field $F$, and for any circulant weighing matrix of size $|F|-1$, the Quasiproduct $W\boxtimes T$ is defined, and is a weighing matrix.
    \end{thm}
    \begin{proof}
    	This is the consequence of Theorem \ref{quasi2}. 
    \end{proof} 

    We call these matrices \emph{quasiprojective} weighing matrices. The quasiproduct $W\boxtimes T$ has the same order and weight as the Kronecker-product $W\tensor T$. But usually, the two matrices are not Hadamard equivalent. R. Craigen \cite{CR1995} has a beautiful construction, called the  \emph{weaving product}, which is a generalization of the Kronecker-product. This construction builds a weighing matrix from a list of smaller weighing matrices. We will show that quasiprojective matrices are not Hadamard equivalent to all Craigen's weaving products of weighing matrices with the same parameters of $W$ and $T$. Therefore quasiproducts give a new construction.\\
    
    Let us first review (the relavant part of) the weaving product. Write the matrices $W=[\mathbf{w}_1,\ldots,\mathbf{w}_r]$ and $T=[\mathbf{t}_1,\ldots,\mathbf{t}_s]^T$, where $\mathbf{w}_i$ and $\mathbf{t}_j$ are column vectors. The Kronecker-product can be described (with the appropriate indexing) as the block matrix $$W\tensor T=(\mathbf{w}_j\mathbf{t}_i^T)_{ 1\le i\le s, \ 1\le j \le r}.$$ The weaving product generalizes this block structure. Let $W_1,\ldots,W_s$ be $W(r,w_1)$ weighing matrices, and $T_1,\ldots,T_r$ be $W(s,w_2)$ weighing matrices. Write $W_i=[\mathbf{w}^i_1,\ldots,\mathbf{w}^i_r]$ and $T_j=[\mathbf{t}^j_1,\ldots,\mathbf{t}^j_s]^T$ by columns. Then the \emph{weaving product} is defined to be the block matrix
    
    $$(W_1,\ldots,W_s)\tensor_w (T_1,\ldots,T_r):=(\mathbf{w}_j^i\mathbf{t}^{jT}_i)_{1\le i\le s,\ 1\le j \le r}.$$
    It is not hard to verify that this is a weighing matrix $W(rs,w_1w_2)$.\\

    We will show by example that $W\boxtimes T$ is not Hadamard equivalent to any weaving products $(W_1,\ldots, W_s)\tensor_w (T_1,\ldots,T_r)$. For our example we take $q=11$, $W_i\in W(12,11)$ and $T_j\in W(10,9)$ be conference matrices over $\mu_2^+$. Let $U$ be the weaving product. Then $U$ and $W\boxtimes T$ are both in $W(120,99)$. Permuting the rows of $W_i$ and the columns of $T_j$ does not change the Hadamard type of $U$. Therefore we may assume that all $W_i$, and all $T_j$ are zero along their main diagonal. In particular $|U|$ is permutationally equivalent to $|W|\tensor |T|$.\\
    
    It is enough to show that $|W\boxtimes T|$ and $|W|\tensor |T|$ are not permutationally equivalent. Nothe that we take here $T$ to be a circulant $W(10,9)$ which is known to exist. Let $Q=(|W|\tensor |T|)\cdot (|W|\tensor |T|)^T$ and $Q'=|W\boxtimes T|\cdot |W\boxtimes T|^T$ be the Gram matrices. If $|W\boxtimes T|$
     and $|W\tensor T|$ were permutationally equivalent, then the sets $V(Q)$ and $V(Q')$ of entries of $Q$ and $Q'$ would be equal. Now, $Q$ is permutationally similar to $(|W|\cdot |W|^T)\tensor (|T|\cdot |T|^T)$. Hence $V(Q)$ is the set of values of the vector $(11,10)\otimes (9,8)$. That is, $V(Q)=\{99,90,88,80\}$.\\
     
     On the other hand, for the quasiproduct we may reinterpret the sets $X',Y'$ as $X'=F^2\setminus\{0,0\}$ with the usual action of $G=GL_2(F)$, and $Y'=F^2\setminus \{0,0\}$ where $g\in GL_2(F)$ acts by multiplication with $g^{-1 T}$. In the matrix $W\boxtimes T$ there are $|F^\times|=10$ orbits $\{O_z\}_{z\in F^\times}$ above the orientable orbit of $W$, $O_z$ consisting of points $(x,y)$ such that the inner product $xy^T=z$. The matrix $T$ is taken to be an $F^\times$-developed matrix, since it is circulant. Since $T$ is a conference matrix, with first row $(t_z)_z$ then one of the $t_z$, namely $t_1$ is zero, and by Definition \ref{def:quasip} the orbits participating in $W\boxtimes T$ are $O_z$ for $z\neq 0,1$. In other words the matrix $|W\boxtimes T|$ is just the characteristic matrix of the set of all pairs $(x,y)$ satisfying $xy^T\neq 0,1$. Take the row $R_1$ of $|W\boxtimes T|$, consisting of all positions $(x,y)$ with $x=(1,0)$. Similarly let the row $R_2$ be the row of all positions $(x,y)$ with $x=(0,1)$. Then $R_1R_2^T$ is the number of $y=(y_1,y_2)\in F^2\setminus\{0\}$ such that $y_1,y_2\neq 0,1$. This number is $81$, showing that $81\in V(Q')$, and $V(Q)\neq V(Q')$. Hence the quasiproduct is not Hadamard equivalent to any of the weaving products. Incidentally, this argument also shows that the sequence \eqref{projext} is non-split.

	\subsection{Grassmannian Varieties and weighing matrices} Grassmannian Varieties are natural generalizations of projective Spaces. We will show how to construct a weighing matrix whose axes are indexed by the Grassmannian, in a way that generalizes the construction of \S \ref{sec:proj}. There are two main complications over the case of Projective-Spaces. First, occurrence relations are more involved: In fact there are few different occurrence relations, depending on the dimension of the intersection of certain vector spaces, see Definition \ref{occ} below. Second, the action on the rows (and columns) is no longer 2-transitive. However, most of the arguments of \S \ref{sec:proj} will work here.\\
	
	The background material on finite Grasmannian spaces is well-known, see e.g. \cite{Zynel+2001+145+160}. We quickly go over the background to set up notation.
	\begin{defn}
		Let $F=\FF_q$ be a finite field, and $1\le k<d$ be integers. The \emph{Grassmannian Variety} with parameters $(d,k)$ over $F$ is the set
		$$ Gr(d,k,F)\ := \ \left\{\text{linear subspaces }V\subset F^d\ |\ \dim V=k\   \right\} .$$
	\end{defn}
	As before there is a notion of duality. We fix a nondegenerate bilinear form $\langle\ ,\ \rangle$ on $F^d$ and identify $Gr(d,k,F)$ with $Gr(d,d-k,F)$ by 
    
   $$ \mathrm{Duality}: V \longleftrightarrow\ W=\{x\in F^{d} \ |\ \forall v\in V,\ \langle v,x\rangle =0 \}.$$
   We will also write $W=V^D$, where $W\in Gr(d,d-k,F)$ corresponds to $V\in Gr(d,k,F)$ by duality. We have $V^{DD}=V$.\\
   
   Each point $V\in Gr(d,k,F)$ is represented by a non unique $d\times k$ matrix $A_V$ over $F$ of rank $k$, such that  the columns of $A_V$ form a basis for $V$. Two representing matrices $A_V$ and $A'_V$ differ by invertible column operations, i.e. $A_V=A'_VT$ for some $T\in GL_k(F)$. In particular $V$ has a unique representing matrix in column-reduced-echelon-form (cref).
   Denote $V=[A_V]$ the column space of $A_V$.\\
   
   The group $GL_d(F)$ acts naturally on $Gr(d,k,F)$. Namely, $g\in GL_d(F)$ acts $[A_V]$ as $[gA_V]$. With respect to duality (working with the standard bilinear form) $g\cdot V^D=((g^{-1})^TV)^D$. The occurrence relations are defined as follows:\\
   
   \begin{defn}\label{occ}
   	A pair $(V,W)\in Gr(d,k,F)\times Gr(d,d-k,F)$ has \emph{occurrence degree} $\delta$, if $\dim(V\cap W)=\delta.$ Equivalently, the occurrence degree of $(V,W^D)$ is
   	$$\delta=k-rank (\langle v_i,w_j\rangle_{i,j}),$$ for bases $\{v_1,\ldots,v_k\}\subset V$ and $\{w_1,\ldots,w_k\}\subset W$.
   \end{defn}
   The possible occurance degrees are the integers $0\le \delta\le \min(k,d-k)$. When $k=1$, $Gr(d+1,1,F)=\PP^d(F)$ and the theory reduces to that of projective spaces. We now begin the construction of the Grassmannian weighing matrices.\\
   
  The order $\# Gr(d,k,F)$ is given in terms of a \emph{Gaussian binomial coefficients}, which are defined as follows:
   \begin{gather}
   [m]_q \ := \ \frac{q^m-1}{q-1}\\
   [m]_q! \ := \ [1]_q[2]_q\cdots [m]_q\\
   \genfrac{[ }{ ]}{0pt}{}{\, m\, }{\, r\, }_q \ := \ \frac{[m]_q!}{[r]_q![m-r]_q!}
   \end{gather}
   It can be shown that (see derivation in \S\ref{sec:flag} for Flag-Varieties below)
   \be \nonumber \# Gr(d,k,F=\mathbb F_q)\ = \ \genfrac{[ }{ ]}{0pt}{}{\, d\, }{\, k\, }_q.\ee

   \begin{rem}
		The reader might notice that analogy with permutation groups. Instead of vector spaces and subspaces of a fixed dimension we can study a set and its subsets of a fixed cardinality. The general linear group is then replaced by the symmetric group. Example \ref{159} is essentially of this type, and the weighing matrix $W(15,9)$ developed later should be thought as the analog of the grassmannian weighing matrix developed here.
   \end{rem}

   Let $G=PGL_d(F)$, $X=Gr(d,k,F)$ and $Y=Gr(d,k,F)$ with the dual action: $g$ acts as $(g^{-1})^T$. Using duality $Y$ can be identified with the dual space $Gr(d,d-k,F)$ and the usual action. By standard linear algebra, the action of $G$ on $X\times Y$ breaks into $1+\min(k,d-k)$ orbits, one for each occurrence degree. The action on $X$ (and $Y$) is no longer 2-transitive. In fact, when one fixes $x_0\in X$, then $X\setminus\{x_0\}$ is not transitive. Points $x\in X$ with different intersection dimensions $\dim(x_0\cap x)$ belong to different $G$-orbits.\\
   
   From now on we fix $x_0\in X$ to be the space represented by $[I_k,0]^T\in M_{d\times k}(F)$, and let $y_0=x_0$ be the basepoint in $Y$. The stabilizers of $x_0$ and $y_0$ are
   \begin{align}
   H_X\ &= \ \left\{ \begin{bmatrix}
   A & B\\ 0 &D 
   \end{bmatrix}\in GL_{d}(F)\ \left|\ A \text{ is } k\times k \right.    \right\},\text{ and }\\
   H_Y=H_X^T\ &= \ \left\{ \begin{bmatrix}
   A & 0\\ C &D 
   \end{bmatrix}\in GL_{d}(F)\ \left|\ A \text{ is } k\times k \right.    \right\}.
   \end{align}
   We will work with the homomorphisms $\psi_X$ and $\psi_Y$ given by:
   \begin{align}
   \psi_X\left( \begin{bmatrix}
   A & B\\ 0 &D 
   \end{bmatrix}\right) \ &= \ \left(\frac{\det(A)}{F}\right)_n, \text{ and }\\
   \psi_Y\left( \begin{bmatrix}
   A & 0\\ B &D 
   \end{bmatrix}\right) \ &= \ \left(\frac{\det(A)}{F}\right)_n.
   \end{align}
   
   The orbit $O_0$ containing $(x_0,y_0)$ consists of all pairs $(x,y)\in X\times Y$, such that $y^D\cap x$ is zero dimensional. This is the `generic' orbit. $\psi_X$ and $\psi_Y$ agree on $H_X\cap H_Y$, so this orbit is orientable. The orbit $O_i$ of intersection degree $i$ contains the point $(x_0,y_i)$, where $y_i$ is the space represented by the block matrix
   $$\begin{bmatrix}
   I_{k-i} & 0_i & 0 & 0\\
   0 & 0_i & I_{i} & 0
   \end{bmatrix}^T.$$ We choose the coset representatives for $G/H_Y$:
   $$g_{y_i}= \begin{bmatrix} I_{k-i} & 0 & 0 & 0\\
   0 & 0 & I_{i} & 0\\
   0 & I_{i} & 0 & 0\\
   0 & 0 & 0 & I_{d-i-k}
   \end{bmatrix}$$ mapping $y_0$ to $y_i$. In a similar way to \S\ref{sec:proj} it is clear that $\psi_X(h)$ and $\psi_Y(g_{y_i}^{-1}hg_{y_i})$ do not agree on $h\in H_X\cap g_{y_i}hg_{y_i}^{-1}$ for $1\le i\le \min(k,d-k)$, therefore $O_0$ is the only orientable orbit in $X\times Y$. Notice that the weight of the orbit $O_0$, i.e. the number of $y$ s.t. $(x_0,y)\in O_0$ is $q^{k(d-k)}$. This is easy to see, since every such $y$ can be uniquely represented by a column-reduced echelon form matrix $[I_k,*]^T$, where $*$ represents an arbitrary submatrix, and the dimension of this space is $k\times (d-k)$.\\
   
   Similarly, we analyze $X\times X$. Again there are $1+min(k,d-k)$ orbits, according to the dimension of the intersection. $(x_1,x_2)$ and $(x_3,x_4)$ are in the same $G$-orbit, if and only if $\dim(x_1\cap x_2)=\dim(x_3\cap x_4)$. A similar analysis shows that the only orientable orbit is the diagonal. Therefore, as for projective spaces we obtain.
   
   \begin{thm}\label{grass}
   	For every prime power $q$, integers $d>k>0$, $1<n|(q-1)$, there exists a generalized weighing matrix 
   	$$A\in GW\left(\genfrac{[ }{ ]}{0pt}{}{\, d\, }{\, k\, }_q,q^{k(d-k)};n\right).$$
   	The matrix $|A|$ is the characteristic matrix of the $G$-orbit of all $(x,y)\in X\times Y$ such that $x\cap y^D=\{0\}$.
   	Moreover, every automorphism in the image of $GL_d(F)\to \PermAut(|A|)$ lifts to an automorphism of $A$.
    \end{thm}

    \subsection{Flag-Varieties and weighing matrices}\label{sec:flag}
    
    \emph{Flag Varieties} are yet a further generalization of grassmannians. We still obtain weighing matrices over $\mu_n$, but only provided that $n$ is large enough. Let $\mathcal P=(d_1,d_2,\ldots,d_r)$ be a \emph{positive} {\bf ordered} \emph{partition} of $d$, which means that $d=\sum_i d_i$, $d_i>0$, and the order in the vector matters. We write $r=len(\mathcal P)$, the \emph{length} of the partition. A \emph{flag} of type $\mathcal P$ over a field $F=\mathbb F_q$ is a sequence of vector spaces
    $$\mathcal F: \ \ V_1\subset V_2 \subset \cdots\subset V_r=V:=F^d,$$ such that $\dim V_i=\sum_{j\le i} d_j$. The \emph{Flag-Variety} of type $\mathcal P$ over $F$ is the set
    $$Flag(\mathcal P,F)\ := \ \{\text{Flags } \mathcal F\ | \ \mathcal F \text{ is of type $\mathcal P$ over }F \}.$$
    The Grassmanian $Gr(d,k,F)$ is the special case where $\mathcal P=(k,d-k)$, and the Projective Space $\mathbb P^d(F)$ is the case $\mathcal P=(1,d)$.\\
    
    A flag $\mathcal F$ can be encoded by a $d\times (d-d_{r-1})$ matrix $B=B_{\mathcal F}$, such that the first $\sum_{j\le i} d_j$ columns of $B$ are a basis for $V_i$. The matrix $B$ is not unique. We may form limited column operations on $B$ in the following manner. It is possible to rescale any column of $B$. It is also possible to add a multiple of column $i$ to any column $j$ for $j>i$. But for $j<i$ this is generally forbidden, except when $j$ belongs to the \emph{same block} as $i$. This means that 
	 for some $l$, $\sum_{k\le l} d_k < i,j \le \sum_{k\le l+1} d_k$. Such operations preserve the flag. The group generated by such operations transform a basis matrix $B_{\mathcal F}$ to any other basis matrix $B'_{\mathcal F}$ of the flag $\mathcal F$. In matrix language, let $T\subset G:=GL_{d-d_r}(F)$ be the subgroup of all block-upper-triangular matrices
    $$ T:=\begin{bmatrix} A_1 & * &\cdots & *\\
                       0  & A_2 & * & *\\ 
                       \vdots & \ddots &\ddots & \vdots\\
                       0 & \cdots & 0 & A_{r-1}
    
    \end{bmatrix}$$
    such that $A_i$ has dimensions $d_i\times d_i$, $1\le i<r$. 
    We can identify the flag variety with the quotient
    
    $$Flag(\mathcal P,F)=\left\{B\in F^{d\times (d-d_r)} \\
                                \ \left| \  rank(B)=d-d_r \right. \right\}\Big/ T.$$
    The group $G=GL_d(F)$ acts by matrix multiplication on the left on basis matrices. This action descends to $Flag(\mathcal P,F)$, and is transitive. Write $X=Flag(\mathcal P,F)$, and let $x_0\in X$ be the \emph{standard} flag of type $\mathcal P$. This is the flag represented by the basis matrix $B_0=[I_{d-d_r},0]^T$. The stabilizer $H_X$ of $x_0$ in $G$ is the subgroup of all block upper triangular $d\times d$  matrices \be \label{FlagHX}\begin{bmatrix} A_1 & * &\cdots & *\\
    0  & A_2 & * & *\\ 
    \vdots & \ddots &\ddots & \vdots\\
    0 & \cdots & 0 & A_{r}\end{bmatrix},\ee
    such that $A_i$ has dimensions $d_i\times d_i$.\\
    
    We define the dual space $Y$ to be the same as $X$ as a set, but where $g\in G$ acts by $(g^{-1})^T$. We may define the dual flag variety via a nondegenarate bilinear form and our choice of $Y$ is the outcome of using the standard bilnear form and identifying via the duality map. The stabilizer of $Y\ni y_0=x_0$ is $H_Y=H_X^T=\{A^T\ |\  A\in H_X\}.$\\
    
    \subsubsection{The order of $Flag(\mathcal P,F)$} We compute the cardinality of $H_X$. The size of the general linear group is
    $$\# GL_d(\mathbb F_q)= (q-1)^dq^{d(d-1)/2}[d]_q!.$$ We need to count how many matrices $S\in H_X$ exist. The blocks $A_i$ should be in $GL_{d_i}(F)$, while the blocks above the block-main diagonal can be general matrices. Then 
    \be \label{uuu}
       \# H_X=\prod_i \#GL_{d_i}(F) \cdot q^u,
    \ee 
    where $u$ is the number of entries above the block-main diagonal. There are $\sum_i d_i^2$ entries of the block-main diagonal, hence $u=\tfrac{1}{2}(d^2-\sum_i d_i^2)$. Therefore
    \begin{multline*}
        \# H_X= \prod_i \#GL_{d_i}(F) \cdot q^{\tfrac{1}{2}(d^2-\sum_i d_i^2)}=
        q^{\tfrac{1}{2}d^2}\prod_i (q-1)^{d_i}q^{\tfrac{1}{2}(d_i^2-d_i)}[d_i]!_q q^{-\tfrac{1}{2}d_i^2}\\
        =q^{\tfrac{1}{2}(d^2-d)}(q-1)^d \prod_i [d_i]!_q.
    \end{multline*}
    Consequently,
    $$ \# Flag(\mathcal P,F)\ = \# GL_{d}(F)/\# H_X=\frac{[d]!_q}{\prod_i [d_i]!_q}=: \genfrac{[ }{ ]}{0pt}{}{\, d\, }{\, d_1,d_2,\ldots,d_r\, }_q,$$
    the \emph{multinomial Gaussian coefficients}.

    \subsubsection{The $G$-orbits.\\}
    The $G$-orbits of $X\times X$ can be described by intersection types of two flags of type $\mathcal P$. Let us explain this point. If $\mathcal F=(V_1\subset V_2 \subset\cdots \subset V_r)$ and $\mathcal F'=(V'_1\subset V'_2\subset \cdots \subset V'_r=V_r=V)$, then $(*)$ $e_{i,j}:=\dim (V_i\cap V'_j)$ is an increasing monotone function in each $i$ and $j$. In addition, $(**)$ for all $i>1$, $e_{i,j}-e_{i-1,j}$ is an increasing monotone function in $j$, and similarly $e_{i,j}-e_{i,j-1}$ is monotonely increasing in $i$. To see this notice that $e_{i,j}-e_{i-1,j}=\dim((V_i\cap V_j)/(V_{i-1}\cap V_j))=\dim((V_{i-1}+V_i\cap V_j)/V_{i-1})$ which is clearly monotone increasing in $j$.
    Also $(***)$ $e_{i,r}=e_i:=\dim V_i=\dim V'_i=e_{r,i}$. We define $e_{i,0}=e_{0,i}=0$. We call a matrix $(e_{i,j})$ of $0\le i,j\le r$ satisfying $(*)$--$(***)$, an \emph{intersection type}. 
    \begin{lem}\label{int.type}
    	Every intersection type $\{e_{i,j}\}$ satisfying $(*)-(***)$ appears as an intersection of certain two flags of type $\mathcal P$.
    \end{lem}
    
    \begin{proof}
    	We first prove a combinatorial version on sets. Write 
    	$\inn{n}=\{1,2,\ldots,n\}$. Let $e_i=\sum_{j\le i}d_j$. Consider the flag of sets $S_1\subset S_2\subset \cdots \subset S_r=\inn{d}$, such that $S_i=\inn{e_i}$. Define another flag of sets $T_1\subseteq T_2\subset \cdots \subset T_r=\inn{d}$ recursively as follows. Let $T_1^l\subseteq S_l\setminus S_{l-1}$ be the subset of cadinality $e_{1,l}-e_{1,l-1}$ consisting of the smallest possible numbers. Let $T_1=\bigcupdot_l T_1^l$, which clearly has $\# T_1=e_{1,r}=e_1$. We continue to define $T_2$. Let $T_2^l\subseteq S_l\setminus (S_{l-1}\cup T_1^l)$ be the set of the cardinality $(e_{2,l}-e_{2,l-1})-(e_{1,l-1}-e_{1,l-1})$ consisting of the smallest possible numbers. Notice that this number is nonnegative since the $e_{i,j}$ satisfy condition $(**)$, and on the other hand this number $\le (e_l-e_{l-1})-(e_{1,l}-e_{1,l-1})=\# S_l\setminus (S_{l-1}\cup T_1^l)$. Also note that $\#(T_l^2\cupdot T_l^1)=e_{2,l}-e_{2,l-1}$.
		Let $T_2=T_1\cupdot\bigcupdot_l T_2^l$. At the $i$th stage we proceed similarly: Let $T_i^l\subseteq (S_{l}\setminus S_{l-1})\setminus (T^l_{i-1}\cupdot \cdots T^l_1)$ be the subset of cardinality $(e_{i,l}-e_{i,l-1})-(e_{i-1,l}-e_{i-1,l-1})\ge 0$ consisting of the smallest possible numbers, and let $T_i=T_{i-1}\cupdot \bigcupdot_l T_i^l$. Note that some of the $T_i^l$'s may be empty. This constructs a flag of sets $(T_i)$ with $\# T_i\cap (S_l\setminus S_{l-1})=\# \bigcupdot_{j\le i} T_j^l\cap S_l=e_{i,l}-e_{i,l-1}$, and $\# T_i\cap S_l=\sum_{k\le l} (e_{i,k}-e_{i,k-1})=e_{i,l}$. Recall that $\e_i$ is the $i$th standard vector. Let $\mathcal F=(V_i)$ be the \emph{standard flag}, which means that $V_i=span\{\e_k, \ k\in S_i\}$ and $\mathcal G$ be the flag $(W_i)$ defined by $W_i=span\{\e_k, \ k\in T_i\}.$ Then $\mathcal F$ and $\mathcal G$ have the desired intersection type.
    \end{proof}
    We call a pair of flags $(\mathcal F,\mathcal G)$ a \emph{standard pair} if both flags are constructed from the standard basis in different orders.
    We next claim that intersection types characterize the $G$-orbits in $X\times X$.
    \begin{lem}\label{lem:610}
    	Two pairs of flags $(x_1,x_2)$ and $(z_1,z_2)$ of type $\mathcal P$ are in the same $G$-orbit, if and only if they have the same intersection type. Every $G$-orbit contains a standard pair.
    \end{lem}

    \begin{proof}
    	If $(z_1,z_2)=(gx_1,gx_2)$ for some $g\in G=GL_d(F)$, then they are of the same intersection type, because $g\in G$ preserves dimensions of vector spaces and their intersections. Conversely, suppose that $(x_1,x_2)$ and $(z_1,z_2)$ have the same intersection type, $e_{i,j}$. Write $x_1=(V_i)$ and $x_2=(W_j)$. We have $e_{i,j}-e_{i,j-1}-e_{i-1,j}+e_{i-1,j-1}=\dim(W_j\cap V_i) - \dim(W_{j-1}\cap V_i)-\dim(W_j\cap V_{i-1})+\dim(W_{j-1}\cap V_{i-1})=\dim\left(Q_{i,j}\right)$ for $Q_{i,j}=\frac{(W_j\cap V_{i})/(W_{j-1}\cap V_{i})}{(W_{j}\cap V_{i-1})/(W_{j-1}\cap V_{i-1})}.$ For every $i,j$, let $B_{i,j}\subset W_j\cap V_i$ be a set of that size bijecting onto a basis of $Q_{i,j}$. The elements of $B=\bigcupdot B_{i,j}$ are linearly independent since $B_{i,j}$ projects to $\{0\}$ in $Q_{i',j'}$ for $i'<i'$ or $j'<j$. Hence $B$ is a basis for $V$. Moreover, $\bigcup_{i,k, k\le j} B_{i,j}$ is a basis for $W_j$ and $\bigcup_{j,k,k\le i} B_{i,j}$ is a basis for $V_i$. If $g$ is a linear transformation sending $B_{i,j}$ to the standard basis in some order, then $(gx_1,gx_2)$ is a standard pair, and by choosing an appropriate ordering we can make $gx_1$ be the standard flag. 
    \end{proof}

    \subsubsection{Homomorphisms and Orientability}
    
    Let $\psi_X:H_X\to \mu_n$ be defined by 
    $$\psi_X(S)=\prod_{i=1}^r \left( \frac{\det(A_i)}{F} \right)_n^i. \ \quad\quad \text{(For $S$ as in \eqref{FlagHX})}$$ Similarly $\psi_Y:H_Y\to \mu_n$ is given by $\psi_Y(S)=\psi_X(S^T).$ We claim:
    
    \begin{lem}\label{lem:flag611}
    	Suppose that $n\ge len(\mathcal P)$. Then the only orientable orbit w.r.t. $(\psi_X,\psi_X)$ in $X\times X$ is the main diagonal.
    \end{lem}

    \begin{proof}
    	Let $O\subset X\times X$ be an orbit. Then according to Lemma \ref{lem:610} there is a standard pair $(\mathcal F,\mathcal G)\in O$, where $\mathcal F=x_0$ is the basepoint. There is a permutation matrix $\sigma\in G$ such that $y=\mathcal G=\sigma x_0$. Thus orientability at $(x_0,y)$ is equivalent to the condition that 
    	$$\psi_X(h)=\psi_X(\sigma^{-1} h \sigma), \ \ \forall h\in H_X\cap \sigma H_X\sigma^{-1}.$$
    	Suppose that $O$ is not the main diagonal orbit. Then $\sigma$ is a permutation that does not preserve the block partition of $S$, or else $\sigma x_0$ would be the same flag $x_0$. So there is an entry $k$ in the $i$th block, such that $\sigma(k)$ is in the $j$th block for $j\neq i$. Take $h=\diag(1,1,\ldots,t,1,\ldots,1)\in H_X\cap \sigma H_X\sigma^{-1}$, where the $t$ is at position $k$. Then $\psi_X(h)=\left(\frac{t}{F}\right)^i_n$, and $\psi_X(\sigma^{-1}h\sigma)=\left(\frac{t}{F}\right)^j_n$. We have $n\ge len(\mathcal P)>|i-j|$. For $t$ a generator of $F^\times$,  $\psi_X(h)\neq \psi_X(\sigma^{-1} h \sigma)$, proving that $O$ is not orientable.
    \end{proof} 

    We can now state and prove the theorem for existence of Flag-Variety weighing-matrices.
    
    \begin{thm}
        Let $q$ be a prime power, $d=\sum_{i=1}^r d_i$, $d_i>0$ integers and $1<n|(q-1)$. Suppose that $n\ge r$.  Then there exists a generalized weighing matrix 
        $$A\in GW\left(\genfrac{[ }{ ]}{0pt}{}{\, d\, }{\, d_1,d_2,\ldots,d_r\, }_q,q^{\tfrac{1}{2}(d^2-\sum_i d_i^2)};n\right).$$
        The matrix $|A|$ is the characteristic matrix of the $G$-orbit in $X\times Y$ spanned by $(x_0,y_0)$.
        Moreover, every automorphism in the image of $GL_d(F)\to \PermAut(|A|)$ lifts to an automorphism of $A$.	
    \end{thm}

    \begin{proof}
    	We construct the $H^1$-monomial cover of $G=GL_d(F)$ with respect to the characters $\psi_X$ and $\psi_Y$. Let $O$ be the orbit spanned by $(x_0,y_0)=(x_0,x_0)$. The orientability of $O$ is established by the fact that $\psi_X(S)=\psi_Y(S)=\prod_i\left( \tfrac{\det(A_i)}{F} \right)_n^i$ for all $S\in H_X\cap H_Y$. Let $A$ be the Cohomology-Developed matrix w.r.t. to this monomial cover, supported on $O$ and normalized so that $A_{x_0,y_0}=1$. By our assumption that $n\ge r$, Lemma \ref{lem:flag611} and Theorem \ref{orth}(c), $A$ is an orthogonal matrix, hence a weighing matrix.\\
    	
    	It remains to determine its weight. The stabilizer of $(x_0,y_0)$ is $H_X\cap H_Y$. Thus the weight is the number of $(x_0,y)\in O$, that is
    	\begin{multline*}
    	\# H_X/ \# (H_X\cap H_Y)= \# H_X/\prod_i \# GL_{d_i}(F)=q^u \text{ (Eq. \eqref{uuu}) }= q^{\tfrac{1}{2}(d^2-\sum d_i^2)}. 
    	\end{multline*}
    \end{proof}

    \begin{rem}
    	Similarly to \S \ref{sec:quasi}, as for projective space matrices, we would like to construct quasiproducts $W\boxtimes T$ when $W$ is a grassmannian of flag variety weighing matrix. This is possible provided that the sequence \eqref{projext} splits when restricted to $H_X$ and $H_Y$. This is not always true, but in many cases it is, such as when one of the $d_i$, of any partial sum of these is coprime to $d$.
    \end{rem}

	\section{$H^2$-developed matrices}
	We turn to the study of $H^2$-developed matrices. This section contains more advanced material on cohomology, and the non-familiar reader may skip this section in a first reading. The entire set of Cohomology-Developed Matrices is covered by the $H^1$-development discussed above, if one is willing to replace $G$ with an extension group $\widetilde{G}$ in a monomial cover.\\

	Recall that we construct Cohomology-Developed matrices (CDM) by computing monomial covers 
	\begin{equation}\label{diag7} \begin{tikzcd}[column sep = huge, row sep=huge]
	\ \ \widetilde{G}\ \  \arrow[dashed,two heads,shorten >= 5pt,shorten <= 5pt]{r}{\rho}  \ar[dashed]{d}{\pi_X\times \pi_Y}
	&\ \ G \ \ \ar{d}{\ p_X\times p_Y}  \ar[dotted,swap]{dl}{\exists \theta ?} \\
	Mon(X,\mu_n)\times Mon(Y,\mu_n)  \arrow{r}{abs\times abs} & Perm(X)\times Perm(Y)
	\end{tikzcd},
	\end{equation}
	which supplies us with an action of $\widetilde{G}$ on $X\times Y$ matrices, and CDMs are those that are invariant under $\widetilde{G}$. In the case of $H^1$-development, the map $\widetilde{G}\to G$ is an isomorphism, or put it more generally, there is a homomorphism $\theta:G\to Mon(X,\mu_n)\times Mon(Y,\mu_n)$ which makes the diagram commutative. The $\widetilde{G}$-action descends to the $G$-action via $\theta$, and for all purposes we can replace $\widetilde{G}$ with $G$.\\
	
	In the general case, however, the diagonal arrow $\theta$ may not exist, and $\widetilde{G}\to G$ is a proper extension. An example has been given above in Example \ref{Had2}. Cocyclic matrices, to be discussed below (\S\ref{subsec:coc}), are examples as well. By the definition monomial covers, the kernel $\widetilde{G}\to G$ is mapped by $\pi_X\times \pi_Y$ to $Triv$. Let $N\trianglelefteq \widetilde{G}$ be the normal subgroup $\ker (\rho) \cap \ker (\pi_X\times \pi_Y)$. Then we may replace  $\widetilde{G}$ with $\widetilde{G}/N$ in the diagram, and the monomial action descends to $\widetilde{G}/N$. Hence without loss of generality we will assume that $N=\{1\}$. A monomial cover satisfying this condition is called \emph{restricted}. In a restricted monomial cover, $\ker(\rho)$ \emph{injects} into $Triv$. Moreover, $\ker(\rho)$ is \emph{central} in $\widetilde{G}$, since if we take $z\in \ker(\rho)$ and $g\in \widetilde{G}$, the commutator $[z,g]$ is mapped by $\pi_X\times\pi_Y$ to a commutator with $Triv$, which is the identity, hence to commutator is in $N=\{1\}$. Thus we obtain a central extension sequence
	\be\label{ext7}  1\to Z\to \widetilde{G} \to G \to 1,\ \  Z=\ker(\rho). \ee 
	
	\begin{lem}
		For restricted monomial covers, The following are equivalent:
		\begin{itemize}
			\item[(i)] A homomorphism $\theta$ exists.
			\item[(ii)] $Z=\{1\}$
			\item[(iii)] The extension \eqref{ext7} is split.
		\end{itemize}
		
	\end{lem}

    \begin{proof}
    	(i)$\implies$ (ii): 
    	suppose that $\theta$ exits. An element in $\ker \rho$ is mapped to $1$ in $G$, and by the commutativity of the diagram it is also mapped to $1$ in $Mon(X,\mu_n)\times Mon(Y,\mu_n)$. Therefore $Z\cong \ker(\rho)\subseteq N=\{1\}$.\\
    	(iii)$\implies$(i): If a section $s:G\to \widetilde{G}$ exists, then define $\theta=(\pi_X\times \pi_Y)\circ s$.\\
    	(ii)$\implies$(iii): Trivial. 
    \end{proof}
    The equivalent conditions in the lemma characterize the case of $H^1$-development. The general case is called $H^2$-development.\\
    
    There are two equivalent methods to construct $H^2$-developed matrices.
    \begin{itemize}
    	\item[Method I:] Compute extensions \eqref{ext7} and then replace $G$ with a cover $\widetilde{G}$ acting permutationally on $X$ and $Y$. The construction will proceed by $H^1$-development with respect to $\widetilde{G}$, using machinery of \S 3 to compute CDMs. But beware that not all extensions $\widetilde{G}$ will work. In any suitable extension, $Z$ must be mapped to $Triv$ by the constructed monomial cover. We will see below that this is a significant constraint.\\
    	\item[Method II:] Try to construct all restricted monomial covers \eqref{diag7} directly from $G$, using Cohomology.
    \end{itemize}
     In both methods, we will need to use the second cohomology group of $G$. We will proceed below with Method II, and at certain points we will explain the equivalents with Method I.\\
     
     In a restricted monomial cover, we can form a map $\tilde s:G \to \widetilde{G}$, which is a section to $\rho$, but in general not a homomorphism. This map yields a function $\tilde\theta:G\to Mon(X,\mu_n)\times Mon(Y,\mu_n)$, which descends to a \emph{homomorphism} $\bar\theta: G\to Mon(X,\mu_n)\times Mon(Y,\mu_n)/Triv$. In fact the homomorphism $\bar\theta$, is all that we need in order to generate CDMs. This is due to the fact that $Triv$ acts trivially on matrices, and $\ker\rho$ maps into $Triv$. Moreover, this data is sufficient to generate a monomial cover accounting for the monomial action, as is guaranteed in lemma \ref{without cover}. There is the split exact sequence 
     $$1 \to (D_X\times D_Y)/Triv \to (Mon(X,\mu_n)\times Mon(Y,\mu_n))/Triv\to Perm(X)\times Perm(Y)\to 1.$$

      By Theorem \ref{class_sect}, homomorphisms $\bar\theta:G\to Mon(X,\mu_n)\times Mon(Y,\mu_n)/Triv$ satisfying $(abs\times abs)\circ \bar\theta=p_X\times p_Y$, identified up to diagonal conjugation, are in bijection with the $1$st cohomology group $H^1(G,(D_X\times D_Y)/Triv) \cong H^1(G,(\mu_n[X]\oplus \mu_n[Y])/\mu_n)$. Here we view $\mu_n$ as the subgroup of $\mu_n[X]\oplus \mu_n[Y]$ via the diagonal embedding $\zeta \mapsto (\sum_x\zeta [x],\sum_y \zeta [y]).$ There is the natural map
    
    \be\label{51} H^1(G,\mu_n[X]\oplus \mu_n[Y]) \to H^1(G,(\mu_n[X]\oplus \mu_n[Y])/\mu_n)\ee
    induced from the surjection of $G$-modules  $\mu_n[X]\oplus \mu_n[Y]\to (\mu_n[X]\oplus \mu_n[Y])/\mu_n$. The $H^2$-developed matrices which are not $H^1$-developed  are coming from classes in the range of \eqref{51}, which are not in the image of the source. We will see below that the cokernel of this map is measured by the second cohomology group $H^2(G,\mu_n)$.

    Consider the exact sequence of $G$-modules,
    \be\label{exactm} 0\to \mu_n \to \mu_n[X]\oplus \mu_n[Y]\to (\mu_n[X]\oplus \mu_n[Y])/\mu_n \to 0. \ee
    
    Then the map \eqref{51} can be extended to the long exact sequence in cohomology:
    
    \begin{multline}\label{bound} H^1(G,\mu_n[X]\oplus \mu_n[Y]) \to H^1(G,(\mu_n[X]\oplus \mu_n[Y])/\mu_n)\\ \stackrel{\partial}{\longrightarrow} H^2(G,\mu_n)\to H^2(G,\mu_n[X]\oplus \mu_n[Y]).
    \end{multline}
    
    Using the lemma of Eckmann Shapiro (Theorem \ref{shap}), we can rewrite this as
    
    \begin{multline}\label{59} Hom(H_X,\mu_n)\oplus Hom(H_Y,\mu_n) \to H^1(G,(\mu_n[X]\oplus \mu_n[Y])/\mu_n)\\ \stackrel{\partial}{\longrightarrow} H^2(G,\mu_n) \stackrel{res}{\longrightarrow} H^2(H_X,\mu_n)\oplus H^2(H_Y,\mu_n),
    \end{multline}
    and the map $res: H^2(G,\mu_n)\to H^2(H_X,\mu_n)\oplus H^2(H_Y,\mu_n)$ is defined by the restriction maps to $H_X$ and $H_Y$.\\

    \subsection{Construction of CDMs - Method II}\label{cdm2}
	The sequences \eqref{bound} or \eqref{59} give us a way to compute the middle cohomology group $H^1(G,(\mu_n[X]\oplus \mu_n[Y])/\mu_n)$, which is what we need to construct monomial covers of the combinatorial action $g\in G \mapsto (p_X(g),p_Y(g))$.
    We begin with a 2-cocycle $w\in Z^2(G,\mu_n)$, representing a cohomology class $[w]\in H^2(G,\mu_n)$. It is necessary that $res([w])=0$. If so, the exactness of the sequence \eqref{59} allows us to pull back (non-uniquely) to a 1-cocycle $z\in Z^1(G,(\mu_n[X]\oplus \mu_n[Y])/\mu_n)$. The cohomology class $[z]$ is not uniquely determined. There still remains the ambiguity of the two characters in $(\psi_X,\psi_Y)\in Hom(H_X,\mu_n)\oplus Hom(H_Y,\mu_n)$ that contribute to $[z]$. All in all, the class $[\omega]$ and the two chacacters $\psi_X,\psi_Y$ are exactly the information needed to construct the class $[z]$.\\	
	
	Next, given the cocycle $z$, we construct a homomorphism $s:G\to Mon(X,\mu_n)\times Mon(Y,\mu_n)/Triv$ by lifting $z$ to a function $\tilde z:G\to D_X\times D_Y$ and then letting $s(g)=\tilde z(g)\cdot (p_X(g),p_Y(g))$. The function $s$ gives a well-defined monomial action of $G$ on matrices, by the pair of monomial matrices $s(g)=(P(g),Q(g))$. This action is well defined, and moreover is a left $G$-action, since $Triv$ acts trivially on matrices. We now use this new action to construct CDMs. Those are matrices that are $G$-invariant  w.r.t. this action. To give such a matrix $A$, it is enough to determine its value at a single point in each $G$-orbit, and then use the $G$-action to determine the values of $A$ at the remaining positions. Some orbits will be non-orientable for this action, as can happen with $H^1$-development.

    \subsection{$2$-cocycles and extensions. Comparison of Methods I and II}
    
    To the central extension $\mathcal E$ in \eqref{ext7}, there is associated a cohomology class in $H^2(G,Z)$. This works as follows. Construct a map $s:G\to \widetilde{G}$, which is a section to $\rho$, but may fail to be a homomorphism. Then (in multiplicative syntax), 
    $$\omega(g_1,g_2):= s(g_1g_2)s(g_1)^{-1}s(g_2)^{-1}$$ is a 2-cocycle, and we take its class $[\omega]$ as the associated cohomology class in $H^2(G,Z)$.\\
    
    Conversely, a 2-cocycle $\omega\in Z^2(G,Z)$ gives rise to a central extension 
    $$\mathcal E_\omega: 1\to Z \to E \to G \to 1,$$
    given by the following recipe. As a set, $E=Z\times G$ with the natural inclusion map $Z\to Z\times \{1\}\subseteq Z\times G$ and the projection map $Z\times G\to G$. The group law on $E$ is given by
    $$(z_1,g_1)\cdot (z_2,g_2):= (z_1z_2\omega(g_1,g_2),g_1g_2),$$
    
    It is well-known (see \cite[Chap IV]{1982cohomology}) that this gives a bijection between $H^2(G,Z)$ and central extensions of $G$ by $Z$, up to isomorphism. By an isomorphism of extensions we mean an isomorphism between the middle terms, which together with the identity maps on $Z$ and $G$ commutes with the arrows of the diagrams.\\
    
    In method I, we start from an extension $\mathcal E_\omega$ of $G$ by $Z$. Then we attempt to fit this extension within a restricted monomial cover. From this we obtain a monomial action of $\widetilde{G}=E$ on matrices. The following Proposition is the basis for the comparison between methods I and II.
    
    \begin{prop}\label{prop:iota} Let 
    	$$1\to Z\to \widetilde{G} \to G \to 1$$ be the extension given by a restricted monomial cover, and let $[\omega]\in H^2(G,Z)$ be the associated cohomology class. Let $s:G\to \widetilde{G}$ be a section to $\rho$, and let $\bar\theta=(\pi_X\times \pi_Y)\circ s \mod Triv$ be a homomorphism $\bar\theta: G\to Mon(X,\mu_n)\times Mon(Y,\mu_n)/Triv$. Finally let $z\in Z^1(G,\mu_n[X]\oplus \mu_n[Y]/\mu_n)$ be the $1$-cocycle corresponding to $\bar\theta$, and let $\iota:Z\to \mu_n\cong Triv$ be the natural inclusion. Then
    	$$\partial [z] =\iota_*[\omega].$$
    \end{prop}

    \begin{proof}
    	This is a straightforward but tedious checking of definitions. The details are left to the reader.
    \end{proof}

    \begin{rem}
    	Suppose that we use method II, starting with a class $[\omega]\in H^2(G,\mu_n)$ (perhaps with values in $Z\subseteq \mu_n)$). Then method II constructs a well-defined monomial action on matrices, and we obtain matrices $A$ that are invariant under this action. Such matrices satisfy that $gA\sim_D A$ for $g\in G$ (with the usual action), and Lemma \ref{without cover} assures that $A$ comes from a (restricted) monomial cover. The reader may check that this monomial cover is isomorphic to the extension $\mathcal E_\omega$.  
    \end{rem}
     
     \begin{rem} 
     	Not every central extension of $G$ by $Z$ can yield a monomial cover. A necessary and sufficient condition, given by \eqref{59}, is that the cohomology class $\iota_*[\omega]\in H^2(G,\mu_n)$ will restrict to $0$ in $H^2(H_X,\mu_n)$ and $H^2(H_Y,\mu_n)$. In the case of Cocyclic matrices discussed below (\S\ref{subsec:coc}), this condition becomes trivial because $H_X=H_Y=\{1\}$. When the extension fails to satisfy the condition, Method I will fail by that $Z$ will not act trivially on matrices.
     \end{rem} 
 
     \begin{rem}
     	In the exact sequence \eqref{59} the case of $H^1$-development is precisely the case where $[\omega]=0$.
     \end{rem}

    \subsection{Cocyclic matrices and cocyclic development}\label{subsec:coc} Cocyclic matrices are the among the simplest examples of $H^2$-developed CDMs. The theory of cocyclic matrices has its origin in the theory of multidimensional combinatorial designs and has been proved later as a useful construction for Hadamard and weighing matrices, see \cite{HoDel93} and \cite{deLauney:2000}. Let $G$ be a finite group and let $\mathcal O\subset G\times G$ be a $G$-stable irreducible subset. If $\omega \in Z^2(G,\mu_n)$ is a 2-cocycle, then a \emph{pure} cocyclic matrix is the matrix $C=C(\omega)$ given by $$C_{x,y}=\omega(x^{-1},y),$$ whenever $(x,y)\in \mathcal O$, and $C_{x,y}=0$ otherwise. A general cocyclic matrix is a matrix of the form $D=K\circ C$, where $C$ is pure cocyclic and $K$ is $G$-developed.\\

    \begin{rem}
    	Notice that our definition is different from the definition appearing in the literature, e.g. as in \cite[Definition 3.1]{deLauney:2000} where it was defined as $C_{x,y}=\omega(x,y)$. The two definitions are Hadamard equivalent, by the transformation $x\mapsto x^{-1}$ in the $X$-axis. The reason for our choice will be apparent below in Theorem \ref{591}.
    \end{rem} 

    In this setting we let $X=Y=G$, where $G$ acts by multiplication from the left.
	Let $A=K\circ C(\omega)$ be a cocyclic matrix. Then using the cocycle condition \eqref{coc3}, for any $g\in G$,
    \begin{multline}\label{autc}
        (gA)_{x,y}=A_{g^{-1}x,g^{-1}y}=K_{g^{-1}x,g^{-1}y}\omega(x^{-1}g,g^{-1}y)\\
        =K_{x,y}\omega(x^{-1},y)\omega(g,g^{-1}y)\omega(x^{-1},g)^{-1}
        =\omega(g,g^{-1}y)A_{x,y}\omega(x^{-1},g)^*.
    \end{multline} 
    and this can be rewritten as $A=P(g)AQ(g)^*$, for the monomial\\ $Q(g)=\diag(\omega(g,g^{-1}y)^*)p_Y(g)$ and $P(g)=\diag(\omega(x^{-1},g)^*)p_X(g)$. In particular, $A$ is a CDM with respect to the above action.\\
    
    On the other hand, we claim that the map $g\mapsto (P(g),Q(g))$ induces a homomorphism $G\to Mon(X,\mu_n)\times Mon(Y,\mu_n)/Triv$. Indeed, by the irreducibility of $\mathcal O$, the equation $P_1AQ_1^*=P_2AQ_2^*$ for monomial $P_i,Q_i$ implies that $(P_1,Q_1)=(P_2,Q_2)\mod Triv$. In turn, the fact that $A=P(g)AQ(g)^*$ for all $g\in G$ implies that $g\mapsto (P(g),Q(g))\mod Triv$ is a homomorphism.\\
    
    This establishes a class $c(A)\in H^1(G,(\mu_n[X]\oplus \mu_n[Y])/\mu_n)$, given by the 1-cocycle 
    \be\label{zc} z(g)= (\sum_x \omega(x^{-1},g)[x], \sum_y \omega(g,g^{-1}y)[y]) \mod \mu_n.\ee We now claim
    \begin{prop}[cf. Proposition \ref{prop:iota}]\label{pzc}
    	The boundary map in \eqref{bound} sends the class $[z]$ to the class $[\omega]\in H^2(G,\mu_n)$.
    \end{prop}

	\begin{proof}
		We have to compute $g_1\bar z(g_2)-\bar z(g_1g_2)+\bar z(g_1)$, for a lift $\bar z$ to $\mu_n[X]\oplus \mu_n[Y]$. We take for $\bar z$ the formula used in \eqref{zc}. In the $X$-coordinate we get
		$$\sum_x \omega(x^{-1},g_2)[g_1x]-\sum_x \omega(x^{-1},g_1g_2)[x]+\sum_x \omega(x^{-1},g_1)[x]$$
		$$= \sum_x \omega(x^{-1}g_1,g_2)[x]- \omega(x^{-1},g_1g_2)[x]+\omega(x^{-1},g_1)[x]=
		\omega(g_1,g_2)\sum_x[x],$$
	where in the last equality we have used the cocycle relation \eqref{coc3}. In the $Y$-coordinate, we first rewrite $z_Y(g)=\sum_y \omega(g,y)[gy]$, and then a similar computation results in $\omega(g_1,g_2)\sum_y [g_1g_2y]=\omega(g_1,g_2)\sum_y [y]$. Notice that the pair\\ $(\omega(g_1,g_2)\sum_x[x],\omega(g_1,g_2)\sum_y [y])$ is the image of $\omega(g_1,g_2)\in \mu_n$ in $\mu_n[X]\oplus \mu_n[Y]$, which proves the result.
	\end{proof}

	In our case we have $X=Y=G$, and the stabilizers are $H_X=H_Y=\{1\}$. Therefore the extreme cohomology groups in \eqref{59} vanish, and we obtain an isomorphism
	
	\be\label{coc} H^1(G,(\mu_n[X]\oplus \mu_n[Y])/\mu_n)\cong H^2(G,\mu_n).\ee We can conclude the following theorem:
	
	\begin{thm}\label{thm:710}
		Suppose that $X=Y=G$, and that $\mathcal O\subset G\times G$ is an irreducible $G$-stable subset. Then
		\begin{enumerate}
			\item $\mathbf h^0(G,\mathcal O)=\mathbf h^1(G,\mathcal O)$,
			\item The group $\mh^2(G,\mathcal O)$ is the $D$-equivalence closure of the set of all cocyclic $G$-matrices with values in $\mu_n^+$ and with support $\mathcal O$.
			\item We have an isomorphism of groups $$\mathbf h^2(G,\mathcal O)/\mh^0(G,\mathcal O)\cong H^2(G,\mu_n).$$
		\end{enumerate}
	\end{thm}

    \begin{proof}
    	(1) By Theorem \ref{nuniq} and the Eckmann-Shapiro Lemma, the quotient\\ $\mh^1(G,\mathcal O)/\mh^0(G,\mathcal O)$ is isomorphic to a quotient of $Hom(H_X,\mu_n)\oplus Hom(H_Y,\mu_n)=0$, so (1) follows. For the proof of (2), we have already seen that all cocyclic matrices are cohomology-developed, hence they are in $\mh^2(G,\mathcal O)$. Conversely, suppose that $A\in \mh^2(G,\mathcal O)$. Then $G$ admits a homomorphism $\bar\theta:G \to Mon(X,\mu_n)\times Mon(Y,\mu_n)/Triv$, and a corresponding $1$-cocycle $z\in Z^1(G,(\mu_n[X]\oplus \mu_n[Y])/\mu_n)$, such that $\bar{\theta}(g)=z(g)\cdot(p_X(g),p_Y(g))$. Under the boundary map $\partial$ (which is defined at the level of cocycles) $z$ maps to a $2$-cocycle $\omega\in Z^2(G,\mu_n)$. Let $C=C(\omega)$ be the cocyclic matrix with respect to $\omega$. Let $z_C$ be the $1$-cocycle for $C$ as in equation \eqref{zc}. Then by Proposition \ref{pzc} we have that $E=A\circ C^{\circ -1}$ ($C^{\circ m}$ is the $m$th Hadamard power) has the corresponding $1$-cocycle $z_E=z-z_C$, and $z_E$ maps to $0$ in $Z^2(G,\mu_n)$. Now, by the isomorphism \eqref{coc}, $z_E$ is cohomologous to $0$, which proves that $E$ is diagonally equivalent to a $G$-invariant matrix $K$. It follows that $A$ is diagonally equivalent to the cocyclic matrix $C$, which proves (2).\\
    	
    	For (3), any element $A$ of $\mh^2(G,\mathcal O)$ is invariant under a $G$-action coming from a monomial cover. This monomial cover defines a map $\bar{\theta}:G\to Mon(X,\mu_n)\times Mon(Y,\mu_n)/Triv$, and in turn a $1$-cocycle $z\in Z^1(G,(\mu_n[X]\oplus \mu_n[Y])/\mu_n)$, as was explained in part (2) of the proof. By the irreducibility of $\mathcal O$, $z$ is uniquely determined (cf. the proof of Theorem \ref{nuniq}). Thus $A\mapsto z$ gives a well-defined map $\mh^2(G,\mathcal O)\to H^1(G,(\mu_n[X]\oplus \mu_n[Y])/\mu_n)$, and an application of the boundary map $\partial$ gives us a map $\kappa:\mh^2(G,\mathcal O)\to H^2(G,\mu_n)$. This map is surjective because every 2-cocycle gives rise to a cocyclic matrix. Finally, the kernel of $\kappa$ is the subgroup of all matrices having $z$ cohomologous to $0$. Such matrices are $D$-equivalent to matrices with $z=0$, which are $G$-invariant. Hence the kernel is $\mh^0(G,\mathcal O)$. This proves (3).
    \end{proof}

	The last observation on cocyclic matrices is that if we modify the definition of cocyclic matrices by a certain diagonal equivalence, then the complex space spanned by all modified cocyclic matrices with respect to a given 2-cocycle $\omega$ is closed under matrix multiplication and defines a matrix algebra of dimension $|G|$. 
	\begin{defn}
		A \emph{modified} cocyclic matrix with respect to $G$ and a 2-cocycle $\omega$ is a matrix of the form $K\circ C'(\omega)$, where $K$ is $G$-invariant, and $C'(\omega)_{x,y}=\omega(x^{-1},y)/\omega(y^{-1},y).$
	\end{defn} 
    So a modified cocyclic matrix is just a cocyclic matrix multiplied on the right by the diagonal matrix $\diag(\omega(y^{-1},y)^*)$. We have
    
    \begin{lem}\label{lem:pp}
    	Let $A$ be a modified cocyclic matrix. Then for any $g\in G$, $gA=L(g)AL(g)^*$ for a \emph{diagonal} $L(g)$ depending only on $\omega$.
    \end{lem}

    \begin{proof}
    	This is a similar computation to \eqref{autc}, and we get that\\ $L(g)=\diag(\omega(x^{-1},g))$. 
    \end{proof}
    In particular $A$ is invariant under the $(P,P)$-action defined by $P(g)=L(g)p_X(g)$ (To be more precise, $P$ can be lifted to a true homomorphism on $\widetilde{G}$). The details are left to the reader.
    As a corollary we get
    
    \begin{thm}\label{591}
    	The complex space spanned by all modified cocyclic matrices with respect to a specific $\omega$ is a $\mathbb C^*$-algebra of dimension $|G|$. 
    \end{thm}

    \begin{proof}
    	This is just the algebra $\mathcal A_P(X)$ for the $(P,P)$-action, where $P(g)=L(g)p_X(g)=L(g)p_Y(G)$ is coming from Lemma \ref{lem:pp}.
    \end{proof}

	\begin{rem}
		A well known special case of Theorem \ref{591} is the algebra of negacyclic matrices.
	\end{rem}

    \begin{rem}\begin{enumerate}
    		\item For the projective space matrices discussed in \S \ref{sec:proj}, we could have started with the group $PGL_{d+1}(F):=GL_{d+1}(F)/F^\times$ which acts faithfully on $X$ and $Y$. It turns out that the projective-space weighing matrix $A$ is cocyclic  w.r.t. the subgroup $T=\mathbb F_{q^{d+1}}^\times/\mathbb F_q^\times \hookrightarrow PGL_{d+1}(\mathbb F_q)$, of order $q^d+q^{d-1}+\cdots+1$. This group acts freely on $X$ and $Y$, and the matrix $|A|$ is $T$-developed. We may identify $X=Y=T$. The matrix $A$ is therefore a CDM over $T$, and thus is $T$-cocyclic. The group $T$ is sometimes called a \emph{Singer cycle} in the literature. The Cohomology class in $H^2(T,\mu_n)$ corresponding to $A$ by Theorem \ref{thm:710} is the restriction to $T$ of the cohomology class in $H^2(PGL_{d+1}(F),\mu_n)$, coming from the extension $GL_{d+1}(F)\to PGL_{d+1}(F)$. In some cases (like with the Paley Conference matrix, $d=1$ and $n=2$ ) it is nontrivial cocyclic.
    		
    		\item For the Grassmannian and Flag matrices we do not see any subgroup analogous to $T$ and we believe that they are not cocyclic.
    		
    		\item A general modified $G$-cocyclic matrix $A$ is $H^1$-developed w.r.t. the extension group $\widetilde{G}$. In this extension, $1\to Z \to \widetilde{G}\to G \to 1$, the stabilizers are $H_X=H_Y=Z$, and the homomorphisms $\psi_X=\psi_Y:Z\to \mu_n$ are the natural embedding $Z\hookrightarrow Triv=\mu_n$. We could have constructed modified cocyclic matrices by the $H^1$-method, starting from this raw data. The action of the extension group $\widetilde{G}$ on cocyclic matrix axes (which are copies of $G$) is centrally regular.
    		
    	\end{enumerate}
    	
    \end{rem}

	\section*{Acknowledgements} The authors would like to thank  Yossi Strassler for introducing them to the topics of  Hadamard, weighing and cocyclic matrices. Some ideas in this paper were inspired by him. The authors also would like to thank a set of anonymus referees for their careful reading of this paper. Some of the refereees gave us very detailed remarks and suggestions that we have found useful, and thanks to them this paper is (hopefully) much more accessible.
	
	\section*{Data Availability Statement}		
	The authors confirm that the data supporting the findings of this study are available within the article.

\bibliography{bibl}{}
\bibliographystyle{plain}

\appendix

\section{Sections and the 1st cohomology} \label{appA}
	
	Let us mention briefly the basic definitions of the $1$st group cohomology. References on group cohomology are e.g. \cite{1982cohomology} or \cite{Neukirch2008}. Let $G$ be a group, and let $M$ be a $G$-module. This means that $M$ is an abelian group with a $G$-action preserving the group structure of $M$.
	
	\begin{defn}
		\begin{itemize}
			\item[]
			\item[(a)] A \emph{$1$-cocycle} or a \emph{crossed homomorphism} of $G$ w.r.t. $M$ is a function $z:G\to M$ satisfying
			\be \nonumber z(g_1g_2)=z(g_1)+g_1z(g_2), \ \forall g_1,g_2\in G.\ee
			The set $Z^1(G,M)$ of $1$-cocycles is closed under addition and is an abelian group.
			\item[(b)] A \emph{$1$-coboundary} of $G$ w.r.t. $M$ is a function $\gamma:G\to M$ satisfying \be \nonumber \gamma(g)=gm-m, \forall g\in G\ee for some fixed $m\in M$. The set of $1$-coboundaries, $B^1(G,M)$ is a subgroup of $Z^1(G,M)$.
			\item[(c)] The \emph{$1$st cohomology group} is
			\be \nonumber H^1(G,M)\ := \ \frac{Z^1(G,M)}{B^1(G,M)}.\ee
			In particular, if $G$ acts trivially on $M$, then $H^1(G,M)=Hom(G,M).$
			\end{itemize}
	\end{defn}

	Now, suppose that 
	$$1\longrightarrow M \longrightarrow J \longrightarrow G_0 \longrightarrow 1$$ is a split exact sequence of groups with $M$ abelian, and with a designated section $t:G_0\to J$. Moreover let $\beta:G\to G_0$ be a homomorphism.
	Then $g\in G$ acts naturally and in a well defined manner on $m\in M$ via conjugation: $g,m\mapsto \widetilde{\beta(g)} m \widetilde{\beta(g)}^{-1}$, where $\widetilde{\beta(g)}$ is a lifting of $\beta(g)$ to $J$. We shall write the group law of $M$, both in multiplicative and additive notation, depending on whether we see $M$ as a subgroup of $J$, or as an $G$-module. We have a map $G\to J$ given by $s_0:=t\circ\beta$. The situation here is similar to diagram \eqref{diag2}.
	
	\be \label{diag5}
	\begin{tikzcd}
	1 \arrow[r] &M \arrow[r]& J \arrow[swap]{r}{\pi} & G_0 \arrow[bend right=40,swap]{l}{t}
	\arrow[r] & 1\\
	& & & G \arrow[swap]{u}{\beta} \arrow[ul,"s",dashed,shorten >= 2pt,shorten <= 2pt]
	\arrow[bend right=-40,dashed,shorten >= 2pt,shorten <= 2pt]{ul}{s_0}
	\end{tikzcd}.
	\ee
	
	 If $s:G\to J$ is another map with $\pi\circ s=\beta$, then we can write $s(g)=z(g)s_1(g)$ for some function $z:G\to M$. Hence, 
	\begin{align*} s(g_1g_2)&=z(g_1g_2)s_0(g_1g_2) \text{ and }\\
	s(g_1g_2)&=s(g_1)s(g_2)=z(g_1)s_0(g_1)z(g_2)s_0(g_2)=z(g_1)g_1(z(g_2))s_0(g_1)s_0(g_2),
	\end{align*}
	where by $g_1(z(g_2))$ we mean the action of $g_1\in G$ on $z(g_2)\in M$. The fact that $s$ is a homomorphism is equivalent then to the condition 
	$$z(g_1g_2)=z(g_1)+g_1(z(g_2)),$$ 
	defining a cocycle (note that we pass to writing in the additive notation of $M$). Now, let $z'$ be another $1$-cocycle, differing from $z$ by a $1$-coboundary. Then $z'(g)=z(g)+gm-m$ for some fixed $m$,
	and $z'$ gives rise to another map $s':G\to  J$, given by $s'(g)=z'(g)s_0(g)$. We will show that $s'$ and $s$ are conjugates of each other by an element of $M$:  
	\begin{gather*} s'(g):=z'(g)s_0(g)=(m^{-1})(gm)z(g)s_0(g) \text{ (multiplicatively written)}\\
	= (m^{-1})(gm)s(g)= m^{-1}s(g)ms(g)^{-1}s(g) \text{ (because $gm=s(g)ms(g)^{-1}$)}\\
	=m^{-1}s(g)m.\end{gather*}
	We conclude:
	\begin{thm}\label{class_sect}
		$H^1(G,M)$ is in bijection with the set of all homomorphisms $s:G\to J$, satisfying $\pi\circ s=\beta$, where two homomorphisms $s,s'$ are being identified if and only if they are conjugates over $M$.
	\end{thm}

	\subsection{$H^2$ of a group}
    We shall review briefly the basic facts about the second cohomology of a group. For a reference, see \cite{1982cohomology}.\\
    
    Let $G$ be any group and $M$ an $G$-module. 
    \begin{defn}
    	\begin{itemize}
    		\item[(a)] A 2-cocycle of $G$ with values in $M$  is a function $\omega:G\times G\to M$ satisfying
    		\be \label{coc3} g_1\omega(g_2,g_3)-\omega(g_1g_2,g_3)+\omega(g_1,g_2g_3)-\omega(g_1,g_2)=0,\ee
    		for all $g_1,g_2,g_3\in G$.
    		The set $Z^2(G,M)$ of all $2$-cocycles is an abelian group under addition.
    		\item [(b)] A 2-coboundary of $F$ with values in $M$ is a function $\omega:G\times G\to M$ such that \be \nonumber \omega(g_1,g_2)=g_1z(g_2)-z(g_1g_2)+z(g_1),\ee for some function $z:G\to M$. The set $B^2(G,M)$ of all  coboundaries is an abelian subgroup of $Z^2(F,M)$.
    		\item[(c)] The 2nd cohomology group of $G$ with coefficients in $M$ is 
    		\be \nonumber H^2(G,M)\ :=\ \frac{Z^2(G,M)}{B^2(G,M)}.\ee 
    	\end{itemize}
    \end{defn}

    For any exact sequence of $G$-modules
    $$0\to M'\to M\to M'' \to 0$$ 
    there is a long exact sequence in cohomology, a portion of which is 
    
    \be\label{long}  H^1(G,M')\to H^1(G,M)\to H^1(G,M'')\stackrel{\partial}{\longrightarrow} H^2(G,M')\to H^2(G,M).
    \ee
    All maps of the sequence, except $\partial$, come from the maps between the modules.
    The map $\partial$, called the \emph{connecting homomorphism}, can be made explicit as follows. A cohomology class $c\in H^1(G,M'')$ can be represented by a $1$-cocycle $\bar z:G\to M''$. This is lifted to a {\bf function} $z:G\to M$, generally not a 1-cocycle. Then  $\omega(g_1,g_2):=g_1z(g_2)-z(g_1g_2)+z(g_1)$ is generally not zero, and by definition $\omega(-,-)$ is a 2-coboundary with values in $M$. But when projecting to $M''$ $\omega(g_1,g_2)$ is zero, as $\bar z$ was a cocycle. By exactness this implies that $\omega(g_1,g_2)\in M'$. As an $M'$-valued function $\omega$ is now a cocycle, but  generally no longer a coboundary. Its class in $H^2(F,M')$ is the target $\partial c$.\\

 	 \section{Induction - the Eckmann-Shapiro Lemma}\label{appB} Let $G$ be a finite group and $H\subseteq G$ a subgroup. There is a natural functor taking a $G$-module $M$ and returning an $H$-module, which is simply the restriction to $H$. There is functor in the opposite direction, taking an $H$-module $N$ and returning a $G$-module $M$. This is the induction of modules. It is defined by 
	 $$ Ind_H^G N:=\ZZ[G]\tensor_{\ZZ[H]} N.$$ Here $\ZZ[G]$ is the group-ring of $G$ whose elements are given by finite formal sums $\sum n_i[g_i]$ for $n_i\in \ZZ$ and $g_i\in G$, and the multiplication is the extension by linearity from the multiplication in $G$. The induced module is given a structure of a $G$-module by the rule
	 $$ g\cdot ([g']\tensor n)\ = \ [gg']\tensor n.$$
	 There is yet another functor in this direction, the \emph{co-induction}, denoted by $CInd_H^G(-)$, and defined by 
	 $$ CInd_H^G(N) \ :=\ Hom_{\ZZ[H]}(\ZZ[G],N).$$ 
	 equivalently $$CInd_H^G N \ = \ \left\{\phi:\ZZ[G]\to N \ \big| \ \phi([hg])=h\phi([g]) \ \forall h\in H \text{ and }  g\in G\right\}.$$
	 The co-induced module also has a structure of a $G$-module, defined by
	 $$ (g\cdot \phi)(g') \ =\ \phi(g'g^{-1}).$$
	 
	 Using coset decompositions $G=\bigsqcup_i Hg_i=\bigsqcup_i g_i^{-1}H$, it is easy to see that there are isomorphisms $$Ind_H^GN\cong \ZZ[G/H]\tensor_\ZZ N$$ and $$CInd_H^GN\cong Hom_\ZZ(\ZZ[H\backslash G],N).$$

	 If $[G:H]<\infty$ this can be used to set up an isomorphism $Ind_H^GN\cong CInd_H^GN$ of $G$-modules by identifying $H\backslash G\cong G/H$ via the inverse map, and using the notion of dual bases to compare tensor with hom. But this isomorphism is not canonical and depends on the choice of coset representatives. Also at inifnite index the induced and co-induced modules aren't generally isomorphic.\\

	 In the following theorem we will use double argument function, such as $z(a)(b)$. This means that $z(a)$ is a function of $b$, and $z(a)$ depends on $a$. Such expressions occur when we have a function $z:A\to CInd_H^G N$, so $z(a)$ is itself a function of $b\in \ZZ[G]$. 
 	 
 	 \begin{thm}[Eckmann-Shapiro Lemma, see \cite{1982cohomology}, p.72 or \cite{Neukirch2008} p. 62]\label{shap}
 	 	\begin{itemize}
			\item[]
 	 		\item[(a)] For all $i\ge 0$ there is an isomorphism \be \nonumber H^i(H,M)\cong H^i(G,CInd_H^G M). \ee
 	 		\item[(b)] The map $H^1(G,Ind_H^G M)\to H^1(H,M)$ can be defined at the level of cocycles as follows:
 	 		A $1$-cocycle $z:G\to CInd_H^G M$ is mapped to the $1$-cocycle $$y(h) = z(h)(1_G),$$ defined for $h\in H$.
 	 		\item[(c)] The inverse map can be defined at the level of cocycles as follows: First pick up a set of coset representatives $\{g_i\}$ so that $G=\bigcup_i Hg_i$. For any $g\in G$, let $\bar g$ be the unique coset representative in the coset $Hg$.   
 	 		Then a $1$-cocycle $y:H\to M$ is mapped to the $1$-cocycle $z:G\to Ind_H^G M$ defined by
 	 		$$z(g')(g) = (g\bar{g}^{-1})y(\bar gg'(\overline{gg'})^{-1}),$$ defined for $g,g'\in G$.
 	 	\end{itemize}
  	 \end{thm}

    \begin{proof}(Sketch) We discuss the general case of higher cohomology. The proof is intended for readers with sufficient background and is given for the sake of completeness.\\
    	
    If $P$ is an $G$-module, the cohomology groups $H^i(G,P)$ ($i>0$) are defined as the homology groups of the cochain complex $\mathcal X(G,P)$, defined as
    $$\mathcal X^i(G,P)=Hom(\ZZ[G^{i+1}],P)^{G}.$$
    Here $\ZZ[G^i]$ is the group ring, and $G$ acts from the left. The superscript $G$ means taking the $G$-equivariant homomorphisms. The differentials are the usual simplicial differentials obtained from alternating sums of simplicial face maps.\\
    
    The cohomology $H^i(H,M)$ is obtained similarly from the cochain complex $\mathcal X(H,M)$, but it can also be obtained from the complex $\mathcal Y(G,M)$ given by $$\mathcal Y^i(G,M)=Hom(\ZZ[G^{i+1}],M)^H.$$
    
    It is not hard to show that there is an isomorphism of cochain complexes
    \be\label{eq:cochain} \mathcal Y^i(G,M)\cong \mathcal X^i(G,Ind_H^G M).\ee  
	This isomorphism proves part (a) of the theorem.\\
	
	The isomorphism \eqref{eq:cochain} can be given explicitly as follows: An element $z\in Hom(\ZZ[G^{i+1}],Ind_H^G M)^G$, is mapped to $x\in Hom(\ZZ[G^{i+1}],M)^H$, given by $x(g_0,\ldots,g_i)=z(g_0,\ldots,g_i)(1_G)$. In the other direction, for $x$ we match $z$ such that \be\label{ind} z(g_0,\ldots,g_i)(g):=x(gg_0,\ldots,gg_i),\ee (see \cite{Neukirch2008} p.62).\\
    
    We now connect the above two computations of the cohomology $H^i(H,M)$ coming from \eqref{eq:cochain}. The missing part is the morphism of free $H$-modules, $\ZZ[G^{i+1}]\to \ZZ[H^{i+1}]$ given by $[(g_0,\ldots,g_i)]\mapsto [(g_0\overline{g_0}^{-1},\ldots g_i\overline{g_i}^{-1})]$. This map is a morphism between free $H$-resolutions of $\ZZ$. In the other diresction the map $\ZZ[H^i]\to \ZZ[G^i]$ is given by the natural inclusion $H\subseteq G$. Both maps give rise to aa chain homotopy equivalence $\ZZ[H^\bullet]\sim \ZZ[G^\bullet]$. Now, for $y\in Hom(\ZZ[H^{i+1}],M)^{H}$, we compute the corresponding $z\in \mathcal X^i(G,Ind_H^G M)$ by 
    $$z(g_0,\ldots,g_i)(g)=x(gg_0,\ldots,gg_i)=y(gg_0\overline{gg_0}^{-1},\ldots,gg_i\overline{gg_i}^{-1}).$$
    This is the homogeneous form of the formula in (c). For the non-homogeneous form, we take $g_0=1$ and then use the $H$-equivariance of $y$ to obtain (c). The proof of (b) is by taking the other direction of the chain homotopy equivalence.
    \end{proof}

\end{document}